\documentclass[letter, 11pt]{article} 
\usepackage{amsmath}
\usepackage{mathrsfs} 
\usepackage{amsthm}
\usepackage{amsfonts}
\usepackage{amssymb}
\usepackage{latexsym}
\usepackage{tikz} 
\usepackage{esint}
\usepackage{mathtools}
\usepackage{stmaryrd}

\usepackage{mathabx}

\usepackage{tikz-cd}
\usepackage{cancel}
\usepackage{sectsty}
\usepackage{hyperref}
\usepackage{caption}
\usepackage{float}
\usepackage{setspace}

\usepackage{comment}
\usepackage{quiver}

\usepackage{xcolor}

\definecolor{mediumtealblue}{rgb}{0.0, 0.33, 0.71}
\definecolor{red(munsell)}{rgb}{0.96, 0.12, 0.24}
\definecolor{forestgreen}{rgb}{0.13, 0.55, 0.13}
\definecolor{ginger}{rgb}{0.69, 0.4, 0.0}

\newcommand\myshade{85}
\colorlet{mylinkcolor}{red(munsell)}
\colorlet{mycitecolor}{mediumtealblue}
\colorlet{myurlcolor}{forestgreen}

\usepackage{lmodern}

\usepackage[utf8]{inputenc}

\usepackage[inline,shortlabels]{enumitem}

\usepackage{amssymb,mathrsfs,stmaryrd}
\usepackage{mathtools}
\usepackage{thmtools,thm-restate}
\usepackage{colonequals}
\usepackage{tikz}
\usetikzlibrary{cd}

\usepackage[capitalize,noabbrev]{cleveref}

\hypersetup{
	linkcolor  = mylinkcolor!\myshade!black,
	citecolor  = mycitecolor!\myshade!black,
	urlcolor   = myurlcolor!\myshade!black,
	colorlinks = true
}

\mathchardef\mhyphen="2D

\usepackage{titlesec}
\titleformat{\subsection}[runin]{\normalfont\bfseries}{\thesubsection}{1em}{}
\titleformat{\subsubsection}[runin]{\normalfont\bfseries}{\thesubsubsection}{1em}{}

\renewcommand{\thesubsubsection}{(\thesubsection.\arabic{subsubsection})}

\usepackage{tocloft} 

\usepackage[OT2,T1]{fontenc}
\DeclareSymbolFont{cyrletters}{OT2}{wncyr}{m}{n}
\DeclareMathSymbol{\Sha}{\mathalpha}{cyrletters}{"58}

\sectionfont{\fontsize{12}{15}\selectfont}
\subsectionfont{\fontsize{11}{15}\selectfont}
\subsubsectionfont{\fontsize{10}{15}\selectfont}

\usepackage[colorinlistoftodos,textsize=scriptsize]{todonotes}
\usepackage{marginnote}

\newcommand{\mytodo}[2][]{{%
 \let\marginpar\marginnote
 \reversemarginpar
 \renewcommand{\baselinestretch}{0.8}%
 \todo[#1]{#2}}}

\usepackage[matrix,tips,graph,curve]{xy}
\usepackage{enumitem}

\makeatletter

\usepackage[
backend=biber,
style=alphabetic,
sorting=anyt, doi=false, url=false, isbn=false, giveninits=true, maxnames = 6,
]{biblatex}

\makeatother


\makeatletter

\theoremstyle{plain}
\newtheorem{theorem}{Theorem}[section]

\newtheorem{proposition}[theorem]{Proposition}
\newtheorem{conjecture}[theorem]{Conjecture}
\newtheorem{lemma}[theorem]{Lemma}

\newtheorem{corollary}[theorem]{Corollary}

\newtheorem{question}[theorem]{Question}

\theoremstyle{definition}
\newtheorem{definition}[theorem]{Definition}
\newtheorem{notn}[theorem]{Notation}
\newtheorem{example}[theorem]{Example}
\newtheorem{construction}[theorem]{Construction}
\newtheorem{setup}[theorem]{Setup}

\newtheorem{convention}[theorem]{Convention}

\newtheorem{notation}[theorem]{Notation}

\newtheorem{remark}[theorem]{Remark}

\newcommand{\IA}{\mathbb{A}}

\newcommand{\IC}{\mathbb{C}}
\newcommand{\ID}{\mathbb{D}}

\newcommand{\IF}{\mathbb{F}}
\newcommand{\IG}{\mathbb{G}}

\newcommand{\IL}{\mathbb{L}}

\newcommand{\IN}{\mathbb{N}}

\newcommand{\IP}{\mathbb{P}}
\newcommand{\IQ}{\mathbb{Q}}
\newcommand{\IR}{\mathbb{R}}

\newcommand{\IV}{\mathbb{V}}

\newcommand{\IZ}{\mathbb{Z}}

\newcommand{\sC}{\mathcal{C}}

\usepackage{calligra}
\DeclareMathOperator{\sheafhom}{\mathscr{H}\text{\kern -3pt {\calligra\large om}}\,}

\newcommand{\End}{\mathrm{End}}

\newcommand{\Hom}{\mathrm{Hom}}
\newcommand{\Aut}{\mathrm{Aut}}

\newcommand{\Spec}{\mathrm{Spec}}

\newcommand{\Lie}{\mathrm{Lie}}

\newcommand{\NL}{\mathrm{NL}}

\newcommand{\che}[1]{\widecheck{#1}}

\newcommand{\rank}{\rm rank} 

\newcommand\tensor{{\otimes}}
 
\newcommand\union{\bigcup} 

\newcommand{\<}{\langle}
\renewcommand{\>}{\rangle}

\newcommand{\into}{\hookrightarrow}

\def\d/{/\mspace{-6.0mu}/}

\def\wt{\widetilde}

\def\what{\widehat}

\newcommand{\p}{\partial}

\newcommand{\fp}{\mathfrak{p}}

\newcommand{\sV}{\mathcal{V}}
\newcommand{\sI}{\mathcal{I}}
\newcommand{\sM}{\mathcal{M}}

\newcommand{\sJ}{\mathcal{J}}
\newcommand{\sF}{\mathcal{F}}
\newcommand{\sH}{\mathcal{H}}
\newcommand{\sX}{\mathcal{X}}
\newcommand{\sS}{\mathcal{S}}

\newcommand{\sE}{\mathcal{E}}
\newcommand{\sO}{\mathcal{O}}
\newcommand{\sL}{\mathcal{L}}
\newcommand{\sZ}{\mathcal{Z}}
\newcommand{\sP}{\mathcal{P}}

\newcommand{\Fil}{\mathrm{Fil}}
\renewcommand{\H}{\mathrm{H}}
\newcommand{\Spf}{\mathrm{Spf}}

\newcommand{\Def}{\mathrm{Def}}

\newcommand{\dR}{\mathrm{dR}}
\newcommand{\et}{\textrm{{\'e}t}}

\newcommand{\bH}{\mathbf{H}}

\newcommand{\NS}{\mathrm{NS}}
\newcommand{\GL}{\mathrm{GL}}
\newcommand{\codim}{\mathrm{codim}}

\newcommand{\w}{\omega}

\newcommand{\sfM}{\mathsf{M}}
\newcommand{\sfD}{\mathsf{D}}

\newcommand{\sfA}{\mathsf{A}}

\newcommand{\sfH}{\mathsf{H}}

\newcommand{\bL}{\mathbf{L}}

\newcommand{\sfK}{\mathsf{K}}
\newcommand{\Gal}{\mathrm{Gal}}
\newcommand{\sA}{\mathcal{A}}
\newcommand{\sB}{\mathcal{B}}
\newcommand{\sG}{\mathcal{G}}

\newcommand{\Sh}{\mathrm{Sh}}

\newcommand\restr[2]{{
	\left.\kern-\nulldelimiterspace
	#1
	\vphantom{\big|}
	\right|_{#2}
	}}

\newcommand{\Sp}{\mathrm{Sp}}

\newcommand{\GSp}{\mathrm{GSp}}

\newcommand{\shS}{\mathscr{S}}

\newcommand{\ct}{\mathrm{ct}}
\newcommand{\nc}{\mathrm{nc}}

\newcommand{\Jac}{\mathrm{Jac}}
\newcommand{\Rad}{\mathrm{Rad}}

\newcommand{\sD}{\mathcal{D}}
\newcommand{\IGr}{\mathrm{IGr}}

\newcommand{\pos}{\mathrm{pos}}

\newcommand{\MT}{\mathrm{MT}}

\newcommand{\bG}{\mathbf{G}}
\newcommand{\bg}{\mathbf{g}}
\renewcommand{\ss}{\mathrm{ss}}

\newcommand{\ab}{\mathrm{ab}}

\newcommand{\sQ}{\mathcal{Q}}

\newcommand{\can}{\mathrm{can}} 

\newcommand{\crys}{\mathrm{crys}}

\newcommand{\sfX}{\mathsf{X}}

\newcommand{\PGL}{\mathrm{PGL}}

\newcommand{\sab}{\mathrm{sab}}

\title{Special Loci in the Moduli of Curves}

\author{David Urbanik and Ziquan Yang}

\vspace{-1ex}
  
\date{\normalsize{\today}}

\begin{document}

\maketitle

\begin{abstract}
    We introduce new techniques to study mixed-characteristic special loci in moduli spaces and give several applications, focusing in particular on the moduli space $\sfM_{g}$ of curves. These applications include Zariski density results for split Jacobians, non-density results for supersingular loci, and lifting results for algebraic cycles, all under appropriate numerical conditions. Our methods generalize to more-or-less arbitrary moduli spaces mapping to Hodge-type Shimura varieties (and even some settings beyond the Shimura case), and relate to questions raised by Oort and de Jong.
\end{abstract}

\tableofcontents

\section{Introduction}

\subsection{Main Applications}

Curves with extra symmetry are a central topic in algebraic and arithmetic geometry. For us, ``having extra symmetry'' means that the Jacobian of the curve has a larger than usual endomorphism ring, or, more generally, admits non-generic algebraic cycle classes in some tensor power of its cohomology. A locus in $\sfM_g$ of such special curves is called a special locus of $\sfM_g$. In this paper we mostly focus on the symmetry of ``splitting up to isogeny'' and prove various density and non-density results, though as we will explain, the techniques generalize to other kinds of symmetries and special loci as well.

To frame the discussion, we make the following definition: 

\begin{definition}
    Let $g \ge 1$ be a natural number and $\bg := \{ g_i \}_{i = 1}^m$ be a partition of $g$. We say that an abelian variety $A$ of dimension $g$ over a field $k$ is \textbf{(prime-to-$p$) of splitting type} $\bg$ if there exists a set of abelian varieties $A_j$ with $\dim A_j = g_j$ and a prime-to-$p$ isogeny $\prod_{j = 1}^m A_j \to A$. Moreover, we define 
    \begin{equation}
    \label{deltaeq}
    \delta(\bg) = \frac{g(g + 1)}{2} - \sum_{j = 1}^m \frac{g_j(g_j + 1)}{2}, 
    \end{equation}
    i.e., the codimension of the image of $\prod_{j = 1}^m \sfA_{g_j}$ in $\sfA_g$.
\end{definition}

Let $\sfM_{g}$ (resp. $\sfM_g^\ct$) denote the moduli space over $\IZ$ of smooth proper curves of genus $g$ (resp. of compact type). Write $\sC$ for the universal curve. 

\begin{theorem}
\label{thm: main density}
    Suppose that $\bg$ is a partition of $g \in \IN$ such that $\delta(\bg) < 3g - 3$. Then there exists $r_\bg \in \IN$ such that for every prime $p > r_\bg$, the locus 
    \begin{equation}
         \mathrm{SP}_{\bg, p} := \{ s \in \sfM_{g, \IF_p} : \Jac(\sC_s) \textit{ is prime-to-$p$ of splitting type } \bg \}
    \end{equation}
    is Zariski dense in $\sfM_{g, \IF_p}$. 
\end{theorem}

\noindent In fact, if $\bg = \{1, g - 1\}$ (i.e., we ask for those Jacobians containing an elliptic curve factor), then one may take $r_\bg = 0$.

One motivation for the above result comes from the ``AIM conjecture'' proposed by a group of experts at an AIM meeting, and documented as \cite[Conj.~5.2]{STSurvey} by Shankar and Tang. We refer the reader to a discussion in \textit{loc. cit.} on earlier evidence of the conjecture, including \cite{MST, DvH25, AHS}. 

\begin{conjecture}
\label{conj: AIM}
    \emph{(AIM)} Let $\shS$ denote the canonical integral model of a Shimura variety with reflex field $E$, $X, Y$ be generically ordinary irreducible subvarieties of $\shS_{k(v)}$, where $v$ is a prime of good reduction with residue characteristic $p$, and suppose that
    \begin{equation}
        \label{eqn: dim condition}
        \dim X + \dim Y \ge \dim \shS_{k(v)}. 
    \end{equation}
    Then the set $\{ (x, y) \in X \times Y : x \in \cup_T T(y) \}$ is Zariski dense in $X \times Y$, where $T$ runs through the prime-to-$p$ Hecke operators. 
\end{conjecture}

\noindent In \textit{loc.~cit.}, the condition in \eqref{eqn: dim condition} is stated with equality. The formulation above, with ``$\geq$'' in place of ``$=$'', is clearly equivalent. Our \cref{thm: main density} verifies the conjecture when $X$ is the Torelli locus in $\sfA_{g, \mathbb{F}_{p}}$ and $Y$ is $\sfA_{g_1, \mathbb{F}_{p}} \times \cdots \times \sfA_{g_m, \mathbb{F}_{p}}$, and the inequality (\ref{eqn: dim condition}) is strict (cf. \cref{rmk: dense in product}). We remark that the original \cite[Conj.~5.2]{STSurvey} also includes a number field case, and earlier evidence includes \cite{ST20, SSTT}.

Under suitable conditions, our method can also yield new instances of \cref{conj: AIM} when equality is achieved. Below we use $\mathrm{IJ}(-)$ to denote the intermediate Jacobian of the argument.



\begin{theorem}
    \label{cubicthreefoldthm}
    Let $\sfM$ be the moduli space of smooth cubic threefolds over $\Spec(\IZ)$, and let $\sX$ be the universal family over $\sfM$. Then there exists an $r$ such that for all $p > r$, the locus 
    \begin{equation}
    \label{cubicspliteq}
        \{ s \in \sfM_{\IF_p} : \mathrm{IJ}(\sX_s) \text{ is prime-to-$p$ of splitting type } (1, 1, 1, 1, 1) \}
    \end{equation}
    is Zariski dense in $\sfM_{\IF_p}$. 
\end{theorem}

\noindent We note that the dimension of $\sfM_{\mathbb{F}_{p}}$ is 10, the dimension of $\sfA_{5,\mathbb{F}_{p}}$ is $15$, and loci in $\sfA_{5,\mathbb{F}_{p}}$ which parameterize products of $5$ elliptic curves have dimension $5$. We also remark that the constants $r_{\mathbf{g}}$ in \cref{thm: main density} and the $r$ in \cref{cubicthreefoldthm} can in principle be effectively bounded, see \cref{rmk: effectivity}. 





The methods of this paper can be used to analyze $p$-isogenies as well; for instance, they can be used to study supersingular loci, a particularly interesting class of special loci determined by ``extra symmetry''. Recall that every supersingular abelian variety is isogenous to a product of supersingular elliptic curves, though the isogeny can never be prime-to-$p$ unless the abelian variety is superspecial. While the core techniques we introduce are largely motivated by our density result for split Jacobians, we can also adapt our methods to prove \textit{non-density} results for supersingular loci. In particular, we address the following question asked by Oort (\cite[Rmk.~6.3]{Oorthyperelliptic}, \cite[\S{A}.11]{LiOort}):


\begin{question}
\emph{(Oort)}
    \label{ques: Oort} Let $\sfA_{g, \IF_p}^{\ss}$ be the supersingular locus of $\sfA_g$ and let $\tau : \sfM_g^\ct \to \sfA_g$ be the Torelli morphism. What is the dimension of $\tau(\sfM_{g, \IF_p}^\ct) \cap \sfA_{g, \IF_p}^{\ss}$?
\end{question}

The interest in the question is that, when $g$ is large (in fact $g \ge 9$), we have $\dim \tau(\sfM^\ct_{g, \IF_p}) + \dim \sfA^{\ss}_{g, \IF_p} < \dim \sfA_{g, \IF_p}$, so $\tau(\sfM^\ct_{g, \IF_p})$ and $\sfA^{\ss}_{g, \IF_p}$ are not supposed to intersect when viewed as ``random'' subvarieties. On the other hand, Oort himself observed that by clutching $h$ supersingular curves of genus $3$, with $g \ge 3h$, there is always at least one component of $\tau(\sfM^\ct_{g, \IF_p}) \cap \sfA^{\ss}_{g, \IF_p}$ of dimension $2h$ (\cite[Rmk.~6.3]{Oorthyperelliptic}). 

We do not give a full answer to Oort's question, but we can prove that $\tau(\sfM_{g, \IF_p}^\ct) \cap \sfA_{g, \IF_p}^{\ss}$ is $\le 0$ dimensional for large $g$ away from a Zariski closed locus independent of $p$.
\begin{theorem}
\label{thm: Oort}
    When $g \ge 16$, the union over all primes 
    $$ \bigcup_{p \in \Spec(\IZ)} [\tau(\sfM^\ct_{g, \IF_p}) \cap \sfA^{\ss}_{g, \IF_p}]^{> 0} $$
    \textit{is not Zariski dense in $\tau(\sfM_g^\ct)$.} Moreover, the image in $\sfA_{g,\IC}$ of the $\IC$-fiber of the Zariski closure of this union lands inside finitely many strict Shimura subvarieties of $\sfA_{g,\IC}$.
\end{theorem}
\noindent In the above and what follows we use the notation $[-]^{> 0}$ to denote the positive-dimensional part of a scheme or stack. In contrast to the theorem, recall that the irreducible components of $\sfA^{\mathrm{ss}}_{g, \IF_p}$ all have dimension $\lfloor g^2/4 \rfloor$ (\cite[\S4.9]{LiOort}), and $\bigcup_{p \in \Spec(\IZ)} \sfA^{\ss}_{g, \IF_p}$ itself is Zariski dense in the $\IZ$-stack $\sfA_g$ (see \cref{rmk: density of ss in A_g}).



We expect the inequality $g \geq 16$ can be improved by carrying out a more refined analysis. Our methods should yield similar statements when $\tau : T \to \sfA_{g}$ is an arbitrary map of Deligne-Mumford stacks such that the image of $\tau_{\IC}$ has large monodromy (subject to an appropriate inequality relating $\dim T_{\IC}$ and $g$). 

In practice one observes that when $g$ is large, there are essentially only two types of constructions which yield positive-dimensional families of supersingular curves: 

\begin{itemize}
    \item \textbf{Type I:} Those that lie inside unexpectedly large intersections of $\tau(\sfM_g^\ct)$ with Shimura subvarieties of $\sfA_g$. One easily checks that Oort's example is of this type.
    \item \textbf{Type II:} Those that make use of some special relations between $g$ and the characteristic $p$; e.g., \cref{ex: Artin-Schreier} below.
\end{itemize}
Our \cref{thm: Oort} states that when $g \geq 16$, these are indeed the only two types of constructions giving rise to components of $[\tau(\sfM^\ct_{g, \IF_p}) \cap \sfA^{\ss}_{g, \IF_p}]^{> 0}$. In particular, for a fixed $g$, there are only finitely many primes for which a construction which is not of Type I (and hence of Type II) is possible. Interestingly, Karemaker and Pries have a construction ``in the opposite direction'' if one fixes $p$ and let $g$ goes to infinity instead. See \cref{ex: Karemaker-Pries}. 

In the case of supersingular loci in orthogonal Shimura varieties, we obtain a more precise description of when such loci can be dense in a subvariety. We refer to \S\ref{orthoshimsec} for a more precise statement. 

\begin{theorem}
\label{orthogonalssthm} \emph{(\cref{preciseorthogonalshimurasupsingthm})}
    Let $\shS(L)$ be an integral model over $\mathbb{Z}[1/N]$ of an orthogonal Shimura variety associated to a quadratic lattice $L$ of signature $(2, b)$. Let $T$ be a smooth scheme over $\IZ[1/N]$, quasi-finite over $\shS(L)$, and such that the image of $T_\IC$ has Zariski dense monodromy. Write $\shS^{\mathrm{ss}, \sigma_0}_{\IF_p}$ for the supersingular locus with Artin invariant $\sigma_0$. Then the union of positive-dimensional atypical components 
    $$ \bigcup_{p \in \Spec(\IZ[1/N])} [T_{\IF_p} \cap \shS^{\mathrm{ss}, \sigma_0}_{\IF_p}]^{\mathrm{atyp}, >0} $$
    is not Zariski dense in $T$, and the $\IC$-fiber of the Zariski closure of this union is contained in a finite union of Shimura subvarieties strictly contained in $\shS(L)_{\IC}$. 
\end{theorem}
\noindent Notice that \cref{orthogonalssthm} has the same flavour as \cref{thm: Oort}, in that one is again saying that there are only finitely many primes $p$ for which a construction not explained by a characteristic zero special intersection --- that is, a construction not of Type I --- is possible.

Likewise, we can partially answer a question posed by de Jong\footnote{\url{https://www.math.columbia.edu/~dejong/wordpress/?p=3859}}:

\begin{question}
    \emph{(de Jong)} Are there infinitely many primes $p$ such that there is a 1-parameter family of supersingular quintic surfaces of characteristic $p$?
\end{question}
\noindent Assume that $X$ is a smooth proper surface over an algebraically closed field $k$ of characteristic $p$. Let $H$ be the torsion-free quotient of $\H^2_\crys(X/W)$, equipped with a Frobenius operator $\varphi$. Let $T = H^{\varphi = p}$ be the Tate module. Then we say that $X$ is supersingular if $H$ is pure of slope $1$, or equivalently, if the natural map $T \tensor_{\IZ_p} W \to H$ has finite cokernel. In this case, the cokernel is always $p$-primary torsion, and we say $X$ is $p$-torsion if this cokernel is killed by $p$ itself. Restricted to this subclass, we obtain the following theorem addressing de Jong's question.

\begin{theorem}
\label{supsingquintic}
    Let $\sfM$ be the open substack of $|\sO_{\IP^3}(5)| / \mathrm{PGL}_4$ parametrizing smooth quintics. Then there exists a strict Zariski closed substack $Z \subset \sfM$, with $Z_{\IC}$ a finite union of special subloci, such that on $\sfM \smallsetminus Z$ there are only finitely many primes $p$ for which there exists a one-parameter family of $p$-torsion supersingular quintics. 
\end{theorem}
There is nothing special about degree $5$, see \cref{thm: supersingular surface general}.

\begin{remark} As de Jong observed, the only explicitly known example of a positive-dimensional family of supersingular quintics is the one defined by Shioda, when $p = 5$, that is, 
    \begin{equation}
    \label{eqn: Shioda}
        X_0^5 = F(X_1, X_2, X_3)
    \end{equation}
    for a general quintic homogeneous form $F$ (see also \cref{rmk: singular quintic}). We note that these surfaces are indeed $5$-torsion: (\ref{eqn: Shioda}) defines a singular quintic with 13 $A_4$ singularities (\cite[Thm.~1.2]{Bla85}). Then the Tate module $T$ of the resolution is given by 13 copies of the $A_4$ root lattice (after tensoring with $\IZ_p$), plus the Chern class of the polarization.
    
    The atypical locus $Z_{\IC}$ is needed to allow for the possibility of ``type I constructions'' --- that is, constructions making use of extra geometry present along special subvarieties --- as happened for curves. Without a more precise understanding of this atypical locus, we believe a statement like that appearing in \autoref{supsingquintic} is ``as good as it gets''. 
\end{remark}

\subsection{Special Intersections --- Geometry and Arithmetic}

The above results can all be understood as questions about special intersections. For concreteness, let $\sfM^\ct_{g}$ and $\sfA_{g}$ be the moduli spaces of curves and principally polarized abelian varieties over the base scheme $\Spec\, R$, respectively, and let $\varphi : \sfM^\ct_{g} \to \sfA_{g}$ be the Torelli map. The variety $\sfA_{g}$ admits many special loci, for instance given by split abelian varieties or supersingular abelian varieties, which we can think of as corresponding to a morphism $w : Y \to \sfA_{g}$ whose image is the desired locus. The corresponding special intersection with $\sfM^\ct_{g}$ is the fibre product
\begin{equation}
\label{fibprod}
\varphi^{*} Y := \sfM^\ct_{g} \times_{\varphi,\sfA_{g},w} Y .
\end{equation}
Then if the pairs $(Y, w)$ belong to some collection $\sC$ of special loci, we can ask if the union
\[ \bigcup_{(Y, w) \in \sC} \varphi^{*} Y \]
has Zariski dense image in $\sfM^\ct_{g}$. In other words, does there exist a proper closed sublocus $E \subsetneq \sfM^\ct_{g}$ such that for each $(Y, w) \in \sC$, the natural map $\varphi^{*} Y \to \sfM^\ct_{g}$ factors through $E$?

Whether one expects density to hold depends on both $R$ and $\sC$. To illustrate, we consider the following examples.
\begin{itemize}
\item[(i)] Suppose first that $R = \mathbb{C}$, and let $\sC$ be the collection of all Hecke translates of a fixed Shimura subvariety $Y \subset \sfA_{g}$. Then one expects density to hold if and only if one has the ``expected dimension'' inequality
\begin{equation}
\label{expdimineq}
\dim Y + \dim \sfM^\ct_{g} \geq \dim \sfA_{g} .
\end{equation}
The ``if'' direction is known, and is a consequence of \cite[Thm. 1.9]{zbMATH07967077}. The ``only if'' direction is open, and would follow from the Zilber-Pink conjecture for Shimura varieties formulated by Pink in \cite{richard2005common}. The special case of the ``only if'' direction corresponding to $\dim Y = 0$ is known, as it is a consequence of the Andr\'e-Oort conjecture for $\sfA_{g}$, resolved in \cite{zbMATH06841543}.
\item[(ii)] Suppose on the other hand that, in the situation of (i), one wishes to show the ``only if'' direction for just those components of $\varphi^{*} Y$ that are positive dimensional. That is, denoting by $[\varphi^{*} Y]^{> 0}$ the union of those components which have positive dimension, one wishes to show there exists a strict subvariety $E \subsetneq \sfM^\ct_{g}$ such that, when (\ref{expdimineq}) fails, the maps
\[ [\varphi^{*} Y]^{>0}(\mathbb{C}) \to \sfM^\ct_{g}(\mathbb{C}) \]
factor through $E(\mathbb{C})$ for all $(Y, w) \in \sC$. Then this case is known as a consequence of the ``geometric Zilber-Pink'' statement formulated in \cite[\S5, \S6]{BKU}. In fact, the locus $E$ constructed there is independent of the fixed Shimura variety $Y$ chosen, and depends only on the natural variation of Hodge structure $\IV$ on $\sfM^\ct_{g,\mathbb{C}}$. The components $E$ of $E_{\mathbb{C}}$ are special subvarieties for $\IV$, and $E_{\mathbb{C}}$ is called the atypical locus of $\IV$. Each component $E$ is ``atypical'' in the sense that a certain Hodge-theoretically defined expected dimension inequality is satisfied: the component $E$ has larger dimension than would be predicted by a na\"ive dimension count in $\sfA_{g}$.

\item[(iii)] Now suppose instead that $R = \mathbb{F}_{p}$, and again that $Y \subset \sfA_{g}$ is (the reduction of the canonical model of) a Shimura subvariety. Take again for $\sC$ the collection of all Hecke translates of $Y$. Then if $Y$ is not generically ordinary, then the reduced loci $[\varphi^{*} Y]^{\textrm{red}}$ are never Zariski dense in $\sfM^\ct_{g}$: the ordinary locus is non-empty and open in $\sfM^\ct_{g}$, and necessarily each component of $[\varphi^{*} Y]^{\textrm{red}}$ maps into the closed non-ordinary locus. (The situation where one does not pass to reduced subschemes is more subtle.)
\item[(iv)] Similarly if $R = \mathbb{F}_{p}$, na\"ive analogues of the Andr\'e-Oort conjecture are false: every abelian variety mod $p$ carries a large enough endomorphism ring to force the Frobenius action on its Tate module to factor through a torus, and so in particular such ``CM-like'' points are dense in $\sfM^\ct_{g}$.
\item[(v)] On the other hand, even when $R = \mathbb{F}_{p}$, certain more refined density questions for intersections with generically ordinary loci do not seem to be resolvable just by looking at the Frobenius action. A case of central interest to us is when $Y = \sfA_{g-k} \times \sfA_{k}$ with $k \in \{ 1, 2, 3 \}$, and $\sC$ is again taken to be the collection of Hecke translates of $Y$. In this case \eqref{expdimineq} holds, so one expects Zariski density. However Chavdarov has shown \cite[Cor. 2.6]{zbMATH01011427} that in any sufficiently general family of curves over $\mathbb{F}_{q}$ with $q = p^{k}$, the curves defined over $\overline{\mathbb{F}_{q}}$ with split Jacobian up to isogeny have density zero in a certain precise sense. In particular, one cannot deduce the \emph{Zariski} density in $\sfM^\ct_{g}$ of the loci $\varphi^{*} Y$ by showing that along any sufficiently general curve $C \subset \sfM^\ct_{g}$ the locus of split Jacobians has positive density in the sense of point counting (this technique does work for large enough collections $\sC$ of Hilbert modular varieties).
\item[(vi)] Finally there is the mixed-characteristic situation, where one takes $R = \mathbb{Z}$, and takes $\sC$ to be a collection of loci $Y \subset \sfA_{g}$ which lie entirely in positive characteristic $p = p(Y)$ depending on $Y$. Then the analogue of the geometric Zilber-Pink theorem of \cite[\S5, \S6]{BKU} is known in this case, and given in \cite[\S5.1, Prop. 5.8]{urbanik2023algebraiccyclelociintegral}; in particular, one can produce a strict subscheme $E \subset \sfM^\ct_{g}$ which factors all components of $[\varphi^{*} Y]^{>0}$ that have ``unexpected dimension'', and such that $E_{\mathbb{C}}$ is contained in the atypical locus of $\sfM^\ct_{g,\mathbb{C}}$. In \cite{zbMATH08159582} we used the techniques developed in \cite{urbanik2023algebraiccyclelociintegral} to establish structural results for Noether-Lefschetz loci, which we will extend to more general special loci in this paper.
\end{itemize}


We note that nothing in points (i) through (vi) depended on us working with $\sfM^\ct_{g}$, and one can carry out a similar analysis for any locally closed locus $T \subset \sfA_{g}$ with Hodge-generic complex fibre $T_{\IC}$, provided that one adjusts the dimension conditions accordingly. Similarly, although we focus on loci $T \subset \sfA_{g}$ in this paper, we expect our techniques to apply to any sufficiently general locus in any Hodge-type Shimura variety, or even to more general parameter spaces $T$ for smooth projective algebraic varieties. The application to the question of de Jong is of the latter type. In these more general cases one replaces the atypical locus of $\sfM^\ct_{g, \mathbb{C}}$ with an atypical locus associated to a natural variation of Hodge structure on $T_{\mathbb{C}}$.

\begin{remark}
    As an extreme case of unlikely intersection between the Torelli locus and a strict Shimura subvariety, one can consider the situation when the latter is completely contained in the former. This extreme situation has been extensively studied and we refer the readers to the excellent survey \cite{MO13} by Moonen and Oort for previous results. One examplary result is that there are exactly twenty families of cyclic covers of $\IP^1$ which are parametrized by Shimura varieties of Hodge type, due to Moonen \cite{Moo10}. Li, Mantovan, Pries, and Tang studied the Newton polygons of these curves extensively in a sequence of works. In particular, they showed that for five of these families there are always positive dimension supersingular family of sufficiently large characteristic satisfying suitable congruence conditions (see \cite[Thm.~7.1]{LMPT18} and its proof).  
\end{remark}

\subsection{Proof Strategy for Density}

Our results in this paper will be obtained by applying the mixed characteristic techniques of (vi) to deduce novel density and non-density results in large positive characteristic. Since the techniques of (vi) deduce \emph{mixed-characteristic non-density}, it is not obvious how one could expect to obtain \emph{fixed characteristic density} results from them. To illustrate, we briefly recall the method of Ciliberto-Harris-Miranda \cite{CHM} for establishing the density of Noether-Lefschetz loci in the parameter space $S(d)$ of degree $d$ smooth surfaces in $\IP^3$, which works over any field $k$. 

The theorem of \cite{CHM} can be proven directly in the case $d = 4$. Given this, the strategy of \cite{CHM} proceeds in two steps, described in \cite[\S3]{CHM}. For $d > 4$, they consider the full parameter space $\Sigma(d)$ of degree $d$ surfaces in $\IP^3$, not necessarily smooth, which contains $S(d)$ as a Zariski open subset. Within the boundary $\Sigma(d) \setminus S(d)$ they consider the subset $R(d)$ of surfaces which admit a plane as a component. In particular, a surface in $R(d)$ has a defining homogeneous equation $H L = 0$, with $L$ a linear form, and where $H = 0$ defines a surface of degree $d-1$. In the first step of the argument, one argues using the inductive hypothesis that the surfaces $H L = 0$ for which the degree $d-1$ surface $H = 0$ carries an extra line bundle are dense in $R(d)$. In the second step, one argues that enough such surfaces can be deformed out from $R(d)$, together with the line bundle, to give rise to a dense locus in $S(d)$. Because the locus $R(d)$ is not codimension one in $\Sigma(d)$, one has to argue that such deformations emanate out in all normal directions to $R(d)$ inside $\Sigma(d)$, which \cite{CHM} accomplish by replacing $R(d)$ with an appropriate blow-up.

It is tempting to attempt the same sort of strategy to study special loci inside $\sfM^\ct_{g}$. But while the geometry of the loci $\Sigma(d), S(d)$ and $R(d)$ is quite explicit --- $\Sigma(d)$ in particular is just a projective space --- the geometry of $\sfM^\ct_{g}$ can be harder to understand. Correspondingly, it is less obvious that, given an appropriate boundary component $B$ of some compactification $\overline{\sfM}_{g}$, special loci in $B$ can be deformed into $\sfM^\ct_{g}$ in all directions normal to $B$ in $\overline{\sfM}_{g}$. Indeed, the work of \cite{zbMATH04179531}, which proves most of \autoref{thm: main density} over $\mathbb{C}$, uses a different complex analytic approach inspired by the Green density criterion (cf. the discussion of Green's criterion in \cite{CHM}). (In \cite{zbMATH04179531} the $\bg := (1, g-1)$ case of \autoref{thm: main density} over $\mathbb{C}$ is attributed to unpublished work of Harris; we do not know whether his proof was valid in positive characteristic.)

\vspace{0.5em}

It turns out that making the Ciliberto-Harris-Miranda approach for density work for $\sfM^\ct_{g}$ naturally leads to unlikely-intersection-type \emph{non-density} questions for special intersections with boundary components. To understand the situation, let $\overline{\sfA}_{g}$ be Mumford's canonical partial compactification of $\sfA_{g}$ obtained by adding semiabelian degenerations of toric rank one (reviewed in \autoref{Mumfcompprop}) and let $\overline{\sfM}_{g}$ be the corresponding partial compactification of $\sfM^\ct_{g}$ by stable curves whose Jacobian is semiabelian of toric rank one. We have an extended Torelli map $\overline{\varphi} : \overline{\sfM}_{g} \to \overline{\sfA}_{g}$. We consider a Shimura morphism $w : Y \to \sfA_{g}$ whose image is a Hecke translate of $\sfA_{g-k} \times \sfA_{k}$ and construct a partial compactification $\overline{w} : \overline{Y} \to \overline{\sfA}_{g}$. Extending the notation from \eqref{fibprod} we can consider the pullback 
\begin{equation}
\label{fibprod2}
\overline{\varphi}^{*} \overline{Y} := \overline{\sfM}_{g} \times_{\overline{\varphi},\overline{\sfA}_{g},\overline{w}} \overline{Y} .
\end{equation}
To argue that the pullbacks $\overline{\varphi}^{*} \overline{Y}$ give rise to a Zariski dense subset of $\overline{\sfM}_{g}$ as the pair $(Y, w)$ varies, we start with an analysis of the situation on the boundary. The boundaries $\partial \sfM^\ct_{g}$ and $\partial \sfA_{g}$ of the compactifications $\overline{\sfM}_{g}$ and $\overline{\sfA}_{g}$ can be identified, up to a quasi-finite covering, with $\sfM_{g-1,2}$ and $\sA_{g-1}$, respectively; here $\sA_{g-1}$ is the total space of the universal abelian family of $\sfA_{g-1}$. (In what follows, for illustrative purposes, we abusively make the identifications $\partial \sfM^\ct_{g} := \sfM_{g-1,2}$ and $\partial \sfA_{g} = \sA_{g-1}$.) The map $\overline{\varphi}$ gives rise to a boundary map $\partial \varphi : \sfM_{g-1,2} \to \sA_{g-1}$ which sends the pair $(C, a, b)$ to $(J(C), \sO_{C}(a-b))$. This boundary map, in turn, fits into a diagram
\begin{equation}
\label{bounddiag}
\begin{tikzcd}
	\sfM_{g-1,2} & \sA_{g-1} \\
	\sfM_{g-1} & \sfA_{g-1}
	\arrow["\partial \varphi", from=1-1, to=1-2]
	\arrow[from=1-1, to=2-1]
	\arrow[from=1-2, to=2-2]
	\arrow["\varphi", from=2-1, to=2-2]
\end{tikzcd} .
\end{equation}
Now as (the image of) $Y$ approaches the boundary stratum $\sA_{g-1}$, the corresponding boundary strata of $Y$ come in two flavours:
\begin{itemize}
    \item[(L)] Hecke translates of $\partial \sfA_{g-k} \times \sfA_{k} \simeq \sA_{g-k-1} \times \sfA_{k}$. That is, loci in $\sA_{g-1}$ which lie over a Hecke translate of $\sfA_{g-k-1} \times \sfA_{k}$ and whose fibres are unions of torsion translates of abelian subvarieties of dimension $g-k-1$.
    \item[(R)] Hecke translates of $\sfA_{g-k} \times \partial\sfA_{k} \simeq \sfA_{g-k} \times \sA_{k-1}$. That is, loci in $\sA_{g-1}$ which lie over a Hecke translate of $\sfA_{g-k} \times \sfA_{k-1}$ and whose fibres are unions of torsion translates of abelian subvarieties of dimension $k-1$.
\end{itemize}
One can easily check that both loci of type (L) and loci of type (R) are dense in $\sA_{g-1, \mathbb{F}_{p}}$ for each prime $p$, and each such locus appears in the boundary of some $Y$. For $k$ in the appropriate range, one also deduces by induction on $g$ that for $p \gg 0$ intersections between $\sfM_{g-1, \mathbb{F}_{p}}$ and Hecke translates of either $\sfA_{g-k-1, \mathbb{F}_{p}} \times \sfA_{k, \mathbb{F}_{p}}$ or $\sfA_{g-k, \mathbb{F}_{p}} \times \sfA_{k-1, \mathbb{F}_{p}}$ are dense in $\sfM_{g-1, \mathbb{F}_{p}}$. The task is then twofold:
\begin{itemize}
    \item[(1)] Show that the latter density of intersections with $\sfM_{g-1, \mathbb{F}_{p}}$ implies the density in $\sfM_{g-1,2, \mathbb{F}_{p}}$ of intersections with strata of type (L) or (R); that is, lift the density statement along the left vertical arrow in \eqref{bounddiag}.
    \item[(2)] Show that the loci produced in (1) deform into the interior $\sfM_{g, \mathbb{F}_{p}}$ of $\overline{\sfM}_{g, \mathbb{F}_{p}}$.
\end{itemize}
To show (1), we want to show that for points $[C] \in \sfM_{g-1, \mathbb{F}_{p}}$ with split Jacobian, the surface in $J(C)$ swept out by $\sO_{C}(a-b)$ as $a, b$ varies intersects a prime-to-$p$ torsion translate of an abelian subvariety of the appropriate dimension. We accomplish this in \S\ref{weilIIest} by a point counting argument, which uses Frobenius eigenvalue estimates from Deligne's Weil II \cite{zbMATH03713855}. Having completed step (1), we are left with the situation where the loci $(\partial \varphi)^{*} (\partial Y_{\mathbb{F}_{p}})$ are dense in $\sfM_{g-1,2, \mathbb{F}_{p}}$. Showing (2) then amounts to showing that the loci $\overline{\varphi}^{*} (\overline{Y}_{\mathbb{F}_{p}})$ are not contained, up to passing to reduced subspaces, inside $\partial \sfM_{g, \mathbb{F}_{p}}$; that is, that we do not have
\begin{equation}
\label{badcontainment}
[\overline{\varphi}^{*} (\overline{Y}_{\mathbb{F}_{p}})]^{\textrm{red}} \subset (\partial \varphi)^{*} (\partial Y_{\mathbb{F}_{p}}) .
\end{equation}

We may observe that this is an unlikely intersection problem. Inside $\sA_{g-1, \mathbb{F}_{p}}$, the locus $\partial Y_{\mathbb{F}_{p}}$ and $\sfM_{g-1,2, \mathbb{F}_{p}}$ has a certain expected dimension of intersection. If in fact this intersection $(\partial \varphi)^{*} (\partial Y_{\mathbb{F}_{p}})$ contains the entirety of the intersection of $\overline{\sfM}_{g, \mathbb{F}_{p}}$ with $\overline{Y}_{\mathbb{F}_{p}}$, then this means that the locus $(\partial \varphi)^{*} (\partial Y_{\mathbb{F}_{p}})$ has larger-than-expected dimension. If such loci with larger-than-expected dimension are not dense in $\sfM_{g-1,2, \mathbb{F}_{p}}$, then necessarily \eqref{badcontainment} cannot happen away from a strict closed locus in $\sfM_{g-1,2, \mathbb{F}_{p}}$, and one obtains the desired density in $\sfM_{g, \mathbb{F}_{p}}$ by deforming out from the boundary. In particular, a question of \emph{density} for expected-dimensional intersections with $\sfM_{g, \mathbb{F}_{p}}$ reduces to a question of \emph{non-density} for unlikely-dimensional intersections with $\sfM_{g-1,2, \mathbb{F}_{p}}$. 

To make this strategy work, we have to adapt the techniques of \cite{urbanik2023algebraiccyclelociintegral} and \cite{zbMATH08159582} to study the period map $\partial \varphi$. Whereas \cite{urbanik2023algebraiccyclelociintegral} and \cite{zbMATH08159582} are concerned with period maps $\psi$ such that $\psi_{\IC}$ is associated to a pure variation of Hodge structure, the period map $(\partial \varphi)_{\mathbb{C}}$ is associated to a mixed variation of Hodge structure. This presents some additional subtleties; in particular instead of giving a mixed characteristic analogue of the geometric Zilber-Pink of \cite{BKU}, one should give a mixed characteristic analogue of \cite[Thm. 8.2]{zbMATH07305885} (cf. \cite[Thm. 7.1]{arXiv:2406.16628} for the general statement for admissible, polarizable $\mathbb{Z}$MVHS). We carry this out in \S\ref{Mg2flatsec}. Because our non-density is a ``mixed characteristic'' result, one is able to deduce the desired non-density in $\sfM_{g-1,2, \mathbb{F}_{p}}$ only for $p \gg 0$, which explains why we require $p \gg 0$ in \autoref{thm: main density}.

\subsection{Proof Strategy for Supersingular Loci}

We now explain the idea behind our results on supersingular loci. These results arise from a somewhat different application of the atypical-intersection techniques used above, which is perhaps more surprising in the following sense. Supersingularity is an intrinsically positive-characteristic condition, defined in terms of Frobenius on crystalline cohomology, and there is no reason a priori to expect transcendence results for complex period maps to constrain supersingular loci in characteristic $p$.

Let us first briefly recall the relevant techniques of \cite{urbanik2023algebraiccyclelociintegral} and \cite{zbMATH08159582} alluded before in more detail. Suppose that $T$ is equipped with a filtered flat vector bundle $(\sH,\Fil^\bullet,\nabla)$, together with a collection of global flat tensors lying in the middle part of the Hodge filtration. We construct a suitable frame bundle $\sM$ over $T$ and a period morphism $\sM\to\che{D}$, where $\che{D}$ is an appropriate flag variety. Using the flat connection, for $r!$ invertible on $T$, this upgrades to a morphism between jet spaces $\sM^d_r\to J^d_r\che{D}$ (see \autoref{refinedjetthm}). Now fix a closed subscheme $Z\subseteq\che{D}$. Ignore the factor issues discussed in \S\ref{refinedjetcorrsec}. If the expected dimension of an intersection of the period image with $Z$ is $e$, we take $d=\max\{1,e+1\}$. We construct a descending sequence $\sE_1\supseteq\sE_2\supseteq\cdots$ where $\sE_r$ is the Zariski closure of the points $s\in T$ for which, for a suitable choice of frame, there exists a non-degenerate jet $j:\ID^d_r\to T$ based at $s$ whose period jet factors through $J^d_rZ$. Since $T$ is Noetherian, $\sE_\infty:=\cap_{r\geq1}\sE_r$ stabilizes at some finite order. These compatible jets to arbitrarily high order then produce a formal intersection of dimension at least $d$. Then \cite{BKU} implies that $\sE_{\infty,\IC}$ locus is contained in a finite union of strict special subvarieties.

The key point for the applications below is that $Z$ need not have any Hodge-theoretic interpretation. For applications to Hodge loci, one normally intersects the period image with a Mumford-Tate subdomain. In our setting, we shall consider a $Z$ that instead arise from linear-algebraic constraints which appear only in characteristic $p$. Nevertheless, $Z$ \textit{makes sense over $\IZ$} as an algebraic subscheme of the compact dual. This is enough for the Ax--Schanuel argument, even though the $\IC$-fiber of $Z$ has no clear Hodge-theoretic meaning. 

The case of orthogonal Shimura varieties gives the cleanest illustration of this phenomenon. For geometric intuition, let us talk in terms of moduli spaces of polarized K3 surfaces. In characteristic zero, the condition $h^{2,0}=1$ means that the deforming a line bundle imposes a linear condition on the period. More generally, deforming a lattice of line bundles of rank $r$ is expected to impose $r$ conditions. In characteristic $p$, however, the supersingular locus has dimension $9$ inside the $19$-dimensional moduli space. But there is a lattice of line bundles of rank $22$ that deforms along the supersingular locus, which violates this expectation. 

To understand this odd situation, let $X$ be a supersingular K3 surface over an algebraically closed field $k$, and we consider the map on primitive parts
$$
c_{1,k}:\NS(X)_{\mathrm{prim}}\tensor_\IZ k\longrightarrow\H^2_\dR(X/k)_{\mathrm{prim}}.
$$
It fails to be injective, which cannot occur in characteristic zero. In fact, $\dim\ker(c_{1,k})$ is the Artin invariant $\sigma_0(X)$. The directions in this kernel impose no constraints on the period map. Since $\sigma_0 \leq 10$, a naive count would therefore predict that the supersingular locus has codimension $21-10=11$, and hence dimension at most $8$. But this is still off by one.

The final correction comes from the quadratic form. Below we just write $\sigma_0$ for $\sigma_0(X)$. The image of $c_{1,k}$ admits an orthogonal decomposition $M=M_I\oplus M_{\mathrm{sd}}$, where $M_I$ is a totally isotropic subspace of dimension $\sigma_0$ and $M_{\mathrm{sd}}$ is a non-degenerate subspace of dimension $21-2\sigma_0$. Writing $\w=\Fil^2$, the condition that the classes in $M$ remain in $\Fil^1$ is $\w\perp M$. Since $M_I$ is a maximal isotropic subspace of $M_{\mathrm{sd}}^\perp$, $\w \perp M$ if and only if $\w\subset M_I$. This implies that we should set the subscheme $Z \subset\che{D}$ to be $\IP(M_I)$, and hence codimension $20-\sigma_0$ in the $19$-dimensional quadric $\che{D}$. This gives the correct maximal dimension $9$ when $\sigma_0=10$. Although the same space $Z$ makes sense inside $\che{D}_\IC$ --- and this is important to our method, it shows up entirely out of a characteristic $p$ phenomenon. By the Hodge index theorem, the pairing on the primitive algebraic classes on a polarized complex K3 surface is definite, and hence there is no isotropic subspace. It is precisely this kind of non-Hodge-theoretic $Z$ to which we apply Ax--Schanuel in \cref{preciseorthogonalshimurasupsingthm}.

Once this picture is understood, it is natural to apply the same idea to de Jong's question. The main difference is that for surfaces of degree $d>4$ one has $h^{2,0}>1$, which creates two new issues. First, we do not know the Tate conjecture in this generality, so we cannot obtain the required horizontal classes by deforming actual line bundles. Instead we use the $p$-adic Tate module of the relevant F-crystal. Katz's equivalence between $\IZ_p$-local systems and unit-root F-crystals, a basic characteristic-$p$ analogue of the Riemann--Hilbert correspondence, allows us to ``spread out'' the generic Tate module to a unit-root sub-F-crystal after suitable shrinking and an \'etale cover. This produces the horizontal sections needed for the jet argument; see \cref{prop: semi-continuity,cor: Frob inv horizontal}. Second, for K3 surfaces, we have a useful bound on Artin invariant $\sigma_0 \leq 10$, but when $h^{2,0}>1$, there appears to be no useful analogue of this fact, unless we impose a $p$-torsion condition appearing in \cref{supsingquintic}. This condition may be viewed as a ``weakly minuscule'' condition and gives the desired bound on the Artin invariant.

For principally polarized abelian varieties, supersingularity is harder to detect infinitesimally because the isogeny to a product of supersingular elliptic curves is very inseparable. We therefore use only a partial consequence of supersingularity: a supersingular abelian variety $A$ of dimension $g$ contains an abelian subvariety $B$ of dimension $h$ for any $h \le g$. Set $Q=A/B$. Then we have a subspace $W \subset \H^1_\dR(A)$ given by $W := \H^1_\dR(Q)$ such that $\dim W \cap \Fil^1 \H^1_\dR(A) = g - h$. The polarization induces a possibly degenerate alternating form on this subspace. We use it to define an analogue of ``Artin-invariant''. Then we define a suitable space $Z$ to remember the condition of having a subspace $W$ with certain desired properties. Although this remembers only part of the supersingularity condition, the codimension of $Z$ grows quadratically with $g$, while $\dim\sfM_g=3g-3$ grows linearly. It turns out that $h = g - \lfloor (g + 1) / 3 \rfloor$ is the optimal choice and allows us to prove \cref{thm: Oort} for $g\geq 16$. We expect that a more refined description of the infinitesimal constraints imposed by supersingularity can improve this numerical bound.

\subsection{Flatness and Lifting} \label{flatliftintrosec} Our methods also allow us to give structural results for arithmetic special loci inside moduli spaces which lie over a Shimura variety. In the following result the reader can keep in mind the Torelli map $\varphi : \sfM_{g} \to \sfA_{g}$, viewed as a morphism of Deligne-Mumford stacks over $\Spec\, \mathbb{Z}$.

\begin{theorem}
\label{genericflatnessMg}
Let $L$ be a number field and $N \in \mathbb{Z}$ be a non-zero integer. Let $\varphi : T \to \sfA_{g}$ be a quasi-finite morphism of Deligne-Mumford stacks over $\sO_{L}[1/N]$, with $T$ smooth over $\sO_{L}[1/N]$, and such that $\varphi_{\mathbb{C}}$ has Hodge-generic image. Then after possibly increasing $N$, there exists a strict Zariski closed locus $E \subsetneq T$ with the following property: for every morphism of (integral canonical models of) Shimura varieties $w : Y \to \sfA_{g}$ such that
\[ \dim \varphi(T_{\mathbb{C}}) + \dim Y_{\IC} \geq \dim \sfA_{g,\IC} \]
the restriction $(\varphi^{*} Y)^{\circ}$ of $\varphi^{*} Y$ to the fibre above $T - E$ is syntomic (i.e., flat and lci) over $\sO_{L}[1/N]$. Moreover $E_{\IC}$ is a union of special loci of $T_{\IC}$, and each non-empty fibre of $(\varphi^{*} Y)^{\circ} \to \Spec\, \sO_{L}[1/N]$ has expected dimension
\[ \dim \varphi(T_{\IC}) + \dim Y_{\IC} - \dim \sfA_{g,\IC} . \]
\end{theorem}
\noindent A generalization of \autoref{genericflatnessMg} where the target is an arbitrary Shimura variety is given in \autoref{shimflatnessthmv2}.

Flatness results like the above imply results on lifting algebraic cycles. We state some consequences for prime-to-$p$ splittings of Jacobians; other results for lifting more general algebraic cycles are easily derived from \autoref{shimflatnessthmv2} (cf. our previous work \cite[Thm. 1.1.1]{zbMATH08159582}).

\begin{theorem}
\label{liftisodecompthm}
Outside the locus $E \subsetneq \sfM_{g}$ of \autoref{genericflatnessMg}, the following holds: for every partition $\bg := \{ g_i \}_{i = 1}^m$ with $\delta(\bg) \leq 3g-3$ (recall \eqref{deltaeq}), if $[C] \in \sfM_{g}(\overline{\mathbb{F}}_{p})$ is a curve such that $J(C)$ admits a prime-to-$p$ splitting of type $\bg$, then $[C]$ lifts to a point $[\widetilde{C}] \in \sfM_{g}(R)$ together with the splitting, with $R$ finite over $W(\overline{\mathbb{F}}_{p})$.
\end{theorem}

We can also lift isogenies between Jacobians for $g \leq 9$.

\begin{theorem}
\label{isogenyliftthm}
Fix an integer $g \leq 9$. Then there exists a strict closed locus $E \subsetneq \sfM_{g} \times \sfM_{g}$ with the following property: for every point $([C], [C']) \in (\sfM_{g} \times \sfM_{g} - E)(\overline{\mathbb{F}}_{p})$ and prime-to-$p$ isogeny $\beta : J(C) \to J(C')$, there exists a finite extension $R$ of $W(\overline{\mathbb{F}}_{p})$ and lifts $(\widetilde{C}, \widetilde{C}', \widetilde{\beta})$ to $R$ of the data $(C, C', \beta)$. Moreover $E_{\IC}$ is a union of special subvarieties.
\end{theorem}

\noindent That the isogeny $\beta$ of \autoref{isogenyliftthm} should lift to an isogeny of abelian varieties is an easy consequence of deformation theory, but that it should lift to an isogeny of Jacobians is more difficult. One does not expect this to be true unless $g \leq 9$, since the deformation locus of $\beta$ in $\sfA_{g} \times \sfA_{g}$ has dimension $g(g+1)/2$ over $\Spec\, W(\overline{\mathbb{F}}_{p})$, and the inequality
\[ 2 \dim \sfM_{g,\mathbb{C}} + g(g+1)/2 \geq 2 \dim \sfA_{g,\IC} \]
for positive integers $g$ is equivalent to $g \leq 9$. 

\begin{remark}
\label{rmk: effectivity}
    All atypical loci --- that is all the loci labelled ``$E$'' in the theorems in both this and the prior sections --- are constructed effectively using an explicit algebro-geometric procedure. This ultimately reduces to effective procedures for computing the filtered vector bundles with connection $(\mathcal{H}, F^{\bullet}, \nabla)$ coming from algebraic de Rham cohomology, as well as the explicit construction of the maps $\eta^{d}_{r}$ in \autoref{bigjetthm} using this data. For more details see the discussion in \cite[\S1.3]{zbMATH08159582}. 
\end{remark}

\subsection{Notations and conventions}


We adopt the following convention. Suppose that $T \to S$ is a quasi-finite morphism between Noetherian schemes. Then by the ``codimension of $T$ in $S$'', we just mean the codimension of the image of $T$ in $S$. If $\sfX$ is a stack and $W \subset \sfX$ is a closed substack, we sometimes refer to $W$ as a ``subvariety'' of $\sfX$. For instance, we will usually refer to ``special subvarieties'' of $\sfX$ in analogy with the situation when $\sfX$ is an actual algebraic variety, even though these substacks are not technically varieties but stacks.

For a curve $C$, we often simply write $J(C)$ for $\Jac(C)$. On a scheme $T$, an \textit{algebro-geometric} point $t \to T$ refers to a geometric point whose field of definition $k(t)$ is algebraically (rather than just separably) closed. 

\section{Jet Theory}
\label{sec: jet theory}

\subsection{Jet Schemes}
\label{subsec: jets}
Our proof will use a jet-theoretic construction developed in \cite{urbanik2023algebraiccyclelociintegral} that allows one to algebro-geometrically parameterize ``infinitesimal period maps'' to various finite orders $r$. Crucially this uses the notion of a jet-scheme, for which we refer to the exposition in \cite[\S3.1]{zbMATH08159582} and adopt the notation and conventions there. In particular, we define, for a ring $R$
\[ A^{d}_{r, R} = R[t_{1}, \hdots, t_{d}]/(t_{1}, \hdots, t_{d})^{r+1} , \]    and write $\mathbb{D}^{d}_{r, R}$ for $\Spec (A^{d}_{r, R})$. For an $R$-scheme $T$ and integers $r, d \geq 0$ we write $J^{d}_{r} T$ for the associated jet scheme (when it exists). 

The construction of $J^d_r T$ is compatible with base-change along maps $R \to R'$. Every map $\mathbb{D}^{d}_{r,R} \to \mathbb{D}^{d'}_{r',R}$ induces a morphism $J^{d'}_{r'} \to J^d_r$ of functors, which induces algebraic maps $J^{d'}_{r'} T \to J^{d}_{r} T$. The most important case for us will be when $d$ is fixed, and we consider the natural maps $\pi^{r}_{r-k} : J^d_r \to J^d_{r-k}$ induced by the natural surjections $A^d_{r} \to A^d_{r-k}$ obtained by quotienting out $(t_1, \hdots, t_d)^{r - k + 1}$.  Given $j$ an $R$-point of $J^d_r T$, the image $\pi^{r}_{r-k}(j)$ is then interpreted as the restriction of $j : \mathbb{D}^d_r \to T$ to the subdisk $\mathbb{D}^d_{r-k} \subset \mathbb{D}^d_r$ of order $r-k$ determined by $A^d_r \to A^d_{r-k}$.



There is an open subscheme $J^{d}_{r,\operatorname{nd}} T \subset J^{d}_{r} T$ which parameterizes those jets that are injective on tangent spaces, constructed as in \cite[Def 3.4]{urbanik2023algebraiccyclelociintegral}. Moreover we have a functorial closed subscheme $c(T) \subset J^{d}_{r} T$ which parameterizes those jets that are ``constant''; this subscheme is isomorphic to $T$ and arises as the image of a natural section $c : T \to J^{d}_{r} T$. We let $J^d_{r, \nc} T$ denote the space of non-constant jets, i.e., the complement of $c(T)$ in $J^d_r T$. 

\subsection{Infinitesimal Period Maps}
\label{infpermapsec}

We will use jets to study infinitesimal period maps. This notion is introduced in \cite{urbanik2023algebraiccyclelociintegral} under the name ``$r$-limp''. To define it, we suppose we have a triple $(\sH, \textrm{Fil}^{\bullet}, \nabla)$ consisting of a filtered vector bundle of rank $m$ with $R$-linear integral connection, all over the smooth $R$-scheme $T$. We write $\che{L}$ for the $R$-scheme parameterizing flags on $R^{m}$ whose graded quotients are locally free and of the same rank as the graded quotients of $(\sH, \textrm{Fil}^{\bullet})$ (cf. \cite[\S1.2]{urbanik2023algebraiccyclelociintegral}). We fix a positive integer $r$ and assume that $r!$ is invertible in $R$.

\begin{notn}
\label{disknot}
Let $k$ be a field together with a map $R \to k$, and $s \in T(k)$ a point. Then we set $\Delta^{r}_{s} := \Spec(\sO_{T_k, s}/\mathfrak{m}^{r+1})$, where $\mathfrak{m} \subset \sO_{T_{k}, s}$ is the maximal ideal corresponding to $s$. 
\end{notn}

\begin{notn}
For a vector space $V$ over a field $k$, we write $\textrm{Fr}(V)$ for the set of frames (bases) of $V$.
\end{notn}

\begin{lemma}
\label{trivlem}
Let $k, s, \Delta^{r}_{s}$ be as in \autoref{disknot}. Then every basis $b^{\bullet} = (b^{1}, \hdots, b^{m}) \in \textrm{Fr}(\sH_s)$ induces a unique horizontal trivialization $i_{b^{\bullet}} : \restr{(\sH, \nabla)}{\Delta^r_s} \xrightarrow{\sim} (\sO_{\Delta^r_s}, d)^{\oplus m}$ whose induced map $\sH_s \to k^{m}$ is the one induced by $b^{\bullet}$. 
\end{lemma}

\begin{proof}
This is \cite[Lem. A.0.1]{zbMATH08159582}, where we note that $r!$ is invertible in $k$ because it is so in $R$. 
\end{proof}

\begin{definition}
\label{rlimpdef}
Let $k, s, \Delta^{r}_s$ be as in \autoref{trivlem}. An $r$'th order local infinitesimal period map ($r$-limp) is a morphism $\Delta^{r}_s \to \che{L}_k$ corresponding to the filtration $i_{b^{\bullet}}(\textrm{Fil}^{\bullet})$ for some choice of basis $b^{\bullet}$ for $\sH_{s}$. 
\end{definition}

Given an $r$-limp $\psi$ and an element $j \in (J^d_r T)(k)$ one obtains an element $\psi \circ j \in (J^d_r \che{L})(k)$ by composition. In \cite{urbanik2023algebraiccyclelociintegral}, it is shown that this operation glues to give an algebraic map. To be more precise, consider, for some non-negative integer $d$, the $\GL_{m,R}$-torsor $\gamma^d_r : \sP^d_r \to J^d_r T$ obtained as the base-change along $J^d_r T \to T$ of the frame bundle of $\sH$. For $k$ a field admitting a map from $R$, the points of $\sP^d_r(k)$ are identified with pairs $(j, b^{\bullet}) \in (\pi^{r}_{0})^{-1}(s)(k) \times \textrm{Fr}(\sH_{s})$. Then the map $(j, \psi) \mapsto \psi \circ j$ induces a map $\alpha^d_r(k) : \sP^d_r(k) \to (J^d_r \che{L})(k)$, and the next theorem says that this map arises from a natural algebro-geometric map $\alpha^d_r : \sP^d_r \to J^d_r \che{L}$ over $R$.

The proof of \cite[Thm~1.14]{urbanik2023algebraiccyclelociintegral} gives the following:

\begin{theorem}
\label{bigjetthm}
There exists a $\GL_{m, R}$-equivariant morphism $\alpha^d_r : \sP^d_r \to J^d_r \che{L}$ whose induced map on $k$-points, for each field $k$ admitting a map $R \to k$, is the map $\alpha^d_r(k)$ described above. The resulting morphism 
\[ \eta^{d}_{r} : J^{d}_{r} T \to \GL_{m, R} \backslash J^{d}_{r} \che{L} \]
of Artin stacks over $R$ is compatible with 
\begin{itemize}
\item[-] base-change along maps $R \to R'$ of $\IZ[1/r!]$-algebras;
\item[-] pullback along maps $g : T' \to T$ of smooth $R$-schemes, in the sense that $\eta^{d}_{r} \circ J^{d}_{r} g$ is the map associated to $(g^{*} \sH, g^{*} \Fil^{\bullet}, g^{*} \nabla)$; and
\item[-] maps of functors $J^{d}_{r} \to J^{d'}_{r'}$ for $r' \leq r$ induced by maps $\mathbb{D}^{d'}_{r',R} \to \mathbb{D}^{d}_{r,R}$, in the sense that the natural square involving $J^{d}_{r} T, J^{d'}_{r'} T, \GL_{m} \backslash J^{d}_{r} \che{L}$ and $\GL_{m} \backslash J^{d'}_{r'} \che{L}$ commutes. 
\end{itemize}
\end{theorem}

\subsection{Refined Jet Correspondence}
\label{refinedjetcorrsec}

We can give a more refined jet correspondence that takes into account global flat sections appearing in $\sH$ and its tensor powers. For simplicity we assume that $R = \sO_{L}[1/N]$. Our construction will require possibly replacing $L$ with a finite extension and increasing $N$. 

\begin{construction}
\label{gentorsorconstr} We define
\begin{equation}
\label{Hotimes}
\sH^{\otimes} := \bigoplus_{a, b \geq 0} \sH^{\otimes a} \otimes (\sH^{\vee})^{\otimes b}
\end{equation}
to be the infinite dimensional vector bundle obtained by summing all tensor powers of $\sH$ and its dual. The vector bundle $\sH^{\otimes}$ has a natural filtration and connection we again denote by $(\textrm{Fil}^{\bullet}, \nabla)$. If one considers the global sections $\sH^{\otimes}(T)$, there is a natural $R$-submodule $\sH^{\otimes}(T)^{\nabla = 0}$ which consists of exactly the flat global sections.

After replacing $L$ with a finite extension and increasing $N$, we can assume that $T$ admits an $R$-point $t$. We fix an isomorphism $\sH_{t} \cong R^{m}$, the choice of which does not matter. Via this isomorphism, the frame bundle $\sP := \sP^{0}_{0}$ of $\sH$ is identified with the principal bundle of isomorphisms associated to the sheaf $\sF := \sheafhom(\sH, \sO_{T} \otimes_{R} \sH_{t})$. Under this identification, $\sP$ is an open subscheme of the relative spectrum (cf. \cite[\href{https://stacks.math.columbia.edu/tag/01LQ}{Section 01LQ}]{stacks-project})
\[ \sP \subset \underline{Spec}_{T}(\textrm{Sym}(\sF)) \]
which is characterized by the property that its points correspond to invertible linear maps. We define a closed subscheme of $\sP$ by defining a closed subscheme of $\underline{Spec}_{T}(\textrm{Sym}(\sF))$ and then intersecting. More precisely, we can, for each global section $h$ of $\sH^{\otimes}$, consider the locus
\[ \sZ(h) \subset \underline{Spec}_{T}(\textrm{Sym}(\sF)) \]
defined by the property that $(s, \varphi) \in \Hom(\sH_{s}, \sH_{t})$ lies in $\sZ(h)$ if $\varphi(h_{s}) = \varphi(h_{t})$.  Using explicit coordinates for $\underline{Spec}_{T}(\textrm{Sym}(\sF))$ one can observe this is a closed locus, and define
\[ \sM := \sP \cap \bigcap_{h \in \sH^{\otimes}(T)^{\nabla = 0}} \sZ(h) .\]

We let $\mathbf{H}_{t} \subset \GL(\sH_{t})$ be the $R$-subgroup scheme defined as the stabilizer of $[\sH^{\otimes}(T)^{\nabla = 0}]_{t}$. 
\end{construction}

\begin{lemma}
After possibly increasing $N$, $\sM$ is an $\mathbf{H}_{t}$-torsor in the \'etale topology on $T$.
\end{lemma}

\begin{proof}
We first observe that the complex analytification $\sM^{\textrm{an}}_{\mathbb{C}}$ is an $\mathbf{H}_{t,\IC}$-torsor in the usual sense of topology. This one can check locally, so we can fix an open neighbourhood $B \subset T(\mathbb{C})$ and check everything after restricting all the data to $B$. In this case we can assume that $\sH$ is trivialized by a flat frame, in which case $\sM$ is identified with a trivial torsor of the form $B \times \mathbf{H}_{t} \to B$ and the result follows. 

Now we consider the statement of the lemma for $\sM_{\mathbb{C}}$. It follows from the complex analytic version that $\sM_{\mathbb{C}} \to T$ is flat and of finite presentation. Considering the trivialization $\sM_{\mathbb{C}} \times \mathbf{H}_{t,\mathbb{C}} \xrightarrow{\sim} \sM_{\mathbb{C}} \times_{T_{\mathbb{C}}} \sM_{\mathbb{C}}$ given on points by $((\beta, s), \gamma) \mapsto ((\beta, s), (\gamma \circ \beta, s))$, one also sees that $\sM_{\mathbb{C}}$ is an fppf torsor. But any fppf torsor is an \'etale torsor \cite{51051}. The statement likewise follows for $\sM_{L}$.

Having shown the statement for $\sM_{L}$, choose an \'etale map $U \to T$ such that $\sM_{L} \times_{T} U \to U$ is trivial. Choosing $N$ large enough the trivialization spreads out, hence so does the torsor structure.
\end{proof}

We now define $\sM^{d}_{r} \subset \sP^{d}_{r}$ by pulling back $\sM$ along $J^{d}_{r} T \to T$. We extend our description of $\sP^{d}_{r}(k)$ above to describe the subset $\sM^{d}_{r}(k) \subset \sP^{d}_{r}(k)$. Recall that we thought of points of $\sP^{d}_{r}(k)$ as pairs $(j, b^{\bullet})$, where $j \in (J^{d}_{r} T)(k)$ was a jet above $s \in T(k)$, and $b^{\bullet}$ was the basis of $\sH_{s}$ corresponding to the associated frame $\varphi : \sH_{s} \xrightarrow{\sim} k^{m}$. Having made the identification $k^{m} \cong \sH_{t,k}$, we can regard $\varphi$ as a map $\varphi : \sH_{s} \xrightarrow{\sim} \sH_{t,k}$, such that the pair $(j, b^{\bullet}) \in \sP^{d}_{r}(k)$ lies in $\sM^{d}_{r}(k)$ if and only if $\varphi(h_{s}) = h_{t,k}$ for all $h \in \sH^{\otimes}(T)^{\nabla = 0}$. 


\begin{definition}
Using the identification $\sH_{t} \cong R^{m}$, we define $\che{D}_{\textrm{mon}} \subset \che{L}$ to be the locally closed subscheme of $\che{L}$ obtained as the orbit of $\mathbf{H}_{t}$. 
\end{definition}

\subsection{Pure VHS Setting}

\begin{definition}
Using the identification $\sH_{t} \cong R^{m}$, we define $\che{D}_{\textrm{Hg}} \subset \che{L}$ to be the closed subscheme consisting of those flags $F^{\bullet}$ such that $h_{t}$ lies in the middle piece of $F^{\bullet}$, for all $h \in \sH^{\tensor}(T)^{\nabla = 0}$ for which $h_{t}$ lies in the middle piece of $\textrm{Fil}^{\bullet}$.
\end{definition}

\begin{lemma}
Suppose that the data $(\sH_{\IC}, \Fil^{\bullet}_{\IC}, \nabla_{\IC})$ underlies a polarized $\mathbb{Q}$-variation of pure Hodge structures $\IV$, and that $\mathbf{H}_{t,\mathbb{C}}$ is connected. Then the locus $\che{D}_{\operatorname{Hg}, \IC}$ is a finite union of $\mathbf{H}_{t,\IC}$ orbits, one of which is $\che{D}_{\operatorname{mon}, \IC}$. Consequently, after possibly increasing $N$ and replacing $L$ with a finite extension, $\che{D}_{\operatorname{Hg}}$ is a finite union of $\mathbf{H}_{t}$ orbits, one of which is $\che{D}_{\operatorname{mon}}$.
\end{lemma}

\begin{proof}
In what follows we set $\che{D} = \che{D}_{\textrm{Hg}}$ for convenience. To start with we replace $R$ with $\IC$, and in particular regard $t$ as a $\IC$-point. Then the fibre $\sH_{t}$, the algebraic group $\mathbf{H}_{t}$, and the varieties $\che{D}$ and $\che{L}$ are all complex-algebraic. The fibre $\IV_{t}$ then gives a $\mathbb{Q}$-structure for the complex vector space $\sH_{t}$ and the algebraic group $\GL(\sH_{t})$, and under our identifications a $\mathbb{Q}$-structure for the flag variety $\che{L}$. We denote the resulting $\mathbb{Q}$-varieties by $\sH_{t,\mathbb{Q}}$, $\GL(\sH_{t,\mathbb{Q}})$ and $\che{L}_{\mathbb{Q}}$, respectively. Likewise, one has an infinite rank local system $\mathbb{V}^{\otimes} = \bigoplus_{a, b \geq 0} \mathbb{V}^{\otimes a} \otimes (\mathbb{V}^{\vee})^{\otimes b}$ which underlies the bundle $\sH^{\otimes}$ from (\ref{Hotimes}). This then equips $W := \sH^{\otimes}(T)^{\nabla = 0}$ with a $\mathbb{Q}$-structure $W_{\mathbb{Q}}$, and hence gives us a $\mathbb{Q}$-structure $\mathbf{H}_{t,\mathbb{Q}}$ for the group $\mathbf{H}_{t}$ which is the pointwise stabilizer of this subspace. It now suffices to give an alternative construction of $\che{D}$ using this $\mathbb{Q}$-structure such that $\che{D}$ is a finite union of $\mathbf{H}_{t}(\IC)$-orbits. 

The subspace $W_{\mathbb{Q}}$ has a natural subspace $W^{\textrm{Hg}}_{\mathbb{Q}}$ consisting of those vectors which are everywhere Hodge for the filtration $\Fil^{\bullet}$. We can consider the fibre $[W^{\textrm{Hg}}_{\mathbb{Q}}]_{t}$ of $W^{\textrm{Hg}}_{\mathbb{Q}}$ at $t$, or alternatively, we can consider the subspace $W^{\textrm{Hg}}_{\mathbb{Q},t} \subset W_{\mathbb{Q},t}$ of the fibre of $W_{\mathbb{Q}}$ at $t$ consisting of those $w \in W_{\mathbb{Q},t}$ such that $w$ is Hodge for the Hodge structure at $t$ (but not necessarily globally). We have a natural map $[W^{\textrm{Hg}}_{\mathbb{Q}}]_{t} \to W^{\textrm{Hg}}_{\mathbb{Q},t}$, which is an isomorphism as a consequence of Schmid's fixed part theorem \cite{zbMATH03437307}. We therefore drop the distinction between $W^{\textrm{Hg}}_{\mathbb{Q},t}$ and $[W^{\textrm{Hg}}_{\mathbb{Q}}]_{t}$.

Let us write $\mathbf{G} = \mathbf{G}_{t} \subset \GL(\sH_{t,\mathbb{Q}})$ for the algebraic group defined as the pointwise stabilizer of $W^{\textrm{Hg}}_{\mathbb{Q},t}$. Write $h_{t} : \mathbb{S} \to \GL(\IV_{t})$ for the Hodge structure at $t$, and write $\che{\NL}_{\mathbf{G}} \subset \che{L}_{\mathbb{Q}}$ for the $\mathbb{Q}$-algebraic subvariety defined by the property that a flag $F^{\bullet} \in \che{L}(\IC)$ lies inside $\che{\NL}_{\mathbf{G}}$ if and only if each element of $W^{\textrm{Hg}}_{\mathbb{Q},t}$ is Hodge for $F^{\bullet}$ (lies in the middle step of the filtration on the appropriate tensor space). Then the following facts are known about the relationships between the above data: 
\begin{itemize}
    \item[(i)] as a consequence of \cite[Thm. 1]{zbMATH00060857} the group $\mathbf{H}_{t}$ is a normal subgroup of the derived subgroup of $\mathbf{G}$;
    \item[(ii)] by \cite[VI.B.1]{zbMATH06031035} each component of $\che{\NL}_{\mathbf{G}}$ is an orbit of $G(\IC)$; and
    \item[(iii)] there exists $\gamma \in \mathbf{G}(\mathbb{R})$ such that $\gamma \cdot h_{t}$ has Mumford-Tate $\mathbf{M} \subset \mathbf{G}$ whose derived subgroup agrees with $\mathbf{H}_{t}$. 
\end{itemize}
To see (iii) we may argue as follows. First let us write $\che{D}_{\mathbf{G}} \subset \che{\NL}_{\mathbf{G}}$ for the component containing $h_{t}$. Then applying \cite[III.A]{zbMATH06031035} the orbit $D_{\mathbf{G}} := \mathbf{G}(\mathbb{R}) \cdot h_{t}$ factors as a product $D_{\mathbf{G}} = D_{\mathbf{H}} \times D_{\mathbf{L}}$, with
\begin{itemize}
    \item $\mathbf{G} = \mathbf{H} \cdot \mathbf{L} \cdot \mathbf{T}$ the almost-direct decomposition of the reductive $\mathbb{Q}$-group $\mathbf{G}$, where $\mathbf{H} = \mathbf{H}_{t}$, $\mathbf{L}$ is the complementary semisimple factor, and $\mathbf{T}$ is the torus part; and
    \item $D_{\mathbf{H}} = \mathbf{H}(\mathbb{R}) \cdot h_{t}$ and $D_{\mathbf{L}} = \mathbf{L}(\mathbb{R}) \cdot h_{t}$. 
\end{itemize}
Now to obtain $\gamma$ as in (iii) we start by choosing $\gamma' \in G(\mathbb{R})$ such that $\gamma' \cdot h_{t}$ is a CM-Hodge structure; this is possible as a consequence of \cite[VI.C.1]{zbMATH06031035} and \cite[VI.A.2]{zbMATH06031035}. Then the map $\gamma' \cdot h_{t} : \mathbb{S} \to \mathbf{G}_{\mathbb{R}}$ factors through some product $\mathbf{T}_{H} \cdot \mathbf{T}_{L} \cdot \mathbf{T}$ of $\mathbb{Q}$-tori compatibly with the decomposition $\mathbf{G} = \mathbf{H} \cdot \mathbf{L} \cdot \mathbf{T}$. It therefore suffices to show that if we take a very general $h \in \mathbf{H}(\mathbb{R})$ and set $\gamma = h \gamma'$, then the Mumford-Tate group of $\gamma \cdot h_{t}$ agrees with $\mathbf{H} \cdot \mathbf{T}_{L} \cdot \mathbf{T}$.

If not, then there must exist a strict $\mathbb{Q}$-subgroup $\mathbf{H}' \subsetneq \mathbf{H}$ which contains the image under the map $\mathbf{G}/(\mathbf{L} \cdot \mathbf{T}) \to \mathbf{H}$ of a very general element of $D_{\mathbf{H}}$, and hence of $D_{\mathbf{G}}$. But that would imply that a very general element of $D_{\mathbf{G}}$ has Mumford-Tate group strictly contained in $\mathbf{G}$, which is not true (take the Mumford-Tate group associated to a very general point in our base $T(\IC)$). We have thus shown (iii) above.


Below we write $V$ for $\IV_t$. 
Now let $w_{1}, \hdots, w_{r} \in \bigoplus_{a, b \geq 0} V^{\otimes a} \otimes (V^{\vee})^{\otimes b}$ be a collection of Hodge tensors such that $\mathbf{M}$ is exactly the pointwise stabilizer of $w_{1}, \hdots, w_{r}$. Write $\che{\NL}_{\mathbf{M}} \subset \che{\NL}_{\mathbf{G}}$ for the locus of Hodge flags for which the $w_{1}, \hdots, w_{r}$ are Hodge tensors. Then again by \cite[VI.B.1]{zbMATH06031035} each component of $\che{\NL}_{\mathbf{M}}$, and in particular the component containing $\gamma \cdot h_{t}$, is a $\mathbf{M}(\IC)$-orbit. Since the torus part of $\mathbf{M}$ acts trivially on the components of $\che{\NL}_{\mathbf{M}}$, all of which are proper algebraic varieties, it follows that each component of $\che{\NL}_{\mathbf{M}}$, and in particular the component containing $\gamma \cdot h_{t}$, is a $\mathbf{H}_{t}(\IC)$-orbit. 

If we now consider the locus $\gamma^{-1} \che{\NL}_{\mathbf{M}}$, then because $\gamma$ commutes with $\mathbf{H}_{t}(\IC)$ its components are also $\mathbf{H}_{t}(\IC)$-orbits. Moreover $\gamma^{-1} \che{\NL}_{\mathbf{M}}$ is the locus of flags for which the complex tensors $\gamma^{-1} w_{1}, \hdots, \gamma^{-1} w_{r}$ are ``Hodge'' (lie in the middle piece of the appropriate filtration on the associated tensor space). If we construct a period map as in \cite[III.A]{zbMATH06031035} by choosing $t$ as our basepoint, then the lifted period map $\widetilde{T}_{\IC} \to \che{L}$ will factor through $D_{\mathbf{H}} \subset \gamma^{-1} \che{\NL}_{\mathbf{M}}$, and so it follows that the $\gamma^{-1} w_{1}, \hdots, \gamma^{-1} w_{r}$ induced global Hodge tensors of the complex local system $\mathbb{V}_{\IC}$; in particular, one obtains the equality $\che{D} = \gamma^{-1} \che{\NL}_{\mathbf{M}}$, completing the proof.
\end{proof}

One can ensure that $\mathbf{H}_{t}$ is (geometrically) connected after increasing $N$ and replacing $T$ with a finite \'etale cover, and we assume that we have done so. For the rest of this section we set $\che{D} = \che{D}_{\textrm{mon}}$.

\begin{construction}
\label{constr: functoriality of pair}
    The pair $(H := \bH_t, \che{D})$ is functorial in the pointed space $(T, t)$ in the following sense. Suppose that $T' \to T$ is a morphism of $R$-schemes and $t' \in T'(R)$ is a base point of $T'$. Then we can form $(H', \che{D}')$ over $T'$ using the restriction $\sH|_{T'}$. Viewing $t'$ as an $R$-point on $T$, we may consider the $H$-orbit of the isomorphism $\varphi : \sH_{t'} \stackrel{\sim}{\to} \sH_t$ inside the torsor $\sM$. Using $\varphi$, we obtain an embedding $(H', \che{D}') \into (H, \che{D})$. Note that the orbit of the embedding under the $H$-action on $(H, \che{D})$ (via $\gamma \cdot (h, F^\bullet) = (\gamma h \gamma^{-1}, \gamma F^\bullet)$) depends only on the morphism $T' \to T$ and not on the choice of $t'$ and $\varphi$. 
\end{construction}

\begin{lemma}
\label{refinedjetthm}
The map $\alpha^{d}_{r} : \sP^{d}_{r} \to J^{d}_{r} \che{L}$ from \autoref{bigjetthm} above restricts to give a map $\sM^{d}_{r} \to J^{d}_{r} \che{D}$. In particular we obtain a map of Artin stacks
\[ \eta^{d}_{r} : J^{d}_{r} T \to \mathbf{H}_{t} \backslash J^{d}_{r} \che{D} \]
over $R$, compatible with base-change along maps $R \to R'$ of $\mathbb{Z}[1/r!]$-algebras, and functorial in maps of functors $J^{d}_{r} \to J^{d'}_{r'}$ for $r' \leq r$ induced by maps $\mathbb{D}^{d'}_{r',R} \to \mathbb{D}^{d}_{r,R}$. When $k$ is a field admitting a map from $R$, the description of the map on $k$-points is given by $(j, \varphi) \mapsto \psi \circ j$, where $\psi$ is the $r$-limp $\Delta^{r}_{s} \to \che{L}_{k}$ of \autoref{rlimpdef}. The map $\psi$ factors through $\che{D}_{k}$ by construction.
\end{lemma}

\begin{proof}
It suffices to prove that the restriction of $\alpha^{d}_{r}$ to $\sM^{d}_{r}$ factors through $J^{d}_{r} \che{D}$, since the rest follows from \autoref{bigjetthm}. Applying the base-change compatibility of the construction one can reduce to checking this on closed points, where it reduces to showing that $r$-limps $\psi$ factor through $\che{D}$. Using the fact that $\che{D}$ is a component of $\che{D}_{\operatorname{Hg}}$, this reduces to the claim that if $\widetilde{\varphi} \in \sheafhom(\restr{\sH}{\Delta^{r}_{s}}, \sO_{\Delta^{r}_{s}} \otimes_{k} \sH_{t})$ is a flat extension of some $\varphi : \sH_{s} \to \sH_{t,k}$, then $\widetilde{\varphi}(h)$ is constant for each $h \in \sH^{\otimes}(T)^{\nabla = 0}$. This can be checked directly using the fact that $h$ is flat.
\end{proof}

\section{Flatness Results (pure case)}

\subsection{Construction of Atypical Loci}

\begin{setup}
\label{setup base}
    Let $R = \sO_L[1/N]$, fix an embedding $\iota : L \into \IC$. Let $T$ be a smooth $R$-scheme with a base point $t \in T(R)$. Let $(\sH, \Fil^\bullet, \nabla)$ be a filtered flat vector bundle on $T$ such that its $\IC$-fiber underlies a polarizable $\IZ$-VHS $\IV$. Assume that $t_\IC$ is Hodge-generic and set $G := \MT(\IV_{t_\IC})$. The group $G$ is reductive; let $G_{1}, \hdots, G_{k}$ be its $\mathbb{Q}$-simple factors. Assume that $G$ (resp. $G_i$ for $i = 1, \cdots, k$) spreads out to a smooth reductive group scheme $\bG_t$ (resp. $\bG_{i, t}$). Then after reordering there exists $j \leq k$ such that $\mathbf{G}_{1,t} \cdots \mathbf{G}_{j,t} = \mathbf{H}_{t}$. We have a corresponding decomposition $\che{D} = \che{D}_{1} \times \cdots \times \che{D}_{j}$. 
\end{setup}

We now apply \autoref{refinedjetthm} to the triple $(\sH, \textrm{Fil}^{\bullet}, \nabla)$ over $T$ in the above setup to construct an atypical locus (with respect to some closed subscheme of $\che{D}$). 

\begin{notation}
For a scheme $X$ and integers $d, r \geq 0$, we write $J^{d}_{r,c} X \subset J^d_r X$ for the closed subscheme of constant jets, and $J^{d}_{r,\operatorname{nc}} X \subset J^{d}_{r} X$ for the complementary open subscheme of non-constant jets, defined as in \cite[\S3]{urbanik2023algebraiccyclelociintegral}.
\end{notation}

\begin{construction} 
\label{constr: atypical E}
\emph{(Atypical locus)}
In \cref{setup base}, we now fix a closed $R$-subscheme $Z \subset \che{D}$ which is geometrically irreducible over $R$ and has relative dimension $e > 0$ over $R$ and do the following constructions. Let $I^c_{Z} \subset \{ 1, \hdots, j \}$ be the subset defined by the property that $Z \to \che{D}_{i}$ is constant for $i \in I^c_Z$, and let $I^\nc_{Z} \subset \{ 1, \hdots, j \}$ be its complement. Accordingly, we set $\che{D}^c_Z := \prod_{i \in I^c_Z} \che{D}_i$ and $\che{D}^\nc_Z := \prod_{i \in I_Z^\nc} \che{D}_i$. Let $\pi_{c} : \che{D} \to \che{D}^{c}_{Z}$ and $\pi_{\nc} : \che{D} \to \che{D}^{\nc}_{Z}$ be the two projections. Define $\sM^{d}_{r}(Z) \subset \sM^{d}_{r}$ to be the inverse image of $J^{d}_{r} Z$. Here is a commutative diagram that summarizes the situation: 
\[\begin{tikzcd}
	{\sM^d_r(Z)} & {\sM^d_r} & {J^d_r T} \\
	{J^d_r Z} & {J^d_r \che{D}} & {J^d_r \che{D}^\nc_Z } \\
	& {J^d_r \che{D}_Z^c }
	\arrow[from=1-1, to=1-2]
	\arrow[from=1-1, to=2-1]
	\arrow["\lrcorner"{anchor=center, pos=0.125}, draw=none, from=1-1, to=2-2]
	\arrow[from=1-2, to=1-3]
	\arrow["{{\alpha^d_r}}", from=1-2, to=2-2]
	\arrow[from=2-1, to=2-2]
	\arrow["{\textrm{constant}}"', from=2-1, to=3-2]
	\arrow[from=2-2, to=2-3]
	\arrow[from=2-2, to=3-2]
\end{tikzcd}\]

We then define, for each $r \geq \xi$, the locus $\sM^{d}_{r}(Z, \xi) \subset \sM^{d}_{r}(Z)$ to be the open subscheme obtained by intersecting the following two open conditions (recall the notion of non-degenerate and non-constant jets in \cref{subsec: jets}):
\begin{equation}
\label{cond: nd}
\begin{cases}
    \textrm{the locus $\sM^{d}_{r}(Z, \xi)$ lies above $J^d_{r,nd} T$, which is open in $J^d_r T$; } \\
    \textrm{the locus $\sM^{d}_{r}(Z, \xi)$ lies above $J^d_{\xi, \nc} \che{D}_{i}$ for each $i \in I^{\operatorname{nc}}_{Z}$, which is open in $J^d_\xi \che{D}_{i}$. } 
\end{cases}
\end{equation}
Finally, fix some $\xi \in \IN$ and for each $d, r \in \IN$, let $E^d_r$ be the scheme-theoretic image of the map $\sM^{d}_{r}(Z, \xi) \to T$, and set $E^d := \cap_{r \ge \xi} E^d_r$. Note that since each $E^d_r$ is a closed subscheme by definition, and $T$ is Noetherian, the infinite intersection $\cap_{r \ge \xi} E^d_r$ has to stabilize, i.e., $E^d = E^d_{r_0}$ for $r_0 \gg 0$. 
\end{construction}

When $d$ is suitably large (cf. \cref{atyplocusconstrprop} below), we say that $E := E^d$ is the \textbf{atypical locus with respect to $Z$}. To formulate this we make the following definition. 

\begin{definition}
\label{def: n(I)}
    In \cref{setup base}, for each subset $I \subset \{ 1, \hdots, j \}$, we define $n(I) \in \IN$ to be the dimension of a general non-empty fibre of the composition $B \stackrel{\psi}{\to} \che{D}_\IC \to \prod_{i \in I} \che{D}_{i, \IC}$, where $\psi : B \to \che{D}_\IC$ is any local period map on some ball $B \subset T^{\textrm{an}}_{\mathbb{C}}$. This is easily seen to be independent of $\psi$ and depends only on $I$.
\end{definition}

\begin{proposition}
\label{atyplocusconstrprop}
Suppose that in \cref{setup base}, 
\begin{equation}
\label{eqn: d > bla}
    d > n(I_Z^c) -\dim \che{D}^\nc_Z + \dim_{R} Z.
\end{equation}
Then $E^d = E^{d}_{r_{0}}$ is not Zariski dense. If in addition the period map induced by $\IV$ has no positive-dimensional fibres, and the algebraic monodromy group $\mathbf{H}_{t}$ is $\mathbb{Q}$-simple, then the underlying reduced subscheme of $E^{d}_{\IC}$ is contained in a finite union of special subvarieties of $T_{\IC}$.
\end{proposition}

\begin{definition}
\label{Zatyplocusdef}
When 
\[ d = \max \{ 1, n(I_Z^c) -\dim \che{D}^\nc_Z + \dim_{R} Z + 1 \} , \]
we call the locus $E^{d}_{\IC}$ of \autoref{atyplocusconstrprop} the $Z$-atypical locus of $(S_{\IC}, \mathbb{V})$. We will sometimes just say ``atypical locus'' when $Z$ is understood.
\end{definition}

\begin{proof}
Note also that by the non-degeneracy condition in (\ref{cond: nd}), for $r \ge 1$ and $d \gg 0$, $E^d_r = \emptyset$. The right hand side of (\ref{eqn: d > bla}) is the ``expected'' dimension of $B \times_{\che{D}_\IC, \psi} Z_\IC$ over $\IC$, where $\psi : B \to \che{D}_\IC$ is a local period map as in \cref{def: n(I)}. 

Fix some $r_{0} \geq \xi$ and consider the intersection of constructible sets
\[ \sJ_{r_{0}} = \bigcap_{r \geq r_{0}} \textrm{im}[\sM^{d}_{r}(Z, \xi) \to \sM^{d}_{r_{0}} \times J^{d}_{r_{0}} \che{D}] . \]
Let $(p_{r_{0}}, \sigma_{r_{0}})$ be a $\mathbb{C}$-point in $\sJ_{r_{0}}$. Set $\ID^d_{\infty, \IC} := \varinjlim_r \ID^d_{r, \IC} \simeq \Spf(\IC[\![x_1, \cdots, x_d]\!])$. Applying \cite[Lem 5.3]{zbMATH07871772} to the sequence $\{ \sM^{d}_{r}(Z, \xi) \}_{r \ge r_{0}}$ at the point $(p_{r_{0}}, \sigma_{r_{0}})$ and arguing as in the proof of \cite[Prop. 4.2.5]{zbMATH08159582} we obtain a commutative diagram of formal objects
\begin{equation}
    \begin{tikzcd}
	{\ID^d_{\infty, \IC}} & {\what{S}_s} \\
	{Z_\IC} & {\che{D}_\IC}
	\arrow["{j_\infty}", from=1-1, to=1-2]
	\arrow["{\sigma_\infty}"', from=1-1, to=2-1]
	\arrow["{\psi_\infty}", from=1-2, to=2-2]
	\arrow[hook, from=2-1, to=2-2]
\end{tikzcd}
\end{equation}
Moreover $j_{\infty}$ is injective on tangent spaces (hence an embedding), and $\sigma_{\infty}$ is non-constant when projected to each $\che{D}_{i}$ where $i \in I^\nc_{Z}$.

The fact that $j_\infty$ is an embedding implies that $\dim (\psi_{\infty})^{-1}(Z_\IC) \geq d$. Writing $\Omega' := \che{D}^\nc_Z$ for simplicity, from our assumption (\ref{eqn: d > bla}) on $d$, we have 
\begin{equation}
\label{atypineq}
\dim \Omega' - \dim Z_{\mathbb{C}} > n(I_Z^c) - \dim (\psi_{\infty})^{-1}(Z) .
\end{equation}
Note also that $(\psi_{\infty})^{-1}(Z) \subset (\psi_{\infty})^{-1}(\Omega' \times \{ pt \})$ by construction, and $(\psi_{\infty})^{-1}(\Omega' \times \{ pt \})$ is the germ of a union of weakly special subvarieties $T'$ (i.e., $(\psi_{\infty})^{-1}(\Omega' \times \{ pt \}) = \what{T}'_s$) in the sense of \cite[Def.~4.1]{KO21}, which are algebraic by Prop.~4.8 \textit{ibid.} Since the image of $\what{T}'_s$ in $\che{D}^c_Z$ is a point, after removing a closed locus from $T$, we may assume that $\dim T' = n(I^c_Z)$. 

Now let $T^{\circ}$ be a component of $T'$ of maximal dimension which factors $j_{\infty}$. Let $\psi^{\circ}_{\infty}$ be the restriction of $\psi_{\infty}$ to the formal completion at $s$ of $T^{\circ}$. Use (\ref{atypineq}) to apply the Ax-Schanuel theorem \cite{zbMATH07066495} to the variation of Hodge structure $\restr{\mathbb{V}}{T^{\circ}}$, where we note that $\psi^{-1}_{\infty}(Z) = \psi^{\circ,-1}_{\infty}(Z)$. We learn that $(\psi^{\circ}_{\infty})^{-1}(Z)$ lies in a strict weakly special subvariety $Y \subset T^{\circ}$. If the Mumford-Tate group of $Y$ agrees with that of $T^{\circ}$ then the algebraic monodromy group of $Y$ is normal in the algebraic monodromy group of $T^{\circ}$. But the algebraic monodromy group of $Y$ is a strict subgroup of the algebraic monodromy group of $T^{\circ}$ (since $Y$ is weakly special), so this means in particular that there is some factor $D_{i} \subset \Omega'$ such that the map $\hat{Y}_{s} \to \che{D}_{i}$ induced by $\psi^{\circ}_{\infty}$ is constant. But this cannot happen because $j_{\infty}$ factors through $\hat{Y}_{s}$ and $\sigma_{\infty}$ is non-constant on each such $\che{D}_{i}$ by construction.

It follows that the Mumford-Tate group of $Y$ is strictly contained in the Mumford-Tate group of $T$. By construction this is true for all points in $E^d(\IC)$ outside of a meagre subset, hence $E^d(\IC)$ lies inside a finite union of special subvarieties of $T$. This completes the proof of the first statement.

\vspace{0.5em}



For the claim under the additional $\mathbb{Q}$-simplicity assumption and the assumption on the fibres of the period map, we observe that in this case $j = 1$ and $I_Z^c = \varnothing$. Then $n(I_Z^c) = \dim T_{\mathbb{C}}$ and the step in the above argument where one removes a closed locus from $T$ can be skipped. The proof then shows $E^{d}_{\mathbb{C}}$ lies inside a finite union of special subvarietes of $T$. 
\end{proof}

\subsection{Flatness in the Shimura Setting}

\begin{definition}
    Let $(G, X)$ be a Shimura datum of Hodge type and $\sfK \subseteq G(\IA_f)$ be a compact open subgroup small enough such that $\Sh_\sfK(G, X)$ is a scheme. Given a number field $E \subseteq \IC$ and $N \in \IN$, we say that $\sO_E[1/N]$ is an \textit{admissible base} for $(G, X, \sfK)$ if the following conditions are satisfied: 
    \begin{enumerate}[label=\upshape{(\roman*)}]
        \item $E$ contains the reflex field $E(G, X)$; 
        \item $G$ admits a reductive model $\sG$ over $\IZ[1/N]$; 
        \item $\sfK$ is of the form $\sfK^N \sfK_N$, where $\sfK^N \subseteq G(\IA^N_f)$ and $\sfK_N = \prod_{p \mid N} \sG(\IZ_p)$.
    \end{enumerate} 
\end{definition}

\noindent By convention, we always choose $\sfK$ small enough so that $\Sh_\sfK(G, X)$ is a scheme (it is a Shimura stack in general). 

 By Lovering's work \cite{Lovering}, for $E = E(G, X)$ and every $N$ satisfying (ii) and (iii) above, there exists a canonical integral model $\shS_\sfK$ over $\sO_{E}[1/N]$ such that, for every $p \nmid N$ and finite place $v \mid p$, $\shS_\sfK \tensor_{\sO_E[1/N]} \sO_{E, (v)}$ is the same as the one constructed by Kisin in \cite{Kis10}. If $E$ is strictly bigger than $E(G, X)$, then we define the canonical integral model over $\sO_E[1/N]$ to be the base change of that over $\sO_{E(G, X)}[1/N]$. Given another Shimura datum $(G', X')$ with $\sfK' \subseteq G'(\IA_f)$ such that $\sO_E[1/N]$ remains an admissible base for $(G', X', \sfK')$, if there is a morphism of Shimura data $(G', X') \to (G, X)$ that sends $\sfK'$ into $\sfK$, then the Shimura morphism $\Sh_{\sfK'}(G', X') \to \Sh_\sfK(G, X)$ extends to a morphism $\shS_{\sfK'}(G', X') \to \shS_\sfK(G, X)$ between their integral models. 

Choose a Siegel embedding $(G, X) \into (\GSp(V)_\IQ, S^\pm)$, where $V$ is a standard symplectic space and $S^\pm$ is the Siegel half spaces, such that up to enlarging $N$, $G \into \GSp(V)$ extends to $\sG \to \GSp(V)$ over $\IZ[1/N]$. By \cite[(1.3.2)]{Kis10} (cf. \cite[Prop.~4.5.3]{Lovering}), we can find a collection of tensors $\{ s_\alpha \in (V[1/N])^{\tensor} \}_{\alpha \in I}$ such that $\sG = \Aut_{s_\alpha}(V[1/N])$. Let $\sV$ denote the first de Rham cohomology of the universal abelian scheme over $\shS := \shS_\sfK(\sG, X)$. As summarized in \cite[\S4.5.5]{Lovering}, by \cite{Kis10} the tensors $s_\alpha$'s give rise to global sections $s_{\alpha, \dR} \in \sV^\tensor(\shS)$. Then one may form the $G$-torsor 
\begin{equation}
\label{eqn: def G-torsor}
\sI : T \in \mathsf{Sch}_{\shS} \mapsto \{ f \in \mathrm{Isom}(V \tensor \sO_T, \sV|_T) : f(s_\alpha \tensor 1) = s_{\alpha, \dR}|_T \}. 
\end{equation}
Let $\mu : \IG_{m, \IC} \to G_\IC$ be a Hodge cocharacter given by a point on $X$, which determines a parabolic $P_\mu$ such that the compact dual $\che{X}$ for $X$ is given by $\mathrm{Gr}_\mu = G_\IC / P_\mu$. Although one may not be able to directly find a $\mu$ which is defined over $\sO_E[1/N]$, Lovering used a Galois descent argument in \cite[\S3.2]{Lovering} to show that $\mathrm{Gr}_\mu$ has a natural model over $\sO_E[1/N]$. We keep denoting it by $\che{X}$ here. 
There is a natural period map $\alpha : \sI \to \che{X}$ defined by sending a point $(s, b^\bullet)$ to $i_{b^\bullet}(\Fil^1 \sV_s)$. In \textit{loc. cit.}, it is denoted by $\gamma : \mathscr{P}_{K^N}(G, \mathfrak{X}) \to \mathscr{GR}_\mu$. Over $\IZ[1/Nr!]$, because of \cref{trivlem}, one can upgrade this to a map $\alpha_r^d : \sI^d_r \to J^d_r \che{X}$ just as in \autoref{refinedjetthm}, by extending the isomorphism $\sV_s \simeq V_{k(s)}$ induced by $b^\bullet$ to $(\sV, \nabla)|_{\Delta^r_s} \simeq (V \tensor \sO_{\Delta^r_s}, 1 \tensor d)$.

We remark that by standard Tannakian arguments, one can show that the maps $\alpha$ and $\alpha^d_r$ do not depend on the choice of the symplectic embedding. Moreover, if $(G', X', \sfK')$ is another triple for which $\sO_E[1/N]$ is an admissible base, with the corresponding $\alpha'$ and $(\alpha')^d_r$'s, and there is a morphism $(G', X', \sfK') \to (G, X, \sfK)$, then there is a natural commutative diagram (cf. \cite[Prop.~4.4.8]{Lovering})
\[\begin{tikzcd}
	{\sI'} & {\che{X}'} & {(\sI')^d_r} & {J^d_r \che{X}'} \\
	\sI & {\che{X}} & {\sI^d_r} & {J^d_r \che{X}}
	\arrow["{\alpha'}", from=1-1, to=1-2]
	\arrow[from=1-1, to=2-1]
	\arrow[from=1-2, to=2-2]
	\arrow["{(\alpha')^d_r}", from=1-3, to=1-4]
	\arrow[from=1-3, to=2-3]
	\arrow[from=1-4, to=2-4]
	\arrow["\alpha", from=2-1, to=2-2]
	\arrow["{\alpha^d_r}", from=2-3, to=2-4]
\end{tikzcd}. \]

\begin{remark}This $\sI$ is a variant of the $\sM$ in \cref{gentorsorconstr}. The only difference is that $\sI$ is a $G$-torsor whereas $\sM$ would be a $G^{\mathrm{der}}$-torsor. 
\end{remark}

Next we assume that we have a quasi-finite morphism $\varphi : T \to \shS_\sfK$ over $\sO_{L}[1/N]$, with $L/E$ a finite extension, and where $T$ is an irreducible scheme which is smooth over over $R := \sO_{L}[1/N]$. We equip $T$ with the automorphic bundles and associated data obtained from $\shS_\sfK$ by pulling back along $\varphi$. In particular we have natural $G$ torsors $\sI(T) := \sI \times_{\shS_{\sfK}} T$ and $\sI^{d}_{r}(T) := \sI^d_r \times_{\shS_\sfK} T$ over $T$. We may compare them to the torsors $\sP, \sP^d_r, \sM$ and $\sM^d_r$ from \S\ref{refinedjetcorrsec} as follows.

First, we set $\sH = \varphi^{*} \sV$, and take for the Hodge filtration and Gauss-Manin connection on $\sH$ the analogous data pulled back from $\sV$. The fixed isomorphism $\sH_{t} \cong R^{m} = V_{R}$ from \S\ref{refinedjetcorrsec} identifies $\sI(T)$ with a subtorsor of the torsor $\sP$ constructed there. We may also observe that $\sM$ is a subtorsor of $\sI(T)$, since $\sM$ is constructed via the same procedure as $\sI(T)$ using all flat tensors instead of just the (pullbacks to $T$ of) the $\{ s_{\alpha, \dR} \}_{\alpha}$. Correspondingly, we have $\sM^d_r \subset \sI^d_r(T)$. The isomorphism $\sH_{t} = V_{R}$ also induces an identification $\che{D} \cong \che{X}$, and under this identification the map $\alpha^d_r : \sM^d_r \to J^d_r \che{D}$ from \autoref{refinedjetthm} factors through the inclusion $\sM^d_r \subset \sI^d_r$ via the pullback to $T$ of the map $\sI^d_r \to J^d_r \che{D}$ described above with the same notation.

\begin{definition}
Applying \cref{setup base}, we let $G_1, \cdots, G_j$ be the $\IQ$-simple factors of $G$, and let $\che{D} = \che{D}_{1} \times \cdots \times \che{D}_{j}$ be the usual factorization of $\che{D}$. Note that $j = k$ in this Shimura setting. Then we say that $Z \subset \che{D}$ is factor-transversal if for each non-empty subset $I \subset \{ 1, \hdots, j \}$, the intersection of $Z$ with each fibre of $\che{D} \to \prod_{i \in I} \che{D}_{i}$ is zero-dimensional over $R$. 
\end{definition}

\begin{definition}
    Let $(H, \che{D})$ be a pair consisting of a complex algebraic group $H$ and a homogeneous space $\che{D}$ for $H$. Two subvarieties $Z, Z' \subset \che{D}$ are said to be equivalent (relative to $(H, \che{D})$) if for some $h \in H(\IC)$ we have $h Z = Z'$. An equivalence class is called a \textit{type} (relative to $(H, \che{D})$). If instead $Z$ and $Z'$ are defined over some subring $R \subset \IC$, we say that $Z$ and $Z'$ are equivalent (or have the same type) if this is true after base-change to $\IC$.
\end{definition}




In the following result we wish to apply \autoref{atyplocusconstrprop} to produce a locus $E \subset T$ which controls bad intersections between $T$ and certain Shimura varieties $\shS'$ lying over $\shS$ of type $Z$. We will assume that $Z$ satisfies the expected dimension inequality
\begin{equation}
\label{Zexpdimnonneg}
\dim Z_{\IC} + \dim T_{\IC} - \dim \che{D}_{\IC} \geq 0 .
\end{equation}
It is important to choose $d$ and $\xi$ appropriately. We will take $d$ to be one more than the expected dimension of the fibre product $T_{\IC} \times_{\shS_{\IC}} \shS'_{\IC}$, i.e., we set
\[ d := \dim Z_{\IC} + \dim T_{\IC} - \dim \che{D}_{\IC} + 1 . \]
As for $\xi$, we take it large enough so that any non-degenerate complex jet $j : \mathbb{D}^d_{r,\IC} \to T_{\IC}$ satisfies the property that $\psi \circ j$ is non-constant, where $\psi$ is any $r$-limp on $T_{\IC}$ with $r \geq \xi$ defined at the base-point of $j$. Such a $\xi$ exists because $T_{\IC} \to \shS_{\IC}$ is quasi-finite by assumption, and $d > 0$ as a consequence of (\ref{Zexpdimnonneg}); in particular the non-existence of such a $\xi$ would imply that the fibres of $\varphi$ contain curves of positive dimension.

\begin{theorem}
\label{shimflatnessthmv2}
Let $(G, X)$ be a Hodge-type Shimura datum with level structure $\sfK \subseteq G(\IA_f)$. Let $\sO_E[1/N]$ be an admissible base for $(G, X, \sfK)$. Let $T$ be a $\shS := \shS_\sfK(G,X)$-scheme which is smooth over $R := \sO_{E}[1/N]$ and quasi-finite over $\shS := \shS_\sfK(G, X)$, and such that the algebraic monodromy group of $T_{\IC}$ agrees with the derived subgroup of $G$. Let $Z \subseteq \che{D} = \che{X}$ be a factor-transversal closed subscheme such that
\[ \dim Z_{\IC} + \dim T_{\IC} - \dim \che{D}_{\IC} \geq 0 . \]
Let $E \subset T$ be given by applying \autoref{atyplocusconstrprop} with $(d, \xi)$ as above, increasing $N$ as necessary. Then the following holds: 

Let $(G', X') \subseteq (G, X)$ be a sub-Shimura datum with a level structure $\sfK' \subseteq G'(\IA_f)$ such that $\sO_E[1/N]$ remains admissible for $(G', X', \sfK')$ and $\sfK' \subseteq \sfK$. If $\che{X}'$ is of type $Z$, then up to replacing $T$ by its open subscheme $T \smallsetminus E$, the fibre product $T \times_{\shS} \shS'$ is equidimensional over $R$, and there is no non-degenerate jet $j : \ID^d_r \to T \times_{\shS} \shS'$ defined over any field $k$ over $R$. 
\end{theorem}





\begin{proof}
Let $W := \shS' \times_\shS T$. By replacing $T$ with the open subscheme $T \smallsetminus E$, we may assume without loss of generality that $E = \emptyset$. We aim to show that $W$ is flat over $R$. 

Let $k$ be any algebraically closed field over $R$. For every algebro-geometric point $y \in \shS'(k)$ with image $x \in \shS$, the induced morphism $(\shS'_{W(k)})^\wedge_{y} \to (\shS_{W(k)})^\wedge_x$ between the formal completions is a closed embedding which is a local complete intersection of codimension $c = \codim_{\che{D}_\IC} Z_\IC$.\footnote{One can deduce this from \cite{Xu26}, or simply from \cite[Prop.~2.3.5]{Kis10}.} By base change, we have that if $y$ and $x$ lift to $k$-points of $W$ and $T$, then the morphism between formal completions $(W_{W(k)})^\wedge_y \to (T_{W(k)})^\wedge_x$ is also a closed immersion cut out by $c$ equations. 

By \cite[\href{https://stacks.math.columbia.edu/tag/00MG}{Tag 00MG}]{stacks-project}, to show that $W$ is flat over $R$, it suffices to show that for every pair of $k$-points $y, x$ on $W, T$ respectively, the special fiber $(W_k)^\wedge_y \into (T_k)^\wedge_x$ has codimension exactly $c$. This can only fail if for some prime $\mathfrak{p} \subset R$ (including the zero prime corresponding to the generic fiber over $L$) and $k = \overline{\kappa(\fp)}$, there exists an irreducible component $Y \subseteq W_{k}$ of dimension $d > \dim T_{k} - c$. For the sake of contradiction, assume that such a $Y$ exists. 

Since the Shimura morphism $\shS' \to \shS$ is unramified, so is $W \to T$ by base change (\cite[\href{https://stacks.math.columbia.edu/tag/02GA}{Tag 02GA}]{stacks-project}). This implies that $Y \to T_{k}$ is also unramified by restriction, so that for a general smooth point $y \in Y$, the map $Y \to T_k$ is injective on the tangent space at $y$. For any such smooth point, its local ring $\sO_{Y, y}$ is regular of dimension $d$, allowing us to construct non-degenerate $d$-dimensional jets $j : \ID^d_{r,k} \to Y \to T_k$ for any $r \ge 0$. 


By construction, these jets factor through $\shS'_{k}$, meaning their associated $r$-limps, with $r = r_{0}$, factor through $Z_{k} \subset \che{D}_{k}$ up to the action of $G_{k}$. Letting $\psi$ be an $r$-limp based at $s$, the factor-transversal assumption together with our choice of $\xi$ implies that the jet $\psi \circ j$ is non-constant when projected to each factor of $\che{D}_{k}$. Thus, these jets lift to points in the locus $\sM^d_r(Z,\xi)$. Because $d > \dim T - c$, the integer $d$ satisfies the critical dimension inequality (\ref{eqn: d > bla}). Consequently, any such jet lies above the image of $\sM^d_r(Z,\xi) \to T$, i.e. in the empty atypical locus $E$.

Therefore, no such $Y$ with dimension larger than $\dim T - c$ can exist, so $W$ is flat over $R$. 
\end{proof}

\subsection{Flatness and Lifting}

We now deduce \autoref{genericflatnessMg}, \autoref{liftisodecompthm} and \autoref{isogenyliftthm} as a consequence of \autoref{shimflatnessthmv2}.

\begin{proof}[Proof of (\ref{genericflatnessMg}):]
The statement can be proven after replacing $\sfA_{g}$ with a level cover $\sfA_{g}[n]$, and hence we may assume that $T$ is a scheme with quasi-finite map $\varphi : T \to \sfA_{g}[n]$ and prove the same statement for this map.

We let $(G, X) = (\GSp_{2g}, \mathbb{H}_{g})$ be the usual Shimura datum associated to the Siegel upper half space, and let $\che{D} = \che{X}$ be the associated sympletic grassmannian of $g$-dimensional Lagrangian subspaces. Then it is well known that, up to equivalence by the $\GSp_{2g}(\mathbb{R})$-action, there are finitely many Shimura subdata $(H, Y) \subset (G,X)$. This implies that we can find finitely many subvarieties $Z_{1}, \hdots, Z_{k} \subset \che{D}$ such that the types of the varieties $\che{Y} \subset \che{X}$ are each represented by one of the $Z_{i}$. We may assume the $Z_{i}$ are defined over $\mathcal{O}_{L'}[1/N]$ with $L'$ a finite extension, enlarging $N$ as necessary.

Now apply \autoref{shimflatnessthmv2} to this situation with $Z = Z_{i}$, and let $E = E_{1} \cup \cdots \cup E_{k}$ be the union of the resulting loci. Then $E$ is defined over $\mathcal{O}_{L'}[1/N]$ (increasing $N$ as necessary), but we may assume it is defined over $\mathcal{O}_{L}[1/N]$ after enlarging $E$ to include the conjugates of the $E_{i}$ over the Galois closure of $L'/L$.

Now the statement of \autoref{shimflatnessthmv2} implies that the fibre products $(T - E) \times_{\sfA_{g}[n]} Y$ are equidimensional over $R = \mathcal{O}_{L}[1/N]$ of expected codimension $\dim \che{D}_{\IC} - \dim Z_{i,\IC}$, where $i$ depends on $Y$. Each $Y$ is l.c.i. over $\sfA_{g}[n]$ cut out locally by $\dim \che{D}_{\IC} - \dim Z_{i,\IC}$ functions, so this implies that $(T - E) \times_{\sfA_{g}[n]} Y$ is likewise l.c.i. over $T - E$. By miracle flatness, each locus $(T - E) \times_{\sfA_{g}[n]} Y$ is also flat over $R$. 
\end{proof}

\begin{proof}[Proof of (\ref{liftisodecompthm}):]
This follows from \autoref{genericflatnessMg} by using the moduli interpretation of the Shimura varieties $w : Y \to \sfA_{g}$ together with the lifting fact \cite[Lem 2.1.5]{zbMATH08159582}.
\end{proof}

\begin{proof}[Proof of (\ref{isogenyliftthm}):]
In this case we take $(G, X) = (\GSp_{2g} \times \GSp_{2g}, \mathbb{H}_{g} \times \mathbb{H}_{g})$, with $\GSp_{2g} = \operatorname{GAut}(V,\omega)$ and $(V,\omega)$ the standard symplectic space, and write $\che{D} = \che{X}$ for the product $\mathrm{IGr}(g, V) \times \mathrm{IGr}(g, V)$ of isotropic Grassmannians. We let $Z \subset \che{D}$ be the diagonally embedded copy of $\mathrm{IGr}(g, V)$. Then $I_Z^c = \varnothing$ (in the language of \autoref{constr: atypical E}), $I^{\operatorname{nc}}_{Z} = \{ 1, 2 \}$ and $\che{D}^\nc_Z := \prod_{i \in I_Z^\nc} \che{D}_i = \che{D}$. We let 
\[ d = n(I_Z^c) - \dim \che{D}^{\operatorname{nc}}_Z + \dim_{R} Z + 1 = 2 \dim_{\IC} \sfM_{g,\IC} - \dim_{\IC} \sfA_{g,\IC} + 1 \]
as in \autoref{Zatyplocusdef}. Applying \autoref{atyplocusconstrprop} we learn that $E^{d} = E^{d}_{r_{0}}$ is not Zariski dense in $\sfM_{g} \times \sfM_{g}$.

Now let $E = E^{d}_{r_{0}}$ be as in \autoref{atyplocusconstrprop}, and fix $([C], [C'])$ and $\beta$ as in \autoref{isogenyliftthm}. We let $w : Y \to \sfA_{g} \times \sfA_{g}$ be the Shimura variety passing through $([J(C)], [J(C')])$ and whose formal completion at $([J(C)], [J(C')])$ contains the formal deformation scheme $B := \Def(J(C), J(C'), \beta) \xrightarrow{w} \sfA_{g} \times \sfA_{g}$ parameterizing deformations of the Jacobians $J(C)$ and $J(C')$ together with the isogeny $\beta$. Then because $\beta$ is prime-to-$p$, the corresponding cycle class $[\beta] \in \Hom(H^1_{\operatorname{dR}}(C'), H^1_{\operatorname{dR}}(C))$ is an isomorphism which preserves $\Fil^1$. Over $B_{\overline{\mathbb{F}}_{p}}$, the class $[\beta]$ extends to a flat section $\widetilde{\beta}$ of $\restr{\sheafhom(\sV, \sV)}{B_{\overline{\mathbb{F}}_{p}}}$ which preserves $\Fil^1$, where $(\mathcal{V}, \nabla)$ is the flat algebraic de Rham vector bundle with connection lying over $\sfA_{g}$.

Now we consider the fibre product $U = (\sfM_{g} \times \sfM_{g}) \times_{\sfA_{g} \times \sfA_{g}} B$, and write $U = \Spf R$. Suppose that $g \leq 9$, and that the point $([C], [C']) \in U(\overline{\mathbb{F}}_{p})$ does not lift to a point of $U(\mathcal{O}_{K})$, where $K$ is a finite extension of $W[1/p]$ with $W = W(\overline{\mathbb{F}}_{p})$. Then by \cite[Lem. 2.1.5]{zbMATH08159582}, this can only happen if $R$ is not flat over $W$. 

Write $B = \Spf\, S$ for a complete local ring $S$. Let $Q$ be the complete local ring of $\sfA_{g} \times \sfA_{g}$ at $([J(C)], [J(C')], \beta)$. Then $S$ is a quotient of $Q$ cut out by exactly $m = g(g+1)/2$ functions, say $y_{1}, \hdots, y_{m}$. Since $g \leq 9$ we have $m < 2(3g-3)$. If $U$ is a complete local intersection in the complete local ring of $\sfM_{g} \times \sfM_{g}$ at $([C], [C'], \beta)$, then $R$ is flat by miracle flatness. It follows that $U$ is not a complete local intersection, and so 
\[ \dim U_{\overline{\mathbb{F}}_{p}} \geq 2 \dim_{\IC} \sfM_{g,\IC} - m + 1 = d . \]
Now consider a component $N$ of $(\sfM_{g} \times \sfM_{g})_{\overline{\mathbb{F}}_{p}} \times_{(\sfA_{g} \times \sfA_{g})_{\overline{\mathbb{F}}_{p}}} Y_{\overline{\mathbb{F}}_{p}}$ that factors the map from $U_{\overline{\mathbb{F}}_{p}}$. Then $\dim N \geq d$. At a general smooth $\overline{\mathbb{F}}_{p}$-point of $N$, we can construct an $r_{0}$-limp $\psi : \mathbb{D}^{d}_{r_{0}} \to \che{D}_{\overline{\mathbb{F}}_{p}}$ which maps to $(\beta, 1) Z_{\overline{\mathbb{F}}_{p}}$. It then follows that $N$ lies over $E$, and hence so does $U$. The result follows.
\end{proof}

\section{Flatness Results (for $\sfM_{g,2}$ and $\sfH_{g,2}$)}
\label{Mg2flatsec}

In this section we give a version of the results in the previous section for the special cases of $T = \sfM_{g,2}, \sfH_{g,2}$ and $\sX = \sA_g$, or more generally for a generically finite map $\varphi : T \to \sA_{g}$ of Deligne-Mumford stacks satisfying some mild assumptions. This does not fit into the framework of the previous sections because $\sX_{\IC}$ naturally carries a mixed variation of Hodge structure rather than a pure one. Nevertheless, we will see that essentially the same arguments can be carried out, with appropriate modifications. Throughout this section we regard both $T$ and $\sX$ as Deligne-Mumford stacks.

To construct the mixed variation of Hodge structure on $\sX_{\IC}$, we may proceed as follows. Write $\sA^{\vee}_{g}$ for the total space of the dual family to $\sA_{g} \to \sfA_{g}$. The construction of $\sA^{\vee}_{g}$ as a relative Picard variety guarantees that there is a universal line bundle $\sL$ over $\sA_{g} \times_{\sfA_{g}} \sA^{\vee}_{g}$. The complement in $\sL$ of the image of the zero section is a $\mathbb{G}_{m}$-torsor $\sL^{\times}$. We can consider the family $u : \sL^{\times} \to \sA_{g}$ and define $\IV := R^{1} u^{\textrm{an}}_{\IC,*} \IZ$. Then $\IV$ is a local system equipped with a natural increasing two-step weight filtration $W$, where the non-trivial step is the image of $R^{1} \rho^{\textrm{an}}_{\IC,*} \IZ \to R^{1} u^{\textrm{an}}_{\IC,*} \IZ$, and where $\rho : \sA_{g} \times_{\sfA_{g}} \sA^{\vee}_{g} \to \sA_{g}$ and the map $\sL^{\times} \to \sA_{g} \times_{\sfA_{g}} \sA^{\vee}_{g}$ is the obvious one. The pair $(\IV, W_{\bullet})$ extends to a variation of mixed Hodge structure $(\IV, W_{\bullet}, \Fil^{\bullet})$. Because the fibres of $W_{1} \IV$ are naturally identified with the cohomology of the dual abelian variety, it will often be more natural to work instead with the dualized object $(\IV^{\vee}(-1), W_{\bullet}, \Fil^{\bullet})$ with its natural weight and Hodge filtrations. 


To adapt the constructions from \S\ref{refinedjetcorrsec} to our situation, we start by determining the algebraic monodromy group of $T_{\IC}$.

\begin{lemma}
\label{monodromygrouplem}
Suppose that $\varphi : T \to \sA_{g}$ is a complex algebraic map of Deligne-Mumford stacks such that:
\begin{itemize}
\item[-] the image of $\varphi$ contains a Hodge-generic point; and
\item[-] the image $Y$ of $T$ in $\sfA_{g}$ positive-dimensional.
\end{itemize}
Then the algebraic monodromy group of $\varphi^{*} \IV^{\vee}$ is equal to $\mathbb{G}^{2g}_{a} \rtimes \Sp_{2g}$. 
\end{lemma}

\begin{proof}
By \cite[\S2.2]{zbMATH07305885} the Mumford-Tate group of a Hodge generic point in $\sA_{g}$, and hence of a Hodge-generic point of $T(\IC)$, is $\mathbb{G}^{2g}_{a} \rtimes \GSp_{2g}$, where the semi-direct action is the natural one given by the standard representation of $\GSp_{2g}$. One easily checks that the derived subgroup of this Mumford-Tate group is $\mathbb{G}^{2g}_{a} \rtimes \Sp_{2g}$. Now by \cite[Thm. 1]{zbMATH00060857} the algebraic monodromy group of $T$ is necessarily a $\mathbb{Q}$-normal subgroup $H$ of $\mathbb{G}^{2g}_{a} \rtimes \Sp_{2g}$. The image $F$ of $H$ in $\Sp_{2g}$ is normal in $\Sp_{2g}$, and equal to the algebraic monodromy group of $Y$ in $\sfA_{g}$. Since $Y$ is positive-dimensional, the group $F$ is not the trivial group, so it follows that $F = \Sp_{2g}$. Then $H$ is a normal subgroup of $\mathbb{G}^{2g}_{a} \rtimes \Sp_{2g}$ which projects onto $\Sp_{2g}$, so necessarily equal to $\mathbb{G}^{2g}_{a} \rtimes \Sp_{2g}$. 
\end{proof}



We consider a fixed map $\varphi : T \to \sA_{g}$ of Deligne-Mumford stacks defined over $\sO_{L}[1/N]$, with $\sO_{L}$ the ring of integers of some number field $L \subset \mathbb{C}$, and with $T$ smooth over $\sO_{L}[1/N]$ and geometrically connected. We assume that $\varphi_{\IC}$ satisfies the hypotheses of \autoref{monodromygrouplem} and contains an $R := \sO_{L}[1/N]$ point $t \in T(R)$. Let $(\sH^{\vee}, \Fil^{\bullet}, \nabla)$ be the algebraic filtered bundle with connection on $T$ whose scalar extenstion to $\IC$ corresponds to $\varphi^{*} \IV^{\vee}(-1)$; the triple $(\sH^{\vee}, \Fil^{\bullet}, \nabla)$ has a natural model over $\sO_{L}[1/N]$. We carry out the constructions of \S\ref{refinedjetcorrsec} with the triple $(\sH^{\vee}, \Fil^{\bullet}, \nabla)$ and continue with the notation from that section. This means we have a $\mathbf{H}_{t}$-torsor $\pi : \sP \to T$, bundles $\sP^{d}_{r} \to T$ as well as subbundles $\pi^{d}_{r} : \sM^{d}_{r} \to T$. Everything is defined over $\sO_{L}[1/N]$, possibly increasing $N$, since replacing $L$ with a finite extension in \S\ref{refinedjetcorrsec} was only possibly necessary to choose a $\sO_{L}[1/N]$-point $t$. In what follows write $\che{D}$ for $\che{D}_{\textrm{mon}}$, which as a consequence of \autoref{monodromygrouplem} is an open orbit in the natural isotropic Grassmannian on the fibre $\sH_{t}$.

\begin{lemma}
Up to increasing $N$, the restriction of $\alpha^{d}_{r}$ from \autoref{bigjetthm} to $\sM^{d}_{r}$ lands inside $J^{d}_{r} \che{D}$. Correspondingly, $\eta^d_r$ factors through $\mathbf{H}_{t} \backslash J^{d}_{r} \che{D}$.
\end{lemma}

\begin{proof}
Since we are free to increase $N$, it suffices to show the statement when $R = \IC$. Consider the multisection of the local system $\Hom(\varphi^{*} \mathbb{V}^{\vee}, (\varphi^{*} \mathbb{V}^{\vee})_{t})$ which extends the identity $\textrm{id}_{t} \in \Hom((\varphi^{*} \mathbb{V}^{\vee})_{t}, (\varphi^{*} \mathbb{V}^{\vee})_{t})$. Because the monodromy group is contained in $\mathbf{H}_{t}$, the image $\sF$ in $\sP$ of this multisection lies inside an $\mathbf{H}_{t}$-torsor, necessarily equal to $\sM$ since $\sM$ contains $\textrm{id}_{t}$. 

Using the description of the maps $\alpha^{d}_{r}(k)$ from \S\ref{infpermapsec}, it suffices to show that for each complex analytic period map $\psi : B \to \che{L}$, with $B \subset T(\IC)$ a small complex analytic ball and where the initial condition of $\psi$ at some point $s \in T(\IC)$ corresponds to a point of $\sM(\IC)$, factors through $\che{D}$. By acting with elements $h \in \mathbf{H}_{t}(\IC)$ and replacing $\psi$ with $h \psi$, we can reduce to the case where the corresponding point from $\sM(\IC)$ lies in $\sF$. Then $\psi$ is a Hodge-theoretic period map in the usual sense, and the result follows from \cite[Lem. A.5]{zbMATH07973997}.
\end{proof}

From now on we abusively write $\alpha^{d}_{r}$ (resp. $\eta^{d}_{r}$) for maps $\sM^{d}_{r} \to J^{d}_{r} \che{D}$ (resp. $J^{d}_{r} T \to \mathbf{H}_{t} \backslash J^{d}_{r} \che{D}$). We now give two analogues of \autoref{atyplocusconstrprop}. These two analogues will control exceptional intersections of $T$ with Shimura subvarieties of $\sA_{g}$ which are ``along $\sfA_{g}$'' and ``along the fibres of $\sA_{g} \to \sfA_{g}$'', respectively. To facilitate this, we consider the vector space $\sH^{\vee}_{t}/W_{0}$, where $W_{0}$ is the subspace of $\sH^{\vee}_{t}$ corresponding to the de Rham realization at $t$ of the weight filtration on $\IV^{\vee}$, and write $\che{D}/W_{0}$ for the associated flag variety. We have a natural map $\mu : \che{D} \to \che{D}/W_{0}$. 

\subsection{The Atypical Locus along $\sfA_{g}$}
\label{alongAgatyplocsec}

Fix an algebraic family $g : \sF \to B$ over a Noetherian base of geometrically irreducible subschemes of $\che{D}$ which have constant fibre dimension $m$. That is, $\sF$ is a subscheme of $\che{D} \times B$ which is flat over $B$, and whose fibres over $B$ are geometrically irreducible of dimension $m$. We also suppose that $B$ and all the morphisms are defined over $\mathcal{O}_{L}[1/N]$.

Define $\sM^{d}_{r}(g) \subset \sM^{d}_{r} \times B$ to be the closed locus whose points $((j, p), b)$ are such that $\alpha(j, p) \in g^{-1}(b)$. Fix an integer $\xi \geq 1$. For each $r \geq \xi$, define $\sM^{d}_{r}(g)' \subset \sM^{d}_{r}(g)$ to be the open subcheme obtained by intersecting the following two open conditions:
\begin{equation}
\label{cond: nd2}
\begin{cases}
    \textrm{the locus $\sM^{d}_{r}(g)'$ lies above $J^d_{r,nd} T$} \\
    \textrm{the locus $\sM^{d}_{r}(g)'$ lies above the complement of the closed locus $\bigcup_{q \in \che{D}/W_{0}} J^{d}_{\xi}\, \mu^{-1}(q)$.} 
\end{cases}
\end{equation}
The second condition has the following interpretation: if $((j, p), b)$ is a point of $\sM^d_r(g)'$, and $\psi_{b}$ is the $r$-limp determined by $b$, then the restriction of $\psi_{b} \circ j$ to the disk $\mathbb{D}^{d}_{\xi}$ does not lie inside a fibre of $\mu$. 

\begin{proposition}
\label{mixedatyplocusconstrprop}
Suppose that $\varphi : T \to \sA_{g}$ is generically finite. Choose $d$ such that $\dim T_{\IC} - d < \dim \che{D}_{\IC} - m$. Let $E^{d}_{r}(\xi)$ be the stack-theoretic image\footnote{That is, the smallest closed substack containing the usual image.} of the map $\sM^{d}_{r}(g)' \to T$, and set 
\begin{equation}
    \label{xiEdinteq}
    E^{d}(\xi) = \bigcap_{r \geq \xi} E^{d}_{r}(\xi) .
\end{equation} Then $E^{d}(\xi)$ is not Zariski dense in $T$, and $E^{d}(\xi)(\IC)$ lies in a finite union of special subvarieties of $\sA_{g,\IC}$.
\end{proposition}

\begin{remark}
We emphasize that the base $B$ of the family $g : \sF \to B$ in the statement of \autoref{mixedatyplocusconstrprop} is finite-type, and possibly disconnected.
\end{remark}

\begin{proof}
From now on we suppress the dependence on $\xi$ and write $E^{d}, E^{d}_{r}$, etc. Define $r_{0}$ analogously to \autoref{constr: atypical E}, and define 
\[ \sJ_{r_{0}} = \bigcap_{r \geq r_{0}} \textrm{im}[\sM^{d}_{r}(g)' \to \sM^{d}_{r_{0}} \times J^{d}_{r_{0}} \che{D} \times B] . \]
Let $(p_{r_{0}}, \sigma_{r_{0}}, b)$ be a $\IC$-point of $\sJ_{r_{0}}$. Let $s \in T(\IC)$ be the image of $p_{r_{0}}$. Define $\ID^{d}_{\infty,\IC}$ as in the proof of \autoref{atyplocusconstrprop}. Applying \cite[Lem 5.3]{zbMATH07871772} to the sequence $\{ \sM^{d}_{r}(g)' \}_{r \ge r_{0}}$ at the point $(p_{r_{0}}, \sigma_{r_{0}}, b)$ and arguing as in the proof of \cite[Prop. 4.2.5]{zbMATH08159582} we obtain a commutative diagram of formal objects
\begin{equation}
\label{Fbjetdiag}
    \begin{tikzcd}
	{\ID^d_{\infty, \IC}} & {\what{T}_s} \\
	{\sF_{b}} & {\che{D}_\IC}
	\arrow["{j_\infty}", from=1-1, to=1-2]
	\arrow["{\sigma_\infty}"', from=1-1, to=2-1]
	\arrow["{\psi_\infty}", from=1-2, to=2-2]
	\arrow[hook, from=2-1, to=2-2]
\end{tikzcd} .
\end{equation}
Moreover $j_{\infty}$ is injective on tangent spaces (hence an embedding), and $\sigma_{\infty}$ does not land in any fibre of $\che{D}_{\IC} \to \che{D}_{\IC}/W_{0}$, since its restriction to $\mathbb{D}^{d}_{\xi}$ does not by construction. Also by construction, we have, formally at $s$, that
\[ \dim \psi^{-1}_{\infty}(\sF_{b}) \geq d > \dim T_{\IC} + m - \dim \che{D}_{\IC} .\]
Applying the Ax-Schanuel theorem for variations of mixed Hodge structures in the abelian setting \cite{zbMATH07305885} (cf. the more general statements for variations of mixed Hodge structures in \cite{zbMATH07973997} and \cite{zbMATH07826787}), one learns that $j_{\infty}$ factors through a strict weakly special subvariety $Y$ of $T_{\IC}$. Moreover because $\psi_{\infty} \circ j_{\infty}$ does not lie inside a fibre of $\mu$ by construction, the corresponding weakly special subvariety does not lie inside a fibre of $\sA_{g} \to \sfA_{g}$.

Using the classification of weakly special subvarieties of $\sA_{g}$ from \cite[Prop. 1.1]{zbMATH06842935}, we learn that $Y$ necessarily maps into a strict \emph{special} subvariety of $\sfA_{g}$. These are of two types: torsion sections of $\sA_{g} \to \sfA_{g}$, and special subvarieties not dominant over $\sfA_{g}$. This reasoning does not depend on the point of $\sJ_{r_0}$ chosen, so it follows that $E^{d}_{\IC}$ is contained in a finite union of special subvarieties of $\sA_{g}$, strictly contained in $\sA_{g}$.  To show that these special subvarieties pull back to strict subvarieties of $T_{\IC}$, use \autoref{monodromygrouplem} and observe that special subvarieties have algebraic monodromy group strictly contained in $\mathbb{G}^{2g}_{a} \rtimes \Sp_{2g}$.
\end{proof}



\subsection{The Atypical Locus along fibres of $\sA_{g} \to \sfA_{g}$}
\label{alongfibatyplocsec}

Fix a Noetherian scheme $B$ over $\mathcal{O}_{L}[1/N]$ and a closed subscheme $\mathcal{F} \subset \che{D} \times B$, flat over $B$. Let $g : \sF \to B$ be the associated algebraic family, and suppose it has the following properties:
\begin{itemize}
\item[-] every fibre of $g$ is contained in a fibre of $\che{D} \to \che{D}/W_{0}$;
\item[-] the fibres of $g$ are geometrically irreducible with constant fibre dimension $m$.
\end{itemize}
Define $\sM^{d}_{r}(g) \subset \sM^{d}_{r} \times B$ to be the closed locus whose points $((j, p), b)$ are such that $\alpha(j, p) \in g^{-1}(b)$. Define $\sM^{d}_{r}(g)' \subset \sM^{d}_{r}(g)$ to be the open subscheme obtained by intersecting the following two open conditions:
\begin{equation}
\begin{cases}
    \textrm{the locus $\sM^{d}_{r}(g)'$ lies above $J^d_{r,nd} T$} \\
    \textrm{the locus $\sM^{d}_{r}(g)'$ lies above $J^{d}_{1,nc} \che{D}$.} 
\end{cases}
\end{equation}
Let $e$ be the generic fibre dimension of $T \to \sfA_{g}$.
\begin{proposition}
\label{mixedatyplocusconstrprop2}
Suppose that the nonempty fibres of $T \to \sfA_{g}$ are geometrically irreducible of dimension $e$. Choose $d$ such that $e - d < g - m$. Let $E^{d}_{r}$ be the scheme-theoretic image of the map $\sM^{d}_{r}(g)' \to T$, and set $E^{d} = \bigcap_{r \geq 0} E^{d}_{r}$. Suppose that $\varphi : T \to \sA_{g}$ satisfies the assumptions of \autoref{monodromygrouplem}. Then $E^{d}$ is not Zariski dense in $T$, and $E^{d}(\IC)$ lies in the union of finitely many strict special subvarieties of $\sA_{g}$.
\end{proposition}



\begin{proof}
Again we reduce to considering the case where $R = \IC$. Define $r_{0}, \sJ_{r_0}, (p_{r_0}, \sigma_{r_0}, b)$ and $s$ as in the proof of \autoref{mixedatyplocusconstrprop}. Once again we obtain a diagram of the form (\ref{Fbjetdiag}). The key difference in this case is that the map $\sF_{b} \hookrightarrow \che{D}_{\IC}$ factors through a fibre $F$ of the map $\che{D}_{\IC} \to \che{D}_{\IC} / W_{0}$. Let $A$ be the fibre of $\sA_{g} \to \sfA_{g}$ which contains $s$. 

By restricting $\IV^{\vee}$ to $A_{\IC}$, we can consider a period map $\psi$ for this restricted VHS which is defined in some analytic neighbourhood of $s$. Then we have that 
\[ \dim (T_{\IC} \cap A_{\IC}) - \dim \psi^{-1}(\sF_{b}) \leq e - d < g - m . \]
Applying the Ax-Schanuel theorem \cite{zbMATH07305885} (cf. \cite{zbMATH07973997} and \cite{zbMATH07826787}) for the variation on $T_{\IC} \cap A_{\IC}$ we learn that $\psi^{-1}(\sF_{s})$ lies in a strict weakly special subvariety of $T_{\IC} \cap A_{\IC}$. Such subvarieties are intersections with translates of strict abelian subvarieties of $A_{\IC}$, so this can only happen if $A_{\IC}$ is not simple. Then in particular $s$ is not Hodge-generic (since its image in $\sfA_{g}$ is not), and $s$ lies in the intersection with $T$ of a special subvariety of $\sA_{g}$. There are countably many such subvarieties, and $E^{d}$ is algebraic, so in fact only finitely many such special subvarieties are required to cover $E^{d}(\IC)$. Moreover since $T$ contains at least one Hodge-generic point, it follows that $E^{d}$ is not Zariski dense in $T$.
\end{proof}

\subsection{Consequences for Flatness}
\label{flatconseqsec}

Now we give a more explicit perspective on the construction of the quadruple $(\sH, W_{\bullet}, \Fil^{\bullet}, \nabla)$. 

Let $S$ be a base scheme. Suppose that $G$ is a semi-abelian scheme over $S$, which fits into an exact sequence $0 \to T \to G \to A \to 0$ of group schemes over $S$. Let $0 \to \Omega_{A^\vee} \to E_A \to A \to 0$ be the universal vector extension of $A$ and let $0 \to \Omega_{A^\vee} \to \wt{E}_A \to G \to 0$ be its pullback to $G$. Then $\wt{E}_A$ is the universal vector extension of $G$ \cite[p.82]{FC90}. Recall that $\underline{\H}^1_\dR(A) = \Lie(E_A)^\vee$. 

\begin{definition}
\label{drcohomdef}
    The de Rham realization of $G$ is $\underline{\H}^1_\dR(G) := \Lie(\wt{E}_{A})^\vee$. Note that this is a locally free sheaf of rank $2 \dim A + \mathrm{rank\,} T$ which fits into an exact sequence 
    \begin{equation}
        0 \to \underline{\H}^1_\dR(A) \to \underline{\H}^1_\dR(G) \to \Lie(T)^\vee \to 0.
    \end{equation}
    We equip $\underline{\H}^1_\dR(G)$ with an ascending weight filtration $W_0 = 0, W_{1} = \underline{\H}^1_\dR(A)$ and $W_2 = \underline{\H}^1_\dR(G)$, and a descending Hodge filtration $F^0 = \underline{\H}^1_\dR(G), F^1 =$ annihilator of $\Omega_{A^{\vee}}$, and $F^2 =0$. 
\end{definition}

\begin{remark}
    In \S\ref{subsec: deform log 1-motif} we will discuss 1-motifs. If $Q$ is the 1-motif defined by $[0 \to G]$, then $\underline{\H}^1_\dR(Q) = \underline{\H}^1_\dR(G)$ except that $W_{i + 1} \underline{\H}^1_\dR(Q) = W_{i} \underline{\H}^1_\dR(G)$. If $S$ is a separated scheme over a base $S_0$, then $\underline{\H}^1_\dR(G)$ is moreover equipped with a Gauss-Manin connection $\nabla_{S/S_0} : \underline{\H}^1_\dR(G) \to \underline{\H}^1_\dR(G) \tensor \Omega_{S/S_0}$ (cf. \cite[\S1.1.3]{Mad19}). 
\end{remark}


\paragraph{The case of $\sfM_{g,2}$}

Let $\varphi : \sfM_{g,2} \to \sA_{g}$ be the natural map. 

\begin{lemma}
\label{immersivelem}
Away from the fibre above the hyperelliptic locus of $\sfM_{g}$, and the diagonal locus
\[ \{ ([C], x, y) \in \sfM_{g,2} : x = y \} \subset \sfM_{g,2} , \]
the map $\varphi$ is injective on tangent spaces.
\end{lemma}

\begin{proof}
It is well known that the map $\sfM_{g} \to \sfA_{g}$ is immersive away from the hyperelliptic locus. Fixing a point $[C] \in \sfM_{g}$ away from the hyperelliptic locus, it suffices to show that the map $\sfM_{g,2,[C]} \to \sA_{g,[J(C)]}$ is injective on tangent spaces away from the diagonal locus where $x = y$. At the point $([C], x, y)$, the map on tangent spaces is given by
\[ T_{(x,y)} C^2 \ni (v_{x}, v_{y}) \mapsto ([\omega] \mapsto v_{x}(\omega) - v_{y}(\omega)) \in \H^0(C, \w_C)^\vee = T_{[x - y]} J(C) . \]
Away from the hyperelliptic locus, it is well known that the map $C \to \mathbb{P}^{g-1}$ given by $x \mapsto [v_{x}] \in \mathbb{P} H^{0}(C, \omega_{C})^{\vee}$ is an embedding, so the above map can only fail to be injective when $x = y$.
\end{proof}

We let $\sfM^{\circ}_{g,2} \subset \sfM_{g,2}$ denote the open locus where $\varphi$ is injective on tangent spaces. 



\paragraph{Convention on Hecke translates.} Strictly speaking, at a fixed finite level a Hecke operator is a correspondence rather than a morphism. We therefore make the following convention. After passing to a sufficiently small auxiliary level $\sfK$, for $\gamma \in \GSp_{2g}(\IA_f^p)$ we set $\sfK_{\gamma} := \sfK \cap \gamma \sfK \gamma^{-1}$ and consider the Hecke correspondence
$$\sA_{g,\sfK} \xleftarrow{p_1} \sA_{g,\sfK_{\gamma}} \xrightarrow{p_2} \sA_{g,\sfK}.$$
If $w : Y \to \sA_{g,\sfK}$ is a morphism, we define $Y_{\gamma} := Y \times_{\sA_{g,\sfK},p_1} \sA_{g,\sfK_{\gamma}}$
and write $w_{\gamma} : Y_{\gamma} \to \sA_{g,\sfK}$ for the morphism induced by $p_2$. We refer to $w_{\gamma} : Y_{\gamma} \to \sA_{g,\sfK}$ as the \emph{Hecke translate} of $w : Y \to \sA_{g,\sfK}$. When no confusion can arise, we suppress the auxiliary level structures from the notation and abusively write $T_{\gamma}(Y)$ for the morphism $w_{\gamma} : Y_{\gamma} \to \sA_g$. This convention is harmless for our purposes, since away from the primes at which the Hecke correspondence ramifies, the maps $p_1$ and $p_2$ are finite \'etale.

\begin{proposition}
\label{prop: main prop mixed}
There exists a strict substack $E \subset \sfM_{g,2}$ and an integer $N$, such that 
\begin{itemize}
\item[-] $E_{\IC} \cap \sfM^{\circ}_{g,2,\mathbb{C}}$ is a finite union of strict special subvarieties of $\sfM^{\circ}_{g,2,\IC}$ for the restriction of the mixed VHS $\IV^{\vee}(-1)$; and
\item[-] Suppose $0 \leq k \leq g$ is an integer such that
\begin{equation}
\label{expecteddimassump}
\dim \sfM_{g,2,\IC} + \dim (\sA_{g-k} \times \sfA_{k})_{\IC} \geq \dim \sA_{g,\IC} + 1 .
\end{equation}
Let $w : Y := \sA_{g-k} \times \sfA_k \to \sA_g$ be the natural morphism. Then if $\gamma \in \GSp_{2g}(\IA_f^p)$ and $w_\gamma : Y_\gamma \to \sA_g$ is the Hecke translate of $w$ in the above sense, the fibre product $(\sfM_{g,2} \setminus E) \times_{\sA_g,\varphi,w_\gamma} Y_\gamma$ is flat over $\IZ[1/N,1/\operatorname{Ram}(T_\gamma)]$ and has the expected codimension, where $\operatorname{Ram}(T_\gamma)$ is the set of primes where the Hecke correspondence ramifies.
\end{itemize}
\end{proposition}

\begin{proof}
For $g = 2$ the map $\sfM_{2,2} \to \sA_{2}$ is dominant, so the result follows from basic facts about $\sA_{2}$. We therefore assume $g \geq 3$, and so that not all curves are hyperelliptic. It suffices to prove the statement for a fixed $k$, and after replacing $\sfM_{g,2}$ with $T := \sfM^{\circ}_{g,2}$. We will then construct $E \subset T$ as a union $E_{\rightarrow} \cup E_{\uparrow}$ of atypical loci which are ``horizontal'' (along $\sfA_{g}$) and ``vertical'' (along the fibres of $\sA_{g} \to \sfA_{g}$). This we will do via applications of \autoref{mixedatyplocusconstrprop} and \autoref{mixedatyplocusconstrprop2}, respectively. 


We recall the description of the algebraic variety $\che{D}$ in \cite[\S2.5]{zbMATH07305885}, where it is called $\sX^{\vee}$. From \cite[Ex. 2.9]{zbMATH07305885} and the fact that the variety $\sX^{\vee}$ is a quotient of the group $P(\IC)$ in \cite[\S2.5]{zbMATH07305885}, one deduces that $\che{D}$ is a rank $g$ complex vector bundle over $\che{D}/W_0$, and that $\che{D}/W_0$ is the usual symplectic Grassmannian that compactifies the Siegel upper half space. Now observe that $\che{D}$ sits naturally inside the Grassmannian of $g$-dimensional Lagrangian subspaces of the fixed fibre $\sH^{\vee}_{t}$. Since $\sH^{\vee}_{t}$ has rank $2g+1$, this space has dimension $g(g+1)/2 + g$, hence contains $\che{D}$ as a Zariski open subset. 

To apply \autoref{mixedatyplocusconstrprop2} we construct a family $g : \sF \to B$ of algebraic subvarieties of $\che{D}$ as follows. The base $B$ we take to be $\che{D}/W_{0} \times \operatorname{Gr}(\sH^{\vee}_{t}, 2k)$. Given a point $(F, \alpha) \in B$, where $F = \Fil^1 + W_{0}$, the fibre $g^{-1}(F, \alpha)$ is the locus of all choices of $\Fil^1$ such that $\Fil^1 + W_{0} = F$ and $\dim \Fil^1 \cap \alpha = k$. We write $E_{\uparrow}$ for the locus from \autoref{mixedatyplocusconstrprop2} associated to $g$, and set $E = E_{\rightarrow} \cup E_{\uparrow}$; we choose $d$ to be
\[ d := \dim \sfM_{g,2,\IC} + \dim (\sA_{g-k} \times \sfA_{k})_{\IC} - \dim \sA_{g,\IC} + 1 . \]

To apply \autoref{mixedatyplocusconstrprop} we construct a family $g : \sF \to B$ of algebraic subschemes of $\che{D}$ over $R = \mathbb{Z}[1/N]$ that contains the algebraic varieties corresponding to all the Hecke translates $w_\gamma : Y_\gamma \to \sA_g$ in the statement. For the base $B$ we take $\Sp_{2g} \times \operatorname{Gr}(\sH^{\vee}_{t}, 2k)$, where the second factor parametrizes subspaces of $\sH^{\vee}_{t}$ of dimension $2k$. Let $\che{E} \subset \che{D}/W_{0}$ be the flag subscheme corresponding to $\sfA_{g-k} \times \sfA_{k}$. Then the fibres of $g$ are of the form
\[ g^{-1}(h, \alpha) = \{ \Fil^1 \in \che{D} : \Fil^1 + W_0 \in h \che{E}, \hspace{1em} \dim \Fil^1 \cap \alpha = k \} . \]
After removing a constructible locus from the base $B$ and reparameterizing, we may assume that the fibres of $g$ all have the same dimension as $Y_\gamma$. The descending chain $E^d_r$ from \autoref{mixedatyplocusconstrprop2} stabilizes, so we choose $\xi$ such that $E_{\uparrow}=E^d_{\xi}$. We then use this $\xi$ in  \autoref{mixedatyplocusconstrprop} and write $E_{\rightarrow}$ for the locus associated to the present family $g$ and this value of $\xi$; we choose the same $d$ as before. Note that the intersection in (\ref{xiEdinteq}) stabilizes at some $r_{0} \geq \xi$ which can be larger than $\xi$.

Now consider the fibre product
\begin{equation}
    \label{Stkeq}
    S_{\gamma,k} := T \times_{\sA_g,\varphi,w_\gamma} Y_\gamma .
\end{equation}
Suppose that $S_{\gamma,k}$ is non-empty and contains a point $s$ lying over $\mathbb{Z}[1/N,1/\operatorname{Ram}(T_\gamma)]$ at which $S_{\gamma,k}$ is not flat, and write $x \in T$ and $y \in Y_\gamma$ for the images of $s$. It suffices to assume that $s$ is a closed point whose residue characteristic $p$ does not divide $N$ or $\operatorname{Ram}(T_\gamma)$, and that $p > r_{0}$. Away from these primes $p_1$ and $p_2$ are \'etale, so the morphism $w_\gamma : Y_\gamma \to \sA_g$ is a local complete intersection of the expected codimension. Hence the lack of flatness implies, by miracle flatness, that $S_{\gamma,k}$ has bigger than expected dimension modulo $p$ at $s$, i.e.
\begin{align*}
\dim_x T_{\kappa(s)} - \dim_s S_{\gamma,k,\kappa(s)} &< \dim_{w_\gamma(y)} \sA_{g,\kappa(s)} - \dim_y (Y_\gamma)_{\kappa(s)} .
\end{align*}
From our assumption (\ref{expecteddimassump}), one deduces that $\dim_{s} S_{\gamma,k,\kappa(s)} \geq d \geq 2$. Let $C$ be a component of $S_{\gamma,k,\kappa(s)}$ passing through $s$ satisfying $\dim C \geq d$. There are now two options.

Let $f : C^{\mathrm{red}} \to \sfA_g$ be the composite of $C^{\mathrm{red}} \to \sA_g$ with the projection $\sA_g \to \sfA_g$. Suppose first that, at a general smooth point $s \in C^{\mathrm{red}}$, there exists a non-degenerate $d$-jet $j_{\xi} : \mathbb{D}^d_{\xi,\kappa(s)} \to C^{\mathrm{red}}$ such that $J^d_{\xi}f(j_{\xi})$ is non-constant. Fix any $r \geq r_{0}$, with $r < p$, and extend $j_{\xi}$ to a non-degenerate $r$-jet $j : \mathbb{D}^d_{r,\kappa(s)} \to C^{\mathrm{red}}$. We take $p$ to be any point of the torsor $\sM$ above $s$. Then $\alpha(j,p)$ is a non-degenerate jet in $J^d_r \che{D}_{\kappa(s)}$ whose image in $J^d_r (\che{D}/W_{0})_{\kappa(s)}$ is non-constant at order $\xi$. To show that $s$ lies in $E^d_r(\xi)$, it therefore suffices to construct $b$ such that $\alpha(j,p) \in J^{d}_{r} g^{-1}(b)$.

To find such a $b$ it suffices to show that the data $(j^{*} \sH^{\vee}, j^{*} W_{\bullet}, j^{*} \Fil^{\bullet}, j^{*} \nabla)$ admits:
\begin{itemize}
    \item a flat idempotent $c : j^{*} \sH^{\vee}/W_{0} \to j^{*} \sH^{\vee}/W_{0}$ which splits the filtered vector bundle $(j^{*} \sH^{\vee}/W_{0}, \Fil^{\bullet})$ as a direct sum of summands $\sH^{\vee}_{1} \oplus \sH^{\vee}_{2}$ of rank $2k$ and $2(g-k)$, compatibly with the polarizing form; and
    \item a flat subbundle $\alpha$ of $j^{*} \sH^{\vee}$ of rank $2k$ such that $\Fil^1 \cap \alpha$ is a vector subbundle of rank $k$;
\end{itemize}
and such that the locus of flags $F^{\bullet}$ in $\che{D}_{\kappa(s)}$ corresponding to $(c_{s}, \alpha_{s})$ under the map $p$ has the correct dimension. 

Recall that we have a universal semiabelian family $u$ over $\sA_g$. The pullback $w_\gamma^{*}u$ to $Y_\gamma$ is prime-to-$p$ isogenous to a semiabelian family of the form $G \times A'$, where $G$ is semiabelian of toric rank one and $A'$ is abelian of dimension $k$. In particular, pulling back further along the morphism $j : \mathbb{D}^{d}_{r,\kappa(s)} \to C \to S_{\gamma,k} \to Y_\gamma$, we obtain a split semiabelian family, prime-to-$p$-isogenous to $j^{*} \varphi^{*}u$, which is of the form $u' : G_s \times A'_s \times \mathbb{D}^{d}_{r,\kappa(s)} \to \mathbb{D}^{d}_{r,\kappa(s)}$. The idempotent $c$ comes from the splitting up to prime-to-$p$ isogeny of the abelian part of the family $u'$. The subbundle $\alpha$ on the other hand is the kernel of 
\[ H^{1}_{\operatorname{dR}}(G_{s} \times A'_{s} \times \mathbb{D}^{d}_{r,\kappa(s)}/\mathbb{D}^{d}_{r,\kappa(s)})^{\vee} \to H^{1}_{\operatorname{dR}}(G_{s} \times \mathbb{D}^{d}_{r,\kappa(s)} / \mathbb{D}^{d}_{r,\kappa(s)})^{\vee} \]
where the universal vector extensions are as in \autoref{drcohomdef}. To check the dimension property for the fibre at $s$, we argue as follows. The condition that $\Fil^1_{s}/W_{0}$ is decomposes as a direct sum into two symplectic subspaces, which are $k$ and $(g-k)$-dimensional subspaces of the images of $c$ and $1-c$, defines a subvariety of the symplectic Grassmannian of $\sH^{\vee}_{t}/W_{0}$ of the same dimension as $\sfA_{g-k} \times \sfA_{k}$. The fibre of $\che{D}$ above $\Fil^1_{s} / W_{0}$ has dimension $g$, and its points correspond to a choice of $g$-dimensional subspace of the $g+1$-dimensional vector space $\Fil^1_{s} + W_{0}$. If such a choice intersects the kernel $K$ of $H^{1}_{\operatorname{dR}}(G_{s} \times A'_{s})^{\vee} \to H^{1}_{\operatorname{dR}}(G_{s})^{\vee}$ in a subspace of dimension $k$ then it is uniquely determined by its image in $H^{1}_{\operatorname{dR}}(G_{s})^{\vee}$. The remaining freedom in choosing this image is a variety of dimension $g-k$. We conclude that $s \in E^d_r(\xi)$. Since we chose $r \geq r_{0}$, $s \in E_{\rightarrow}$. This holds at a general point of $C$, and $E_{\rightarrow}$ is closed, so $C \subset E_{\rightarrow}$; in particular our original point lies in $E_{\rightarrow}$.

Suppose instead that at a general smooth point $s \in C^{\mathrm{red}}$ every non-degenerate $d$-jet $j_{\xi} : \mathbb{D}^d_{\xi,\kappa(s)} \to C^{\mathrm{red}}$ has constant image in $\sfA_g$ to order $\xi$. Choose such a jet and a point $p$ of the torsor $\sM$ above $s$. Since $T=\sfM^\circ_{g,2}$ and $j_{\xi}$ is non-degenerate, its full period $1$-jet is non-constant. The same splitting data coming from the Hecke source places $\alpha(j_{\xi},p)$ in a fibre of the family used to define the loci $E^d_r$ in \autoref{mixedatyplocusconstrprop2}. Hence $s \in E^d_{\xi}=E_{\uparrow}$, and therefore $C \subset E_{\uparrow}$. Notice that this argument does not require $C$ itself to lie in a fibre of $\sA_g \to \sfA_g$; it only uses that its $d$-jets of order $\xi$ have constant image in $\sfA_g$. The numerical hypothesis of \autoref{mixedatyplocusconstrprop2} is the same as before: here $m \leq g-1$ and $e=2$, so it is enough that $d>1$, which holds for an unexpected component under \eqref{expecteddimassump}.

We have shown that all unexpected components of the loci $S_{\gamma,k,\kappa(s)}$ lie in the locus $E = E_{\rightarrow} \cup E_{\uparrow}$. Flatness away from $E$ now follows as usual from miracle flatness, as does the expected dimension condition (by construction).
\end{proof}


The above proposition considers isogeny decompositions with two factor case, but the proof extends verbatim to the multi-factor case up to more book-keeping. In particular, suppose that $\bg$ is a partition of $g$. Then \cref{prop: main prop mixed} holds when $\sA_{g - k} \times \sfA_k$ is replaced by $\sfA_{g_1} \times \cdots \times \sfA_{g_{k - 1}} \times \sA_{g_k}$. In our applications, we will make use of the following special case. 

\begin{corollary}
\label{cor: intersect with torsion}
Let $\bg$ be a partition of $g$ satisfying
\begin{equation}
\label{atypintineqMg2}
\delta(\mathbf{g}) + (g-g_{k}) \leq \dim \sfM_{g,2,\IC} - 1 ,
\end{equation}
and let $T_j$ be an \'etale scheme over $\sfA_{g_j}$. Suppose $T := \prod_j T_j$ is equipped with a morphism $T \to \sfA_g$ and that
$$\psi : \prod_{j = 1}^k \sA_{g_j}|_{T_j} \to \sA_{g}|_T$$
is a prime-to-$p$ symplectic similitude. Assume $n \in \sO_T^\times$ and the image of $T$ is disjoint from the image of $E$ in $\sA_g$. Then the fiber product$$\sfM_{g, 2} \times_{\sA_g} \Big[ \big( \prod_{j = 1}^{k - 1} \sA_{g_j}|_{T_j} \big)[n] \times \sA_{g_k} \Big]$$ is flat over $\IZ$ and has the expected codimension $\delta(\bg) + (g - g_k)$ in $\sfM_{g, 2}$. 
\end{corollary}

\paragraph{The case of $\sfH_{g,2}$}

 We recall that by $\sfH_{g,2}$ we mean the moduli space of triples $(C, p, \iota(p))$, where $C$ is a hyperelliptic curve, $p \in C$ is a point, and $\iota : C \to C$ is the hyperelliptic involution. Let $\varphi : \sfH_{g,2} \to \sA_{g}$ be the natural map which sends $(C, p, \iota(p))$ to $(J(C), \mathcal{O}_{C}(p - \iota(p)))$. 

\begin{lemma}
\label{Hinjlem}
For any point $q = ([C], p, \iota(p)) \in \sfH_{g,2}$ defined over a field of characteristic $\neq 2$, the induced map on tangent spaces $T_{q} \sfH_{g,2} \to T_{\varphi(q)} \sA_{g}$ is injective.
\end{lemma}

\begin{proof}
Let $\pi : \sfH_{g,2} \to \sfH_{g}$ be the obvious projection. By \cite[Thm. 1.2]{zbMATH07346761} the maps $T_{\pi(q)} \sfH_{g} \to T_{\tau(\pi(q))} \sfA_{g}$ are injective, where $\tau : \sfH_{g} \to \sfA_{g}$ is the usual map. It therefore suffices to pass to the fibres above $\pi(q)$ and $\tau(\pi(q))$ and verify that the maps $T_{(p, \iota(p))} C^{\iota} \to T_{p - \iota(p)} J(C)$ are injective, where $C^{\iota} \subset C^2$ is the locus of points $(x, y) \in C^2$ such that $y = \iota(x)$.

The map on tangent spaces is then
\[ T_{(p, \iota(p))} C^{\iota} \ni (v_{p}, \iota(v_{p})) \mapsto ([\omega] \mapsto v_{p}(\omega) - v_{p}(\iota^{*} \omega)) \in H^{0}(C, \omega_{C})^{\vee} = T_{[p - \iota(p)]} J(C) . \]
This map is injective provided that $\omega - \iota^{*} \omega$ does not vanish at $p$ for all $\omega$. But $\iota$ acts on the standard hyperelliptic forms $x^{i} dx / y$ by multiplication by $(-1)$, so $\omega - \iota^{*} \omega = 2 \omega$ which is non-zero in characteristic not equal to $2$.
\end{proof}

We now give an analogue of \autoref{prop: main prop mixed}

\begin{proposition}
\label{prop: main prop mixed2}
There exists a strict subscheme $E \subset \sfH_{g,2}$ and an integer $N$, such that 
\begin{itemize}
\item[-] $E_{\IC} \cap \sfH_{g,2,\mathbb{C}}$ is a finite union of strict special subvarieties of $\sfH_{g,2,\IC}$ for the restriction of the mixed VHS $\IV^{\vee}(-1)$; and
\item[-] Suppose $0 \leq k \leq g$ is an integer such that

\begin{equation}
\dim \sfH_{g,2,\IC} + \dim (\sA_{g-k} \times \sfA_{k})_{\IC} \geq \dim \sA_{g,\IC} .
\end{equation}
Let $w : Y := \sA_{g-k} \times \sfA_k \to \sA_g$ be the natural morphism. Then if $\gamma \in \GSp_{2g}(\IA_f^p)$ and $w_\gamma : Y_\gamma \to \sA_g$ is the Hecke translate of $w$ in the above sense, the fibre product $(\sfH_{g,2} \setminus E) \times_{\sA_g,\varphi,w_\gamma} Y_\gamma$ is flat over $\IZ[1/N,1/\operatorname{Ram}(T_\gamma)]$ and has the expected codimension, where $\operatorname{Ram}(T_\gamma)$ is the set of primes where the Hecke correspondence ramifies.
\end{itemize}
\end{proposition}

\begin{proof}
The argument is essentially identical to \autoref{prop: main prop mixed}, with a few minor changes, which we describe. In this case one takes
\[ d := \dim \sfH_{g,2,\IC} + \dim (\mathcal{A}_{g-k} \times \sfA_{k})_{\IC} - \dim \sA_{g,\IC} + 1 \]
when constructing atypical loci. The variety $\che{D}$ is chosen as before, as is the family $\sF \to B$. The locus $E = E_{\rightarrow} \cup E_{\uparrow}$ is constructed as before, with appropriate integers $\xi$ and $r_{0}$. We set $T = \sfH_{g,2}$.

Again one reduces to considering the loci $S_{\gamma, k}$ defined by (\ref{Stkeq}). In this situation, where $S_{\gamma, k}$ is not flat at a point $s$, we obtain that $\dim_{s} S_{\gamma,k,\kappa(s)} \geq d \geq 1$. We suppose that the residue characteristic of $s$ is $p > r_{0}$. Fixing a component $C$ of $S_{\gamma,k,\kappa(s)}$ of dimension $\geq d$, one again considers two cases. The first case occurs when all non-degenerate $\xi$-order $d$-dimensional jets factoring through $C$ in fact factor through a fibre of $\sfH_{g,2} \to \sfH_{g}$ (this can only happen in $d = 1$). In this case $C$ lies inside $E_{\uparrow}$. In the other case, a general $r$-order jet factoring through $C$, with $r \geq r_{0}$, is non-constant at order $\xi$ in $\sfH_{g}$. In this case one argues as before to conclude that $C \subset E_{\rightarrow}$.

When arguing that a non-degenerate jet mapping to $\sfH_{g,2}$ remains non-degenerate after taking its image in $\sA_{g}$, one uses \autoref{Hinjlem} instead of \autoref{immersivelem}. Everything else proceeds as before. 
\end{proof}

Analogously to \autoref{cor: intersect with torsion}, we have the following corollary of \autoref{prop: main prop mixed2}.

\begin{corollary}
\label{cor: intersect with torsion2}
Let $\bg$ be a partition of $g$ satisfying
\begin{equation}
\delta(\mathbf{g}) + (g-g_{k}) \leq \dim \sfH_{g,2,\IC} - 1 ,
\end{equation}
and let $T_j$ be an \'etale scheme over $\sfA_{g_j}$. Suppose $T := \prod_j T_j$ is equipped with a morphism $T \to \sfA_g$ and that
$$\psi : \prod_{j = 1}^k \sA_{g_j}|_{T_j} \to \sA_{g}|_T$$
is a prime-to-$p$ symplectic similitude, where $p$ is some prime. Assume $n \in \sO_T^\times$ and the image of $T$ is disjoint from the image of $E$ in $\sA_g$. Then the fiber product$$\sfH_{g, 2} \times_{\sA_g} \Big[ \big( \prod_{j = 1}^{k - 1} \sA_{g_j}|_{T_j} \big)[n] \times \sA_{g_k} \Big]$$ is flat over $\IZ$ and has the expected codimension $\delta(\bg) + (g - g_k)$ in $\sfH_{g, 2}$. 
\end{corollary}

\section{Proof of \cref{thm: main density}}


\subsection{Applications of Weil II}
\label{weilIIest}




In unlikely intersection theory, one generally observes that special intersections of expected dimension are dense if and only if they exist. In the Hodge-theoretic setting, such a principle is coined an ``all or nothing'' principle in \cite[\S10]{BKU}, where a general Hodge-theoretic version of this principle is proven, following an argument of Chai \cite{chai1998density}. With hindsight, the density criterion of Green (see for instance \cite[Prop. 3.3]{zbMATH05177030}) is of a similar form, in that it reduces a global density statement to the verification of a local property at one point.

In the setting of intersections inside algebraic groups, or more generally homogeneous spaces $X$ under an algebraic group $G$, a similar statement exists for intersections of arbitrary subvarieties, known as Kleiman's theorem \cite{zbMATH03451984}. Kleiman's Theorem says that given two irreducible subvarieties $V_{1}, V_{2} \subset X$, the locus of $\gamma \in G$ such that the intersection $V_{1} \cap \gamma V_{2}$ has the expected dimension is Zariski open in $G$, hence such intersections are dense in $V_{1}$ if and only if they exist. If $X = G = A$ is an abelian variety one can apply this with $V_{2}$ special (i.e., a torsion translate of an abelian subvariety) to get an all-or-nothing principle for intersections with special translates of $V_{2}$, using that torsion points are dense in $A$.

The situation when $A$ is defined over a finite field $\IF_q$ is often more subtle, since every $\overline{\IF}_{q}$-point is torsion, and therefore torsion points are dense in any subvariety $X \subset A$. This is not true if one restricts to just those torsion points of prime-to-$p$ order, since $X$ could lie in a coset $B + x$ of a strict abelian subvariety $B \subsetneq A$ such that $x$ maps to a point of $A/B$ with order divisible by $p$. We call such cosets $B + x$ \textit{$p$-special} subvarieties. Curiously, one still has the following ``all or nothing'' principle for the density of prime-to-$p$ torsion points.

\begin{theorem}
\label{thm: finite MM}
    Let $k = \IF_q$ be a finite field with $q$ a power of the prime $p$ and $X \subseteq A$ be an irreducible subvariety. Then the following are equivalent: 
    \begin{enumerate}[label= \upshape{(\alph*)}]
        \item The set of closed points on $X$ of prime-to-$p$ order is Zariski dense. 
        \item The set of closed points on $X$ of prime-to-$p$ order is non-empty. 
        \item $X$ is not contained in a $p$-special subvariety. 
    \end{enumerate} 
\end{theorem}

We start with a presumably standard fact, for which we provide a proof for lack of a direct reference. Below we write $v_p(-)$ for $p$-adic valuation. 


\begin{proposition}
    Let $A$ be an abelian variety over $\IF_q$ of $p$-rank $f$. Then up to replacing $q$ by a finite power, we have $|A[p^\infty](\IF_{q^m})| = O(m^f)$. 
\end{proposition}
\begin{proof}
Let $\pi$ be the $q$-power Frobenius endomorphism acting on the abelian variety $A$. Then we have 

\begin{equation}
    \label{eqn: counting on AV}|A(\mathbb{F}_{q^m})| = \prod_{i=1}^{2g} (1 - \alpha_i^m)
\end{equation}
where $\alpha_1, \dots, \alpha_{2g} \in \overline{\IQ}_p$ are the Weil roots. Taking $p$-adic valuation, we get 
$$ \log_p(|A[p^\infty](\IF_{q^m})|) = v_p(|A(\IF_{q^m})|) = \sum_{i=1}^{2g} v_p(1 - \alpha_i^m).$$
Because $A$ has $p$-rank $f$, up to reordering, we may assume that $v_p(\alpha_i) = 0$ for $i \le f$, and $v_p(\alpha_i) > 0$ for $i > f$. For $i > f$, $v_p(1 - \alpha_i^m) = 0$ for all $m$. 

For $\alpha_1, \cdots, \alpha_f$, since the residue field of any finite extension of $\IQ_p$ is finite, these roots must be roots of unity modulo the maximal ideal $\mathfrak{p}$. Thus, there exists some $d \gg 0$ such that for all $i \le f$, we have $\alpha_i^d = 1 + \gamma_i$
where $c_i := v_p(\gamma_i) > 1/(p-1)$ for $\gamma_i = \alpha_i^d - 1$. Then we may replace $q$ by $q^d$ and then assume that $d = 1$. 

Now, note that 
$$v_p(1 - \alpha_i^m) = v_p(1 - (1 + \gamma_i)^{m})$$
Using that $c_i > 1/(p - 1)$, one deduces that in the binomial expansion 
$$ (1 + \gamma_i)^m - 1 = \sum_{j=1}^m \binom{m}{j} \gamma_i^j  $$
the $j = 1$ term, i.e., $m \gamma_j$ has the minimal $v_p$. Therefore,
$$v_p(1 - \alpha_i^m) = v_p(m \gamma_i) = v_p(m) + v_p(\gamma_i) = v_p(m) + c_i. $$
Set $C = \sum_{i \le f} c_i$. Note that it is independent of $m$. We now have 
$$ v_p(|A[p^\infty](\IF_{q^m})|) = \sum_{i \le f} (v_p(m) + c_i) = f \cdot v_p(m) + C \Rightarrow |A[p^\infty](\IF_{q^m})| = p^{f v_p(m) + C} = O(m^f). $$
Indeed, trivially we have $p^{v_p(m)} \le m$. 
\end{proof}

The main lemma for \cref{thm: finite MM} is the following.

\begin{lemma}
Let $A$ be an abelian variety over $\IF_q$ and let $A_{p'} := \{ x \in \bigcup_m A(\IF_{q^m}) : n x = 0 \textit{ for some } n \in \IN \smallsetminus p \IN \}$. Let $X$ be a geometrically irreducible smooth proper curve over $\IF_q$ and a morphism $\psi : X \to A$ whose image generates $A$. Then up to replacing $q$ by a finite power, as $m \to \infty$, the growth rate of $X(\IF_{q^m}) \cap A_{p'}$ is asymptotically:
$$|X(\IF_{q^m}) \cap \psi^{-1}(A_{p'})| = \frac{C q^m}{N_m} + O(q^{m/2}),$$
where $G_m := A[p^\infty](\IF_{q^m})$, $N_m := |G_m|$, and $C$ is a constant. In particular, there are infinitely many $x \in X(\overline{\IF}_q)$ such that $\psi(x)$ is of prime-to-$p$ order. 
\end{lemma}
The constant $C$ is explicitly constructed in the proof. 
\begin{proof}
The group of rational points $A(\IF_{q^m})$ decomposes uniquely into a direct sum of its $p$-part and prime-to-$p$ component. The indicator function for a point $y \in A(\IF_{q^m})$ having a trivial $p$-part (and thus belonging to $A_{p'}$) is given by the orthogonality relations of the characters of $G_m$:
$$I(y) = \frac{1}{N_m} \sum_{\chi \in \widehat{G}_m} \chi(y)$$
By summing this indicator function over all rational points, we obtain 
\begin{equation}
\label{ptcountexp}
|X(\IF_{q^m}) \cap \psi^{-1}(A_{p'})| = \sum_{x \in X(\IF_{q^m})} I(\psi(x)) = \frac{1}{N_m} \sum_{\chi \in \widehat{G}_m} \big( \sum_{x \in X(\IF_{q^m})} \chi(\psi(x)) \big)
\end{equation}

Before proceeding, let us recall that the Lang map $A \to A$ defined by $x \mapsto \mathrm{Frob}_{q^m}(x) - x$ is a Galois \'etale cover with Galois group $A(\IF_{q^m})$, which yields a surjective map $\pi_1^\et(A) \to A(\IF_{q^m})$. Here we always use a geometric point over the origin as the base point for $\pi_1^\et$. Composing this map with $\chi$ gives rise to rank $1$ $\overline{\IQ}_\ell$-local system $\sL_\chi$ such that for each $a \in A(\IF_{q^m})$ with geometric point $\bar{a}$, $\mathrm{Frob}_a$ acts on $\sL_{\chi, \bar{a}}$ as $\chi(a)$. Note that because $A$, as a cover over itself via the Lang map, is geometrically connected, the resulting map $\pi_1^\et(A) \to A(\IF_{q^m})$ remains surjective when restricted to $\pi_1^\et(A_{\overline{\IF}_q})$.

Now let us choose a base point $x_0 \in X(\IF_q)$ and factor $\psi$ as 
$$X \stackrel{j}{\to} \Jac(X) \stackrel{\psi'}{\to} A \stackrel{+a_0}{\longrightarrow} A$$ 
where $j$ is the Abel-Jacobi map, $\psi'$ is a homomorphism of abelian varieties, and $+a_0$ is translation by the point $a_0 = \psi(x_0) \in A(\IF_q)$. Thus, $\psi(x) = \psi'(j(x)) + a_0$. Applying similar constructions to $\Jac(X)$, we obtain a commutative diagram for the homomorphism $\psi'$:
\begin{equation}
    \label{eqn: Jac(X) to A}
    \begin{tikzcd}
	{\pi^\et_1(\Jac(X_{\overline{\IF}_q}))} & {\pi^\et_1(\Jac(X))} & {\Jac(X)(\IF_{q^m})} & \\
	{\pi^\et_1(A_{\overline{\IF}_q})} & {\pi^\et_1(A)} & {A(\IF_{q^m})} & {\overline{\IQ}_\ell^\times}
	\arrow[from=1-1, to=1-2]
	\arrow[curve={height=-18pt}, two heads, from=1-1, to=1-3]
	\arrow["{\pi_1^\et(\psi'_{\overline{\IF}_q})}"', from=1-1, to=2-1]
	\arrow[from=1-2, to=1-3]
	\arrow["{\pi_1^\et(\psi')}"', from=1-2, to=2-2]
	\arrow["{\psi'}"', from=1-3, to=2-3]
	\arrow[from=1-3, to=2-4]
	\arrow[from=2-1, to=2-2]
	\arrow[curve={height=-18pt}, two heads, from=2-1, to=2-3]
	\arrow[from=2-2, to=2-3]
	\arrow["\chi", from=2-3, to=2-4]
\end{tikzcd}
\end{equation}

Note that by assumption that $X$ generates $A$, $\psi' : \Jac(X) \to A$ is surjective. Let $K$ be the $p$-primary part of the cokernel of $\Jac(X)(\IF_{q^m}) \stackrel{\psi'}{\to} A(\IF_{q^m})$, so that $\hat{K} \subseteq \hat{G}_m$. Without loss of generality, we may replace $q$ by a finite power and assume that $|K|$ is independent of $m$. Indeed, let $H$ be the identity component of $\ker(\psi')_{\mathrm{red}}$. By Lang's theorem, $\H^1(\IF_{q^m}, H) = 0$, so that $\Jac(X)(\IF_{q^m}) \to (\Jac(X)/H)(\IF_{q^m})$ is surjective. Now the resulting $\psi'' : \Jac(X)/H \to A$ is an isogeny. As $|(\Jac(X)/H)(\IF_{q^m})| = |A(\IF_{q^m})|$, $|K|$ can also be viewed as the $p$-primary part of $|\ker(\psi'')(\IF_{q^m})|$. As $\ker(\psi'')$ is finite, we may replace $q$ by a finite power and assume that all $\bar{\IF}_q$ points are defined over $\IF_{q}$. Moreover, we may assume that $a_0$ vanishes in the cokernel of $\Jac(X)(\IF_{q}) \stackrel{\psi'}{\to} A(\IF_{q})$. This way for every $\chi \in \hat{K}$, $\chi(a_0) = 1$. Note also that translation $+a_0$ acts trivially on $\pi_1^\et$. 

If $\chi \in \hat{K}$, then the restriction of $\sL_\chi$ to $\Jac(X)$, and hence $X$, is trivial, so that the inner sum in \eqref{ptcountexp} merely counts the total number of rational points on $X$. Therefore, 
\begin{equation}
\label{eqn: K hat}
    \frac{1}{N_m} \sum_{\chi \in \widehat{K}} \big( \sum_{x \in X(\IF_{q^m})} \chi(\psi(x)) \big) = C \cdot \frac{|X(\IF_{q^m})|}{N_m} \textrm{ where } C := |K|.
\end{equation}
This is the main term of \eqref{ptcountexp}, which gives the naive expectation for $|X(\IF_{q^m}) \cap \psi^{-1}(A_{p'})|$ when $X(\IF_{q^m})$ is distributed absolutely randomly in $A(\IF_{q^m})$. 
By the Weil bound for curves (or Lang-Weil estimates), we have 
$$|X(\IF_{q^m})| = q^m + O(q^{m/2}). $$

Now suppose that $\chi \not\in \widehat{K}$. We compute the inner sum of \eqref{ptcountexp} by the Grothendieck-Lefschetz trace formula:
$$\sum_{x \in X(\IF_{q^m})} \chi(\psi(x)) = \sum_{x \in X(\IF_{q^m})} \mathrm{Tr}(\mathrm{Frob}_x | \sL_{\chi, \psi(\bar{x})}) = \sum_{i=0}^2 (-1)^i \text{Tr}\left(\text{Frob}_{q^m} \mid \H^i(X_{\overline{\IF}_q}, \psi^* \sL_\chi)\right)$$
Note that by the surjectivity of the map $\pi^\et_1(\Jac(X_{\overline{\IF}_q})) \to \Jac(X)(\IF_{q^m})$, the restriction of $\chi$ to $\pi^\et_1(\Jac(X_{\overline{\IF}_q}))$ is nontrivial. Equivalently, $\psi^* \sL_\chi$ is geometrically nontrivial, so that $\H^0(X_{\overline{\IF}_q}, \psi^* \sL_\chi) = 0$. By Poincar\'e duality, this forces $\H^2(X_{\overline{\IF}_q}, \psi^* \sL_\chi) = 0$. Consequently, only $\H^1(X_{\overline{\IF}_q}, \psi^* \sL_\chi)$ has nonzero a contribution.

Since $X$ is smooth and proper, by the Grothendieck-Ogg-Shafarevich formula, we have $$\chi(X_{\overline{\IF}_q}, \psi^* \sL_\chi) = \chi(X_{\overline{\IF}_q}, \overline{\IQ}_\ell). $$ Therefore, the Betti number $\dim \H^1(X_{\overline{\IF}_q}, \psi^* \sL_\chi) = -\chi(X_{\overline{\IF}_q}, \overline{\IQ}_\ell) = 2g - 2$ is a fixed constant independent of $\chi$, where $g$ is the genus of $X$. By Deligne's Weil II, the eigenvalues of $\text{Frob}_{q^m}$ acting on $\H^1$ of weight $0$ local system are bounded in absolute value by $q^{m/2}$. Therefore, we have

$$\left| \sum_{x \in X(\IF_{q^m})} \chi(\psi(x)) \right| \le (2g - 2)q^{m/2}$$

By summing over the $N_m - C$ characters in $\hat{G}_m \smallsetminus K$, the total accumulated error is bounded by:

$$\text{Error} \le \frac{1}{N_m} \sum_{\chi \not\in \widehat{K}} (2g - 2)q^{m/2} = \frac{N_m - C}{N_m} (2g - 2) q^{m/2} = O(q^{m/2})$$
This readily gives the result. 
\end{proof}

Now we are ready to prove \cref{thm: finite MM}. 

\begin{proof}
    The case where $\dim X = 0$ follows immediately from the definitions, so we may assume that $\dim X > 0$. We may moreover replace ``irreducible'' with ``geometrically irreducible'', since the order of a point of $A(\overline{k})$ is stable under Galois conjugation. Consequently, we may replace $k$ with a finite extension and assume that $X$ has a rational point $x_0 \in X(k)$.
    
    The implications (a) $\Rightarrow$ (b) $\Rightarrow$ (c) are elementary. It remains to show that (c) $\Rightarrow$ (a). 

    Let $B \subseteq A$ be the abelian subvariety generated by the differences $X - x_0$. Thus, $X \subseteq B + x_0$, and $B$ is the minimal abelian subvariety with this property. By assumption, the coset $B + x_0$ cannot be $p$-special, i.e., the image $\bar{x}_0 \in A/B$ must have an order strictly prime to $p$. 
    
    We can lift $\bar{x}_0$ to a rational prime-to-$p$ torsion point $x_0' \in A(k)$ such that $x_0 \equiv x_0' \pmod B$. This allows us to re-center the coset as $B + x_0 = B + x_0'$. 
    
    Define $X' = X - x_0'$. Then $X'$ is a subvariety of $B$ that generates $B$. To prove that $X \cap A_{p'}$ is dense in $X$, it suffices to prove that $X' \cap B_{p'}$ is dense in $X'$. If $\dim X' = 1$, then by replacing $X'$ by its normalization we conclude by the proceding lemma. If $\dim X' > 1$, we may find an alteration $\wt{X}' \to X'$ such that $\wt{X}'$ is smooth. Applying Lefschetz' hyperplane theorem for $\pi_1^\et$ \cite[XII, Cor.~3.5]{SGA2}, we find a Zariski dense set of smooth proper curves $\{ C_\alpha \}$ on $\wt{X}'$ such that $\pi_1^\et(C_\alpha) \to \pi_1^\et(\wt{X}')$ is surjective for each $\alpha$. By considering Albanese varieties, one readily sees that condition (c) holds for the image of each $C_\alpha$, and hence we reduce to the curve case.
\end{proof}

\subsection{Extensions of the Torelli morphism}

It is well known that there is a well defined moduli space of ``principally polarized abelian varieties with their rank $1$ degenerations''. See for example \cite[\S1]{Mum83}, where the construction is credited to Igusa. In loc. cit. the construction is carried out at the coarse moduli space level over $\IC$; for completeness, we explain why this moduli space exists at the level of Deligne-Mumford stacks over $\Spec\, \mathbb{Z}$. 

\begin{proposition}
\label{Mumfcompprop}
    There exists a normal Deligne-Mumford stack $\overline{\sfA}_g$, which is unique up to unique isomorphism, and contains $\sfA_g$ as an open substack, with the following properties: 
    \begin{enumerate}[label=\upshape{(\roman*)}]
        \item The universal abelian scheme $\sA$ extends as a normal semi-abelian scheme $\sG$ over $\overline{\sfA}_g$. 
        \item Let $S$ be an irreducible normal scheme and let $G/S$ be a semiabelian scheme whose generic fiber $G_\eta$ is abelian with a principal polarization $\lambda_\eta$. Let $U \subseteq S$ be a dense open subset such that $(G_\eta, \lambda_\eta)$ extends to a morphism $f_U : U \to \sfA_g$. Then $f_U$ extends to a (necessarily unique) morphism $f : S \to \overline{\sfA}_g$ if and only if for every point $s \in S \smallsetminus U$, the fiber $G_s$ has toral rank at most $1$. 
    \end{enumerate}
\end{proposition}

\begin{proof}
    The uniqueness statement follows from its characterizing properties (i) and (ii), see for example the explanation in \cite[Rmk.~4.1.7]{Ols08}. For the existence, one can take any smooth $\GL(X)$-admissible cone decomposition $\Sigma$ of $X = \IZ^g$, and take the corresponding toroidal compactification $\sfA_g^\Sigma$. Then one simply defines $\overline{\sfA}_g$ to be the open substack of $\sfA_g^\Sigma$ over which the fibres of the associated semiabelian family have toral rank at most $1$. Then property (ii) simply follows from \cite[IV, Thm.~5.7(5)]{FC90}. The point is that, for any $s$ of toral rank $1$, $B(X_s) \simeq \IR$ and the positive cone is just $\IR_{> 0}$. 
\end{proof}

Similarly, let us denote by $\overline{\sfM}_g$ not the full Deligne-Mumford compactification, but only the open part of $\sfM^{\mathrm{ct}}_g$ plus the open part of the boundary given by proper curves with a single node. Then by (ii) above, the Torelli morphism $\tau :  \sfM_g^{\mathrm{ct}} \to \sfA_g$ extends to $\overline{\sfM}_g \to \overline{\sfA}_g$. We will write $\partial \sfA_{g}$ for the boundary $\overline{\sfA}_{g} \setminus \sfA_{g}$ and similarly write $\partial \sfM^{\mathrm{ct}}_{g}$ for $\overline{\sfM}_{g} \setminus \sfM^{\mathrm{ct}}_{g}$.

There is a natural projection from the universal abelian variety $\sA_{g - 1} \to \p \sfA_g$ that identifies the open boundary stratum as the quotient $\sA_{g - 1}/ [\pm 1]$. Geometrically, for any abelian variety $A$ over a field $k$, extension classes are classified by the dual abelian variety via the natural isomorphism $\mathrm{Ext}^1_k(A, \IG_m) \simeq A^\vee(k)$ (see, e.g., \cite[Ch.~VII, {\S}16, Thm.~6]{Ser88}). Passing to the quotient by the negation map $[-1]$ is necessary to mod out the choice of identification between $\IG_m$ and the toral part of the extension.

Next, we recall the explicit description of the extended Torelli morphism along the boundary. The boundary $\p \sfM_g^\ct$ of $\sfM_g^\ct$ in $\overline{\sfM}_g$ parameterizes stable curves with non-separating nodes and is the image of the clutching morphism $\kappa: \overline{\sfM}_{g - 1, 2} \to \overline{\sfM}_g$. 

This boundary locus is irreducible, and its dense open substack $\p \sfM_g^\ct$ parameterizes irreducible genus $g$ curves with a single node. Restricting the clutching morphism yields a natural map $\sfM_{g -1, 2} \to \p \sfM_g^\ct$ that sends a tuple $([C], P, Q)$ to the nodal curve $C'$ obtained by gluing $P$ and $Q$. This map is a generic double cover that quotients out the ordering of the marked points $P$ and $Q$. 


Combining these local descriptions, the compatibility of the clutching involution on $\sfM_{g-1, 2}$ with the negation map on $\sA_{g-1}$ yields the following commutative diagram:
\begin{equation}
\label{eqn: boundary Torelli}
    \begin{tikzcd}
	{\sfM_{g - 1, 2}} & {\sA_{g - 1}} & {([C], P, Q) } & {(\Jac(C), [P - Q])} \\
	{\p \sfM_g^\ct} & {\p \sfA_{g}} & {[C']} & {[\Jac(C')]}
	\arrow[from=1-1, to=1-2]
	\arrow[from=1-1, to=2-1]
	\arrow[from=1-2, to=2-2]
	\arrow[maps to, from=1-3, to=1-4]
	\arrow[maps to, from=1-3, to=2-3]
	\arrow[maps to, from=1-4, to=2-4]
	\arrow["\tau", from=2-1, to=2-2]
	\arrow[maps to, from=2-3, to=2-4]
\end{tikzcd}
\end{equation}

\subsection{Deformations of split log $1$-motifs}
\label{subsec: deform log 1-motif} 

We first recall the description of $1$-motifs and log $1$-motifs that we will use. We then record a criterion for detecting prime-to-$p$ splittings from the abelian part of the degeneration data, and describe the local deformation space at a toral-rank-one boundary point. A locally free \'etale lattice over a scheme $S$ means a sheaf which is \'etale locally isomorphic to $\IZ^r_{S}$ for some $r$. For an abelian scheme $B$, we write $\sP_{B}$ for the universal Poincar\'e torsor associated to $B$.

\begin{definition}
\label{def: 1-motif}
Let $S$ be a scheme. A $1$-motif over $S$ is given by a tuple
$Q=(Y,X,B,c,c^\vee,\tau)$, where $Y$ and $X$ are locally free \'etale lattices, $B$ is an abelian scheme, $c:Y\to B$ and $c^\vee:X\to B^\vee$ are homomorphisms, and
$I:=(c\times c^\vee)^*\sP_B^{-1}$
is equipped with a trivialization
$\tau:1_{Y\times X}\xrightarrow{\sim}I$
of $\IG_m$-biextensions.

A polarization $\lambda$ on $Q$ is given by a pair $(\lambda_Y,\lambda_B)$, where $\lambda_Y:Y\to X$ is an embedding with finite cokernel and $\lambda_B$ a polarization on $B$, viewed as a map $B \to B^\vee$. We require that $c^\vee\circ\lambda_Y=\lambda_B\circ c$, and the pairing on $Y \times Y$ induced by $\tau$ is symmetric. The polarization $\lambda$ is said to be principal if both $\lambda_Y$ and $\lambda_B$ are isomorphisms.
\end{definition}

To relate this description to the usual one \cite[Def. 10.1.2]{zbMATH03375638}, let $T:=\Hom(X,\IG_m)$. By the Weil-Barsotti identification, the homomorphism $c^\vee:X\to B^\vee$ classifies a semi-abelian scheme 
\[
0\to T\to Q^{\mathrm{sab}}\to B\to 0.
\]
The trivialization $\tau$ is equivalent to a lift $u:Y\to Q^{\mathrm{sab}}$ of $c$ (cf. \cite[Lem.~1.2.2.3(1)]{Madthesis}). Thus $Q$ may equivalently be written as the complex
$[Y\xrightarrow{u}Q^{\mathrm{sab}}]$.

\begin{definition}
\label{def: log 1-motif}
Let $(S,M_S)$ be a fine saturated log scheme. Recall that $\IG_m^{\log}$ is the sheaf on the Kummer log flat site \cite[Def.~2.13, Cor.~2.22]{Niz08} given by
$\IG_m^{\log}(S')=\Gamma(S',M_{S'}^{\mathrm{gp}})$.
If $E$ is a $\IG_m$-torsor or biextension, write
$E^{\log}$ for $E\times^{\IG_m}\IG_m^{\log}$, which can be viewed either as a pushout or as a contracted product.

A log $1$-motif over $(S,M_S)$ is a tuple
$Q=(Y,X,B,c,c^\vee,\tau)$
as in \cref{def: 1-motif}, except that
$\tau:1^{\log}_{Y\times X}\xrightarrow{\sim}I^{\log}$
is a trivialization of $\IG_m^{\log}$-biextensions. Equivalently, it is a complex
$[Y\to(Q^{\mathrm{sab}})^{\log}]$.
\end{definition}

It is explained in \cite[Lem.~1.2.2.3]{Madthesis} that $\tau$ determines the monodromy pairing
$N_Q:Y\times X\to M_S^{\mathrm{gp}}/\sO_S^\times$. 
For a principally polarized log $1$-motif, the pairing
$(y,y')\mapsto N_Q(y,\lambda_Y(y'))$
is symmetric. We say that $(Q,\lambda)$ is positive if it satisfies the usual positivity condition of \cite[\S1.2.4]{Madthesis}. Over the spectrum of a DVR with its standard log structure, this means that the resulting $\IZ$-valued symmetric pairing on $Y$ is positive definite.

For our purposes, it suffices to take $S=\Spec(R)$, where $R$ is a normal complete local ring, with log structure defined by a divisor $D\subseteq S$. Let $U:=S\smallsetminus D$. We write $\mathsf{DD}_{\mathrm{pol}}$ for the category of polarized abelian schemes over $U$ which extend to semi-abelian schemes over $S$, and $\mathsf{DEG}_{\mathrm{pol}}$ for the category of positive polarized log $1$-motifs over $S$. The degeneration theory of Faltings--Chai gives an equivalence
$\mathsf{DD}_{\mathrm{pol}}\simeq\mathsf{DEG}_{\mathrm{pol}}$; see \cite[Ch.~III]{FC90} and \cite[\S1.2.4]{Madthesis}. It is compatible with prime-to-$p$ polarizations and therefore restricts to
$\mathsf{DD}^{(p)}_{\mathrm{pol}}\simeq\mathsf{DEG}^{(p)}_{\mathrm{pol}}$.

\begin{lemma}
\label{lem: polarized adjoint}
Let $(Q, \lambda) = (Y,X,\sB,c,c^\vee,\tau,\lambda)$ be a principally polarized log $1$-motif over $S$, and write $\lambda=(\lambda_Y,\lambda_{\sB})$. Let
$\Psi:(\sB,\lambda_{\sB})\to(\sC,\lambda_{\sC})$ be a similitude such that $\Psi^* \lambda_\sC = d \lambda_\sB$, and set
$\Psi^\sharp:=\lambda_{\sC}\circ\Psi\circ\lambda_{\sB}^{-1}:\sB^\vee\to\sC^\vee$.
Then $\Psi^\vee\circ\Psi^\sharp=[d]_{\sB^\vee}$ and $\Psi^\sharp\circ c^\vee
=\lambda_{\sC}\circ\Psi\circ c\circ\lambda_Y^{-1}$.

Suppose that $\sC=\sB_1\times\sB_2$ is equipped with the product polarization, and write $p_i$ and $p_i^\vee$ for the projections from $\sC$ and $\sC^\vee$, respectively. Then $p_1^\vee\circ\Psi^\sharp\circ c^\vee
=\lambda_{\sB_1}\circ p_1\circ\Psi\circ c\circ\lambda_Y^{-1}$.
In particular, if $p_1\circ\Psi\circ c(Y)$ is annihilated by an integer $N$, then so is
$p_1^\vee\circ\Psi^\sharp\circ c^\vee(X)$.
\end{lemma}

\begin{proof}
The first identity follows from
$\Psi^\vee\circ\lambda_{\sC}\circ\Psi=[d]\circ\lambda_{\sB}$,
and the second from
$c^\vee\circ\lambda_Y=\lambda_{\sB}\circ c$.
The assertion for the first factor follows from
$p_1^\vee\circ\lambda_{\sC}=\lambda_{\sB_1}\circ p_1$.
\end{proof}

\begin{proposition}
\label{prop: torsion implies splitting}
Suppose that $S$ is the spectrum of a complete DVR with its standard log structure. Let $(A,\lambda)$ be a prime-to-$p$ polarized abelian variety in $\mathsf{DD}^{(p)}_{\mathrm{pol}}$ and $(Q, \lambda) =(Y,X,\sB,c,c^\vee,\tau,\lambda)
\in\mathsf{DEG}^{(p)}_{\mathrm{pol}}$ be the corresponding object under the equivalence $\mathsf{DD}^{(p)}_{\mathrm{pol}}\simeq\mathsf{DEG}^{(p)}_{\mathrm{pol}}$. 
Fix an integer $r\geq 0$. Then the following are equivalent.

\begin{enumerate}[label=\upshape{(\alph*)}]
    \item There is a prime-to-$p$ similitude
    $(A,\lambda)\to(A_1,\lambda_1)\times(A_2,\lambda_2)$ in $\mathsf{DD}^{(p)}_{\mathrm{pol}}$ such that $A_1$ has good reduction over $S$, while $A_2$ has semi-abelian reduction of toral rank $r$.

    \item We have $\mathrm{rank\,} Y = r$, and there is a prime-to-$p$ similitude
    $\Psi:(\sB,\lambda_{\sB})\to
    (\sB_1,\lambda_{\sB_1})\times(\sB_2,\lambda_{\sB_2})$
    such that $p_1\circ\Psi\circ c(Y)$ is annihilated by an integer prime to $p$.
\end{enumerate}
\end{proposition}

\begin{proof}
Assume (b). After replacing $\Psi$ by $[N]\circ\Psi$ for a suitable integer $N$ prime to $p$, we may assume that
$p_1\circ\Psi\circ c=0$.
Set $\sC:=\sB_1\times\sB_2$, let $\lambda_{\sC}$ be the product polarization, and let $\Psi^\sharp:=\lambda_{\sC}\circ\Psi\circ\lambda_{\sB}^{-1}$.
By \cref{lem: polarized adjoint}, we also have
$p_1^\vee\circ\Psi^\sharp\circ c^\vee=0$.

Define $c_2:=p_2\circ\Psi\circ c$ and $c_2^\vee:=p_2^\vee\circ\Psi^\sharp\circ c^\vee$. Then $c_2^\vee\circ\lambda_Y=\lambda_{\sB_2}\circ c_2$.
Let $d$ be the factor such that $\Psi^* \lambda_\sC = d \lambda_\sB$. Then $\Psi \circ \Psi^\sharp = d$. Consider the commutative diagram 
\begin{equation}
    \begin{tikzcd}
	& {\sB \times \sB^\vee} & \\
	{\sC \times \sC^\vee} & {\sB \times \sC^\vee} & {\sB \times \sB^\vee}
	\arrow["{\Psi \times \Psi^\sharp}"', from=1-2, to=2-1]
	\arrow["{\mathrm{id} \times \Psi^\sharp}", from=1-2, to=2-2]
	\arrow["{\mathrm{id} \times [d]}", from=1-2, to=2-3]
	\arrow["{\Psi \times \mathrm{id}}"', from=2-2, to=2-1]
	\arrow["{\mathrm{id} \times \Psi^\vee}", from=2-2, to=2-3]
    \end{tikzcd}. 
\end{equation}
By the universal property of the Poincar\'e line bundle (see e.g., bottom of \cite[p.~130]{Mum70}), on the bottom we have a canonical isomorphism
$(\Psi \times \mathrm{id})^* \sP_\sC \simeq (\mathrm{id} \times \Psi^\vee)^* \sP_\sB$. Pulling it back along $\mathrm{id} \times \Psi^\sharp$, we obtain 
\begin{equation}
    \label{eqn: functoriality of Poincare}
    (\Psi \times \Psi^\sharp)^* \sP_\sC \simeq (\mathrm{id} \times [d])^* \sP_B \simeq \sP_\sB^{\tensor d}.
\end{equation}
The second isomorphism follows from the bi-additivity of the Poincar\'e lind bundle. Now, by pulling back (\ref{eqn: functoriality of Poincare}) along $c \times c^\vee : Y \times X \to \sB \times \sB^\vee$ and applying $\Psi \circ c = (0, c_2), \Psi^\sharp \circ c^\vee = (0, c_2^\vee)$, we obtain 
$(c_2\times c_2^\vee)^*\sP_{\sB_2}^{-1} \simeq \bigl((c\times c^\vee)^*\sP_{\sB}^{-1}\bigr)^{\otimes d}$. 
Thus $\tau^{\otimes d}$ induces a log trivialization $\tau_2$ on $(c_2\times c_2^\vee)^*\sP_{\sB_2}^{-1}$. Let us set $Q_2:=(Y,X,\sB_2,c_2,c_2^\vee,\tau_2, (\lambda_Y,\lambda_{\sB_2}))$. Then $Q_2$ is a positive polarized log $1$-motif. Set $Q_1:=(0,0,\sB_1,0,0,1^{\log}, \lambda_{\sB_1})$. Let $(A_i, \lambda_i) \in \mathsf{DD}^{(p)}_{\mathrm{pol}}$ be the object corresponding to $Q_i \in \mathsf{DEG}^{(p)}_{\mathrm{pol}}$ under the equivalence $\mathsf{DEG}^{(p)}_{\mathrm{pol}} \simeq\mathsf{DD}^{(p)}_{\mathrm{pol}}$. Then we obtain (a). 

Now we assume (a) and prove (b). Let $(Q_i, \lambda_i)$ be the polarized log $1$-motif attached to the chosen extension of $A_i$. The lattice parts of $Q_1$ vanish, while those of $Q_2$ have rank $r$. This implies that $\mathrm{rank\,} Y = r$. After we multiply the given similitude by an integer prime to $p$, it induces a prime-to-$p$ isogeny
$\Phi :Q\to Q_1\times Q_2$.
If $\Psi := (\Psi_\sB : \sB \to \sB_1\times\sB_2)$ is the induced map on the abelian parts, compatibility with the period maps gives $\Psi \circ c = (0,c_2) \circ \Phi_Y$.
This implies $p_1\circ \Psi \circ c=0$. 
\end{proof}

We conclude by relating these deformation spaces to the local geometry of the boundary. But first let us remark on our convention on talking about the ``formal completion'' of a stack at a point. 

\begin{convention}
    Let $\sfX$ be an algebraic stack locally of finite type over a field $k$, and let $x:\Spec(k)\to\sfX$. For an Artinian local $k$-algebra $R$ with residue field $k$, let $\Def_x(R)$ be the groupoid of pairs $(\xi,\alpha)$, where $\xi:\Spec(R)\to\sfX$ and $\alpha:\xi|_{\Spec(k)}\xRightarrow{\sim}x$ is a fixed isomorphism. This is the deformation category associated with $x$ in \cite[\href{https://stacks.math.columbia.edu/tag/0DQW}{Tag 0DQW}]{stacks-project}, which admits a versal complete Noetherian local ring by \cite[\href{https://stacks.math.columbia.edu/tag/0DQV}{Tag 0DQV}, \href{https://stacks.math.columbia.edu/tag/0DQX}{Tag 0DQX}]{stacks-project}. Now suppose that $\sfX$ is Deligne--Mumford. Choose a representable étale morphism $U\to\sfX$, a point $u\in U(k)$, and a $2$-isomorphism between the image of $u$ and $x$. Since $U \to \sfX$ is formally \'etale, $\Def_x$ can be identified with the ordinary deformation functor of $u$. The formal scheme $U^\wedge$, together with the tautological pair $(\xi, \alpha)$ given by the point $U^\wedge \to \Def_x$, is unique up to unique isomorphism. We shall write it as $\sfX^\wedge_x$. 
\end{convention}

\begin{remark}
For $\sfX=\sfA_{g,k}$, a point $x\in\sfA_{g,k}(k)$ corresponds to a principally polarized abelian variety $(A,\lambda)$ over $k$. In this case, $\Def_x$ is the usual deformation functor of $(A,\lambda)$: an $R$-point correspond to a principally polarized abelian scheme $(A_R,\lambda_R)$ over $R$ equipped with an isomorphism $(A_R,\lambda_R)\otimes_R k\xrightarrow{\sim}(A,\lambda)$. Thus $\Def_x$ is precisely the universal deformation space of the principally polarized abelian variety $(A,\lambda)$.
\end{remark}

\begin{construction}
    \label{constr: local boundary deformation diagram}Let $k$ be an algebraically closed field and $x : \Spec(k) \to \p \sfA_g$ be a geometric point on the boundary. Then there is a principally polarized log $1$-motif $(Q, \lambda)$ attached to $x$. Let $Q^{\sab}$ (resp. $Q^\ab$) denote the semi-abelian (resp. abelian) part of $Q$. Then there are natural forgetful morphisms 
\begin{equation}
\label{eqn: forgetful morphisms}
    \Def(Q,\lambda) \to \Def(Q^{\mathrm{sab}},\lambda^{\ab}) \to \Def(Q^{\ab},\lambda^{\ab}).
\end{equation}
between the corresponding deformation functors over $k$. By \S3.1.2 of \textit{loc. cit.}, these deformation functors are pro-representable by smooth formal schemes over $k$. Moreover, \cite[Thm.~4.2.1.3]{Madthesis} tells us that (\ref{eqn: forgetful morphisms}) can be identfied with suitable formal completions:
\begin{equation}
\label{eqn: deformation spaces}
    \begin{tikzcd}
	{\Def(Q,\lambda)} & {\Def(Q^{\mathrm{sab}},\lambda^{\ab})} & {\Def(Q^{\ab},\lambda^{\ab})} \\
	{(\overline{\sfA}_{g, k})^\wedge_x} & {(\p \sfA_{g, k})^\wedge_x} & {(\sfA_{g - 1, k})^\wedge_x}
	\arrow[from=1-1, to=1-2]
	\arrow["\simeq", from=1-1, to=2-1]
	\arrow[from=1-2, to=1-3]
	\arrow["\simeq", from=1-2, to=2-2]
	\arrow["\simeq", from=1-3, to=2-3]
	\arrow[from=2-1, to=2-2]
	\arrow[from=2-2, to=2-3]
    \end{tikzcd}
\end{equation}
Here we are abusively viewing $x$ as a geometric point on $\sfA_{g-1, k}$ via the natural projection. Recall that $\p \sfA_{g, k} \simeq \sA_{g - 1} / [\pm 1]$. If $Q = (Y, X, B := Q^{\ab}, c, c^\vee, \tau)$, then choosing a lift $\wt{x}$ of $x$ to $\sA_{g - 1}$ is equivalent to choosing a generator $1_Y$ for $Y$. Let us make such a choice. Then by \'etaleness we have $(\sA_{g - 1, k})^\wedge_{\wt{x}} = (\p \sfA_{g, k})^\wedge_x$. This gives us a concrete description of $\sfA_{g, k})^\wedge_x$: Let $b = c(1_Y) \in B$, and $\sB$ be the universal deformation of $B$ over $\Def(Q^{\ab}, \lambda^\ab)$. Then $(\sfA_{g, k})^\wedge_x$ is the formal completion of $\sB$ at the point $b$ on the fiber over the origin. Meanwhile, this is exactly the construction of $\Def(Q^\sab, \lambda^\ab)$ out of $\Def(Q^\ab, \lambda^\ab)$. 

There is a natural inclusion $(\p \sfA_{g, k})^\wedge_x \into (\sfA_{g, k})^\wedge_x$, but it is less obvious why there is a natural projection in (\ref{eqn: deformation spaces}). In short, the projection exists because the formal completion of $\overline{\sfA}_g$ along $\p \sfA_{g}$ can be identified with the completion of a line bundle over $\p \sfA_g$ along the zero section. This can be read off from \cite[Ch.~IV, Prop.~5.9]{FC90}. We briefly explain this using the notation therein, where our $\p \sfA_{g}$ is their $D_{g - 1}$. Faltings--Chai construct an $E_\xi$-torsor $\Xi_\xi$ and its contracted product $\overline{\Xi}_\xi=\Xi_\xi\times^{E_\xi}\overline{E}_\xi$.
Since $\rank\, X_\xi=1$, one has $E_\xi=B(X_\xi)\otimes\IG_m\simeq\IG_m$ as sheaves of groups, and the relevant torus embedding is
$\IG_m\subset\IA^1$. Hence $\overline{\Xi}_\xi$ is the total space of the line bundle associated with $\Xi_\xi$, and its closed orbit is the zero section. The formal scheme $S_\xi$ is precisely the formal completion of $\overline{\Xi}_\xi$ along the zero section. We leave to the reader to check that when completed at $x$, the resulting projection $(\sfA_{g, k})^\wedge_x \to (\p \sfA_{g, k})^\wedge_x$ is compatible with $\Def(Q, \lambda) \to \Def(Q^\sab, \lambda^\ab)$. 
\end{construction}

\subsection{Proof of the elliptic factor case}

\begin{lemma}
    \label{Hurwitzspacelem}
    Assume that $g \ge 1$, $n \ge 2$, $p > 2$ and $p \nmid n$. Then the locus $\sH^{(n)}_g \subseteq \sfM_{g, 2, \IF_p}$ defined by $n(P - Q) \equiv 0$ has pure dimension $2g - 1$.
\end{lemma}
\begin{proof}
    The condition $n(P - Q) \equiv 0$ is equivalent to the existence of a degree $n$ cover $f: C \to \IP^1$ such that $f(P)=0$, $f(Q)=\infty$, and $f$ is totally ramified over $0$ and $\infty$. Generically, apart from $P, Q$, the branch locus consists of simple branch points in $\IP^1 \smallsetminus \{0, \infty\}$, and by Riemann-Hurwitz there are $2g$ of them. Since $\mathrm{Aut}(\IP^1, \{0, \infty\}) \cong \IG_m$, $\dim \sH^{(n)}_g = 2g - 1$.
\end{proof}

\begin{remark}
\label{rmk: C^2 not p-special}
    We will use the following simple observation: let $C$ be a smooth proper curve of genus $\ge 2$ over a field $k$. The image of the natural map $C^2 \to \Jac(C)$ defined by $(P, Q) \mapsto [P - Q]$ does not lie in any $p$-special subvariety. This is evident either because the origin lies in the image, or because for any fixed $Q \in C$, the restriction $C \simeq (C \times \{ Q \}) \to \Jac(C)$ is precisely the Abel-Jacobi map with base point $Q$. 
\end{remark}

\begin{theorem}
\label{thm: elliptic case}
    For every prime $p$, the locus $\mathrm{SP}_{\bg, p}$ for $\bg = \{1, g - 1\}$ is Zariski dense in $\sfM_{g, \IF_p}$; moreover, every irreducible component of $\mathrm{SP}_{\bg, p}$ is of codimension $\delta(\bg) = g - 1$. 
\end{theorem}

\begin{proof}
    Let $C$ be any curve of genus $g - 1$ over $\bar{\IF}_p$.  By \cref{thm: finite MM} and \cref{rmk: C^2 not p-special}, the set $\{ (P, Q) \in C^2 : n(P - Q) \equiv 0 \textit{ for some } n \not\in p \IN \}$ is Zariski dense on $C^2$. Let us fix such a pair $(P, Q)$ with $P \neq Q$, so that $([C], P, Q)$ is a point on $\sfM_{g - 1, 2}(\overline{\mathbb{F}}_{p})$. Let $[C'] \in \p \sfM_{g}^\ct(\overline{\mathbb{F}}_{p})$ be the corresponding point obtained by clutching $P, Q$. It suffices to argue that $[C']$ can be deformed to a smooth genus $g$ curve $C''$ such that $\Jac(C'')$ admits a prime-to-$p$ splitting of type $\bg$. 

    Let $(Q_0, \lambda_0)$ be the log $1$-motif attached to $\tau([C'])$. Then by \autoref{constr: local boundary deformation diagram} the deformation space $\Def_{(Q_0, \lambda_0), \overline{\mathbb{F}}_{p}}$ can be identified with the local completion of $\overline{\sfA}_{g,\overline{\mathbb{F}}_{p}}$ at $\tau([C'])$. Note that the abelian part of $(Q_0, \lambda_0)$ is $B_0 := \Jac(C)$. Let $\lambda_0^\ab$ be the abelian part of $\lambda_0$, which is also the canonical polarization on $B_0$ as a Jacobian. Let $\sB$ be the universal family over the deformation space $\Def_{(B_0, \lambda^\ab_0),\overline{\mathbb{F}}_{p}}$. There is a natural morphism $\varpi : \Def_{(Q_0, \lambda_0), \overline{\mathbb{F}}_{p}} \to \sB_{\overline{\mathbb{F}}_{p}}$ from \autoref{constr: local boundary deformation diagram}. The \'etale section $\sB[n] \subseteq \sB$ is of codimension $g - 1$, and hence since $\varpi$ is smooth, the locus $\varpi^{-1}(\sB[n])$ is also of codimension $g-1$ in $\Def_{(Q_0, \lambda_0),\overline{\mathbb{F}}_{p}}$. 
    
    We note that for dimension reasons, $\varpi^{-1}(\sB[n]) \times_{\overline{\sfA}_g,\overline{\mathbb{F}}_{p}} \overline{\sfM}_{g,\overline{\mathbb{F}}_{p}}$ cannot be contained entirely in the boundary $\p \sfM_{g,\overline{\mathbb{F}}_{p}}^{\ct}$. Indeed, $\varpi^{-1}(\sB[n]) \times_{\overline{\sfA}_{g,\overline{\mathbb{F}}_{p}}} \overline{\sfM}_{g,\overline{\mathbb{F}}_{p}}$ has codimension at most $g - 1$ in $(\overline{\sfM}_{g,\overline{\mathbb{F}}_{p}})^{\wedge}_{[C']}$. On the other hand, $\varpi^{-1}(\sB[n]) \times_{\overline{\sfA}_{g, \overline{\mathbb{F}}_{p}}} \sfM_{g - 1, 2, \overline{\mathbb{F}}_{p}}$ lies in $\sH^{(n)}_{g - 1, \overline{\mathbb{F}}_{p}}$, and hence has dimension at most $2g - 3$ by \autoref{Hurwitzspacelem}. Therefore, there must be a discrete valuation field $F$ and a $\wt{s} \in [\Def_{(Q_0, \lambda_0),\overline{\mathbb{F}}_{p}} \cap \varpi^{-1}(\sB[n])] (\sO_F)$, such that $\wt{s}_F \in \tau(\sfM_{g, \overline{\mathbb{F}}_{p}})$. Moreover, by \cref{prop: torsion implies splitting}, $\sA_{\wt{s}_F}$ admits a prime-to-$p$ splitting of type $\bg$. 
\end{proof}




\subsection{Proof of the general case}
We first state a more refined version of \cref{thm: main density}. Let us say that a \textit{principally polarized} abelian variety (PPAV) $(A, \lambda)$ over a field $k$ has prime-to-$p$ splitting type $\bg$ if there exist PPAVs $(A_j, \lambda_j)$ of dimension $g_j$ and morphisms $\psi_j : A_j \to A_{\bar{k}}$ with finite prime-to-$p$ kernels, such that the product $\psi := \prod_j \psi_j$ satisfies $\psi^*(\lambda_{\bar{k}}) \in \mathbb{Z}_{(p)}^\times (\boxtimes_j \lambda_j)$.


\begin{lemma}
\label{lem: dense in fibration}
    Let $\pi : P \to X$ be an open morphism between topological spaces, and $I \subseteq P$ be a subset. Suppose that for $x$ in a dense subset $S \subseteq X$, the fiber $I_x$ is nonempty and dense in $P_x$. Then $I$ is dense in $P$.  
\end{lemma}
\begin{proof}
    Let $U \subseteq P$ be a nonempty open subset. Then $\pi(U) \subseteq X$ is open. As $S \subseteq X$ is dense, there exists some $x \in S \cap \pi(U)$. As $U_x$ is open in $P_x$ and $I_x$ is dense in $P_x$, $I_x \cap U_x \neq \emptyset$. Therefore, $U \cap I \neq \emptyset$. 
\end{proof}

\begin{theorem}
\label{thm: main density2}
    Suppose that $\bg$ is a partition of $g \in \IN$ such that $\delta(\bg) < 3g - 3$. Then there exists $r_\bg \in \IN$ such that for every $p > r_\bg$, the locus 
    \begin{equation}
         \mathrm{SP}_{\bg, p}^{\mathrm{pol}} := \{ s \in \sfM_{g, \IF_p} : (\Jac(\sC_s), \lambda^{\can}) \textit{ is prime-to-$p$ of splitting type } \bg \}
    \end{equation}
    is Zariski dense in $\sfM_{g, \IF_p}$. Moreover, away from a strict closed subscheme $E \subsetneq \sfM_g$, all components of $\mathrm{SP}^{\mathrm{pol}}_{\bg, p}$ have the expected codimension $\delta(\bg)$ in $\sfM_{g, \IF_p}$.
\end{theorem}

\begin{proof}
    Since for $g \le 3$, the Torelli morphism $\sfM_g \to \sfA_g$ is generically \'etale, the theorem is clear for $g \le 3$. For example, one can invoke the hard fact that the prime-to-$p$ Hecke orbit of any ordinary point is dense in $\sfA_g$ for any $g$ (\cite{Chai95}) for a quick overkill. For $g \ge 4$, consider the partition $\bg' = \{ g_1, g_2, \cdots, g_{k - 1}, g_k - 1 \}$ of $g - 1$. One readily checks that $$\delta(\bg) - \delta(\bg') = \sum_{j = 1}^{k  - 1} g_j = g - g_k.  $$

    We argue that we may always, after possibly reordering the $g_{i}$, assume that $g_k$ is such that $\delta(\bg') < \dim_\IZ \sfM_{g - 1}$. One readily checks the case for $g = 4$. Hence we assume that $g \ge 5$. Then we may always choose a $g_k$ such that $g - g_k \ge 3$, so that $\delta(\bg) < \dim_\IZ \sfM_{g}$ implies that $\delta(\bg') < \dim_\IZ \sfM_{g - 1}$.

    Assume by induction that the theorem is proved for $g - 1$. That is, for some $r = r(\bg')$ and $p > r$, the locus $\mathrm{SP}^{\mathrm{pol}}_{\bg', p}$ is Zariski dense in $\sfM_{g - 1, \IF_p}$. Moreover, let $E \subsetneq \sfM_{g - 1, 2}$ be the strict closed subscheme obtained by applying \cref{prop: main prop mixed}. Then we prove the theorem when $E_{\IF_p} \subsetneq \sfM_{g - 1, 2, \IF_p}$. 
    
    By \cref{thm: finite MM}, \cref{rmk: C^2 not p-special} and \cref{lem: dense in fibration}, one easily sees that the following set
    $$ \wt{B} := \{ ([\wt{C}], P, Q) \in (\sfM_{g -1, 2})_{\bar{\IF}_p} : [\wt{C}] \in \mathrm{SP}^{\mathrm{pol}}_{\bg', p} \text{ and } n[P - Q] \equiv 0 \text{ for some } n \not\in p\IN \} $$
    is Zariski dense. Now take some $([\wt{C}], P, Q) \in \wt{B}$ outside $E_{\IF_p}$ and let $C$ be the resulting nodal curve obtained by gluing $P$ and $Q$. We reduce to showing that $[C] \in \p \sfM_g$ can be deformed into a point $[C'] \in \sfM_g$ in the interior such that $[C'] \in \mathrm{SP}^{\mathrm{pol}}_{\bg, p}$. 

    For the remainder of the proof we work over $\mathbb{F}_{p}$. For ease of notation, we omit the base-change $(-)_{\mathbb{F}_{p}}$ from the notation.

    Let $(Q, \lambda)$ be the polarized log $1$-motif given by the boundary point $\tau([C])$ of $\overline{\sfA}_g$, so that the deformation space $\Def(Q, \lambda)$ is identified with the formal completion of $\overline{\sfA}_{g}$ at $\tau([C])$. Let $0$ be the origin of $\Def(Q, \lambda)$. Note that $(Q^\ab, \lambda^\ab) = (\Jac(\wt{C}), \lambda^\can)$.

    Let $\sB$ be the universal family over $\Def(\Jac(\wt{C}), \lambda^\can)$. Then $\sB_0 = \Jac(\wt{C})$, and there is a factorization 
    \begin{equation}
        \begin{tikzcd}
	{\Def(Q, \lambda)} & \sB \\
	& {\Def(\Jac(\wt{C}), \lambda^\can)}
	\arrow["\varpi", from=1-1, to=1-2]
	\arrow[from=1-1, to=2-2]
	\arrow[from=1-2, to=2-2]
        \end{tikzcd}
    \end{equation}
    with $\varpi$ as in \autoref{constr: local boundary deformation diagram}. By our assumption that $[\wt{C}] \in \mathrm{SP}^{\mathrm{pol}}_{\bg', p}$, there exist morphisms $\psi_j : A_j \to \sB_0$ with prime-to-$p$ kernels such that 
    $$ \psi = \prod_{j = 1}^k \psi_j : (\prod_{j = 1}^k A_j, \lambda_j) \to (\Jac(\wt{C}), \lambda^\can) $$ 
    is a symplectic similitude, and $\dim A_j = g_j$ for $1  \le j \le k - 1$ and $\dim A_k = g_k - 1$. Let $\psi^{-1}$ be an isogeny which is an inverse of $\psi$ in the prime-to-$p$ isogeny category of abelian varieties. 

    There exists a unique closed embedding 
    $$ \iota : \prod_{j = 1}^k \Def(A_j, \lambda_j) \into \Def(\Jac(\wt{C}), \lambda^\can) $$
    characterized by the deformation of $\psi$. That is, if $\sB_j$ is the universal family over $\Def(A_j, \lambda_j)$, then $\psi$ deforms to $\Psi : \prod_{j = 1}^k \sB_j \to \iota^* \sB$. 

    Let $(a_j)_{j = 1}^k \in \prod_j A_j$ be a point lifting the image of $[P - Q]$ in $\Jac(\wt{C})$, and let $\sD \subseteq \prod_{j = 1}^{k - 1} \sB_{j}$ be the prime-to-$p$ torsion section which extends $(a_j)_{j = 1}^{k - 1}$. Define $\Delta$ and $\p \Delta$ by the following fiber diagram
    \begin{equation}
        \begin{tikzcd}
	{\p \Delta} & \Delta & {\sD \times \sB_k} \\
	{\p \Def(Q, \lambda)} & {\Def(Q, \lambda)} & \sB
	\arrow[from=1-1, to=1-2]
	\arrow[from=1-1, to=2-1]
	\arrow["\lrcorner"{anchor=center, pos=0.125}, draw=none, from=1-1, to=2-2]
	\arrow[from=1-2, to=1-3]
	\arrow[from=1-2, to=2-2]
	\arrow["\lrcorner"{anchor=center, pos=0.125}, draw=none, from=1-2, to=2-3]
	\arrow["\Psi", from=1-3, to=2-3]
	\arrow[from=2-1, to=2-2]
	\arrow["\varpi", from=2-2, to=2-3]
        \end{tikzcd}. 
    \end{equation}  
    
    It follows from \cref{cor: intersect with torsion} that $\p \Delta \times_{\sfA_g} \sfM_{g- 1, 2}$ is of the expected codimension $\delta(\bg') + (g - g_k) = \delta(\bg)$ in $\sfM_{g - 1, 2}$, or equivalently in $\p \sfM_{g}$. Note that $\delta(\bg)$ is also the codimension of $\sD \times \sB_k$ in $\sB$. Indeed, we have \begin{align*}
        \dim \sB &= \frac{(g - 1)g}{2} + (g - 1) = \frac{g(g + 1)}{2} - 1 \\
        \dim \sD \times \sB_k &= \left( \sum_{j = 1}^{k - 1} \frac{g_j(g_j + 1)}{2} \right) + \frac{(g_k - 1) g_k}{2} + (g_k - 1) = \left( \sum_{j = 1}^{k} \frac{g_j(g_j + 1)}{2} \right) - 1. 
    \end{align*}
    As $\varpi$ is smooth, $\Delta$ has codimension $\delta(\bg)$ in $\Def(Q, \lambda)$. This implies that $\Delta \times_{\overline{\sfA}_g} \overline{\sfM}_g$ has codimension at most $\delta(\bg)$ in $\overline{\sfM}_g$, as $\Def(Q, \lambda)$ is the formal completion of $\overline{\sfA}_g$ at $\tau([C])$.

    On the other hand, since $\p \Delta \times_{\overline{\sfA}_g} \overline{\sfM}_g$ is cut out by a single equation in $ \Delta \times_{\overline{\sfA}_g} \overline{\sfM}_g$, the codimension of $ \Delta \times_{\overline{\sfA}_g} \overline{\sfM}_g$ in $\overline{\sfM}_g$ must be exactly $\delta(\bg)$, and $$ \dim \p \Delta \times_{\overline{\sfA}_g} \overline{\sfM}_g  = \dim  \Delta \times_{\overline{\sfA}_g} \overline{\sfM}_g - 1. $$
    This implies that, there must be a complete local ring $R$ together with a morphism $\Spec(R) \to \Delta \times_{\overline{\sfA}_g} \overline{\sfM}_g$ such that the generic point of $\Spec(R)$ does not map into the boundary. It follows from \cref{prop: torsion implies splitting} that the image of this generic point defines a point in the interior locus $\mathrm{SP}^{\mathrm{pol}}_{\mathbf{g}, p}$. 


    We have now proved the density statement in the theorem. The statement on the dimension follows from \cref{shimflatnessthmv2}. 
\end{proof}

\begin{remark}
\label{rmk: dense in product}
    Note that in the \cref{conj: AIM}, the density is asked for $I := \{ (x, y) : x \in \cup_T T(y) \}$ in the product $X \times Y$. In the above, $X = \tau(\sfM_g)$ and $Y = \sfA_{g_1} \times \cdots \times \sfA_{g_k}$, and we proved the projection of $I$ to $X$ is dense. But this is equivalent to proving the density on $X \times Y$. Indeed, one just applies \cref{lem: dense in fibration} with $P = X \times Y$ together with the usual Hecke-orbit conjecture for the factors $\sfA_{g_i}$, proved by Chai \cite{Chai95}, which gives the required density in the fibers. The same observation applies when $X$ is the image of the moduli of hyperelliptic curves or cubic threefolds discussed below. 
\end{remark}

\subsection{The Hyperelliptic Case and Cubic Threefolds}

We now prove a variant of \autoref{thm: main density} for hyperelliptic curves as well as \autoref{cubicthreefoldthm}. Since the arguments for the hyperelliptic curve theorem are essentially identical to the result for the moduli space of curves, we will only sketch them. In what follows we write $\sfH_g$ for the moduli space of hyperelliptic curves over $\Spec\, \IZ$, and write $\overline{\sfH}_{g}$ for the natural compactification obtained by taking the closure in $\overline{\sfM}_{g}$. Then $\overline{\sfH}_{g}$ has a unique boundary component $\partial \sfH_{g} = \overline{\sfH}_{g} \cap \partial \sfM_{g}$. The clutching map $\kappa: \sfM_{g - 1, 2} \to \partial \sfM^\ct_g$ restricts to a map $\sfH_{g-1,2} \to \partial \sfH_{g}$ which we also denote by $\kappa$. Here $\sfH_{g-1,2}$ is the moduli space of triples $(C, p, \iota(p))$, where $C$ is a hyperelliptic curve with point $p \in C$, and $\iota : C \to C$ is the hyperelliptic involution.

\begin{theorem}
\label{thm: main density3}
    Suppose that $\bg$ is a partition of $g \in \IN$ such that $\delta(\bg) < 2g - 1$. Then there exists $r_\bg \in \IN$ such that for every prime $p > r_\bg$, the locus 
    \begin{equation}
         \mathrm{SP}^{\mathrm{pol}}_{\bg, p} := \{ s \in \sfH_{g, \IF_p} : \Jac(\sC_s) \textit{ is prime-to-$p$ of splitting type } \bg \}
    \end{equation}
    is Zariski dense in $\sfH_{g, \IF_p}$. Moreover, away from a strict closed subscheme $E \subsetneq \sfH_g$, all components of $\mathrm{SP}^{\mathrm{pol}}_{\bg, p}$ have the expected codimension $\delta(\bg)$ in $\sfH_{g, \IF_p}$.
\end{theorem}

\begin{proof}
For $g \leq 2$ the map $\sfH_{g} \to \sfA_{g}$ is generically \'etale, so the result is clear. For $g \geq 3$, we choose $\mathbf{g}'$ as in the proof of \autoref{thm: main density2}. We argue that we may assume $g_{k}$ is such that $\delta(\mathbf{g}') < \dim_{\IZ} \sfH_{g-1} = 2(g-1) - 1$. Indeed, as $g \geq 3$ we can always choose $g_{k}$ such that $g - g_{k} \geq 2$, so $\delta(\mathbf{g}) < \dim_{\IZ} \sfH_{g}$ implies $\delta(\mathbf{g}') < \dim_{\IZ} \sfH_{g-1}$.

Assume by induction that the result holds for $g-1$, hence that there is $r = r(\mathbf{g}')$ such that the locus $\mathrm{SP}_{\bg', p}$ is Zariski dense in $\sfH_{g-1,\mathbb{F}_{p}}$ for $p > r$. Let $E \subset \sfH_{g-1,2}$ be the strict closed subscheme obtained by applying \cref{prop: main prop mixed2}. The rest of the argument proceeds exactly as in the proof of \autoref{thm: main density2}, replacing $\sfM_{g}$ with $\sfH_{g}$ and $\sfM_{g-1,2}$ with $\sfH_{g-1,2}$.
\end{proof}

As a corollary we also obtain \autoref{cubicthreefoldthm}.

\begin{proof}[Proof of (\ref{cubicthreefoldthm}):]

    Consider the Torelli map $\tau : \sfM \to \sfA_5$ given by taking intermediate Jacobians. Write $\overline{\tau(\sfM)}$ for the Zariski closure in $\sfA_5$ of the image. Then over $\IC$, the image $\overline{\tau(\sfM)}$ has a boundary component $B := \mathrm{im}(\sfH_5 \to \sfA_5)$ by \cite[Thm.~1.1]{CML09}. We may assume this is also true over $R := \mathbb{Z}[1/N]$ for $N$ large enough by spreading out. 
    
    We have $\dim_{R} \sfH_5 = 9$, $\dim_{R} \sfA_5 = 15$, and $\delta(\bg) = 15 - (1 + 3 + 3) = 8$ with $\bg = (1,2,2)$, so by \autoref{thm: main density3} we have that $\mathrm{SP}_{\bg, p}$ is Zariski dense in $\sfH_{5,\mathbb{F}_{p}}$ for $p \gg 0$. We claim that this even implies that $\mathrm{SP}_{\bg', p}$ is Zariski dense in $\sfH_{5,\mathbb{F}_{p}}$ for $\bg' = (1,1,1,2)$ and $p \gg 0$. It suffices to show that, for each component $U$ of $\mathrm{SP}_{\bg, p}$, the locus $\mathrm{SP}_{\bg', p} \cap U$ is Zariski dense in $U$. We may moreover consider just those $U$ which intersect the ordinary locus in $\sfH_{5,\mathbb{F}_{p}}$. In fact, we may even replace $U$ with a general curve $C \subset U$ and show that $\mathrm{SP}_{\bg', p} \cap C$ is Zariski dense in $C$.

    Now we use \autoref{atyplocusconstrprop} with $T = \sfH_{5}$, and with the variation of Hodge structure the one associated to the natural map $T \to \sfA_{g}$. Consider the embedding $\sfA_{1} \times \sfA_2 \times \sfA_2 \hookrightarrow \sfA_5$ and let $\che{X}_1 \times \che{X}_2 \times \che{X}_2 \hookrightarrow \che{D}$ be the associated embedding of flag varieties. We take the $Z$ of \autoref{atyplocusconstrprop} to be $\che{X}_1 \times \{ \star \} \times \che{X}_2$ for some $\mathbb{Z}[1/N]$-point $\star \in \che{X}_2$, where we take $N$ to be sufficiently large. Taking $N$ large enough and $d = 1$, we obtain a strict locus $E_{5} \subset \sfH_{5}$ from \autoref{atyplocusconstrprop}, which is called $E^{d}$ in the notation there. We may now further replace $U$ with $U \setminus E_{5}$, taking $p \gg 0$ as necessary.

    Now over $C$, the universal family of abelian fivefolds splits up to isogeny as a product family; there is a dominant generically \'etale map $C' \to C$ and a family $A' \to C'$ such that $A' \simeq B_{1} \times_{C'} B_{2} \times_{C'} B'_{2}$, where $B_{1}$ (resp. $B_{2}$ and $B'_{2}$) is a family of abelian varieties of relative dimension one (resp. $2$) over $C'$. It now suffices to show that the family $B_{2} \to C'$ admits infinitely many specializations where the fibre $B_{2,c}$ above $c \in C'$ splits as a product of elliptic curves. We may do this by applying \cite[Thm.~1.4]{Tay25}, which builds on \cite{MSTI}; note that the splittings referred to there are prime-to-$p$, as follows from the more precise statement \cite[Thm.~4.8]{Tay25}. To verify the non-isotriviality hypothesis on $B_{2} \to C'$, we observe that if $B_{2} \to C'$ was isotrivial, we could construct a non-degenerate jet $\mathbb{D}^{1}_{r} \to C'$ landing at a general point of $C'$ which maps into $E_5$, which would imply the image of $C'$, and hence $C$, lies inside $E_5$. But $C$ does not lie inside $E_5$ by assumption.

    Now consider a component $V$ of $B_{\mathbb{F}_{p}} \cap \mathrm{SP}_{\bg', p}$. Observe that, away from the closed sublocus $E \subsetneq B$ determined by \autoref{prop: main prop mixed2}, the component $V$ must be zero-dimensional: indeed, the expected dimension of $V$ in $B_{\mathbb{F}_{p}}$ is 
    \[ \underbrace{\dim B_{\mathbb{F}_{p}}}_{9} +\, (3 + 1 + 1 + 1) - \dim \sfA_{5,\mathbb{F}_{p}} = 0 , \]
    so this follows from \autoref{cor: intersect with torsion2}.

    Now let $w : Y \to \sfA_{5}$ be a Shimura variety parameterizing split abelian varieties of type $(1,1,1,2)$ with good integral model at $p$, and such that the intersection of the image with $B_{\mathbb{F}_{p}}$ does not lie inside $E_{\mathbb{F}_{p}}$. Then since $Y$ is lci over $\sfA_{5}$, it follows that the fibre product $Y \times_{\sfA_{5}, w, \tau} \sfM$ does not lie entirely above the boundary locus $B$. Since, ranging over all such $w$, the intersections of the images of the $w$ with $B$ are Zariski dense, and $\dim B = \dim \sfM - 1$, it follows that the images of the fibre products $(Y \times_{\sfA_{5}, w, \tau} \sfM)_{\mathbb{F}_{p}}$ are Zariski dense in $\sfM_{\mathbb{F}_{p}}$, or that the locus of cubic threefolds whose intermediate Jacobians have splitting type $(1,1,1,2)$ is dense in $\sfM_{\mathbb{F}_{p}}$ for $p \gg 0$. The components of this locus are positive-dimensional, so we can use the same strategy as before, together with \cite[Thm.~1.4]{Tay25} and \cite[Thm.~4.8]{Tay25}, to show this also holds for those cubic threefolds whose intermediate Jacobians have splitting type $(1,1,1,1,1)$.
\end{proof}

\section{Intersections with supersingular loci}

\subsection{Orthogonal Shimura varieties}
\label{orthoshimsec}

In this section we prove \autoref{orthogonalssthm}. Let $L$ be a quadratic lattice over $\IZ$ of signature $(2, b)$. Before proving the theorem let us recall  basic facts about the supersingular loci of the associated orthogonal Shimura varieties $\shS$ defined over $\IZ[1/N]$. If $b = 2m - 1$ is odd, then $\shS^{\mathrm{ss}}_{\IF_p}$ is of dimension $m - 1$. If $b = 2m$ is even, then there are two cases
\begin{equation*}
    \dim \shS^{\mathrm{ss}}_{\IF_p} = \begin{cases}
        m, \text{ when } L_{\IQ_p} \not\simeq U^{\oplus (m + 1)} ;\\
        m - 1, \text{ when } L_{\IQ_p} \simeq U^{\oplus (m + 1)}, 
    \end{cases}
\end{equation*}
where $U$ denotes the standard hyperbolic plane. 

\begin{definition}
\label{def: orthogonal F-crystal}
    Let $k$ be an algebraically closed field over $\IF_p$. A \textit{orthogonal F-crystal of level $2$} is a tuple $(H, \<-, -\>, \varphi)$ where $(H, \< -, -\>)$ is a self-dual quadratic lattice over $W := W(k)$, $\varphi$ is a $\sigma$-linear Frobenius action on $H$, such that $p^2 H \subseteq \varphi(H)$ and $\< \varphi (x), \varphi(y) \> = p^2 \sigma(\< x, y \>)$ for all $x,y$. We define $h^0 := \dim_k \mathrm{im}(\varphi \tensor_W k)$. 
\end{definition}

A K3 crystal is an orthogonal F-crystal of level $2$ with $h^0 = 1$, and we say that it is supersingular if $\varphi$ is purely of slope $1$. Given a supersingular K3 crystal $H$, its Tate module $T := H^{\varphi = p}$, as a quadratic lattice over $\IZ_p$, orthogonally decomposes as 
$$ (T, \< -, -\>_T) = (T_0, p \< -, -\>_{T_0}) \oplus (T_1, \< -, -\>_{T_1}), $$
where $\< -, -\>_{T_0}, \< -, -\>_{T_1}$ are self-dual. Moreover, $\mathrm{rank\,} T_0$ is always even and we define $\sigma_0 := \mathrm{rank\,} T_0 / 2$. This number $\sigma_0$ is called the \textit{Artin invariant}, and it is also the dimension of $\ker (T \tensor_{\IZ_p} k \to H \tensor_W k)$. One can deduce these facts from \cite[Thm.~3.3-3.5]{Ogu79}.

Over $\shS_{\IF_p}$, there is a family of K3 crystals $\bL_\crys$.

\begin{theorem}
    \label{preciseorthogonalshimurasupsingthm}
    Let $T$ be a smooth scheme over $\IZ[1/N]$ which admits a quasi-finite morphism $T \to \shS$ such that $T_\IC$ has maximal monodromy. For each $p$, let us say that a component $C$ of the restriction $T_{\IF_p} \times_\shS \shS^{\mathrm{ss}, \sigma_0}_{\IF_p}$ is atypical if 
    \begin{equation}
    \label{eqn: dim of C}
         \dim C > \dim_\IZ T + \dim \shS^{\mathrm{ss}, \sigma_0}_{\IF_p} - \dim_\IZ \shS.
    \end{equation}
    Then the union of positive-dimensional atypical components 
    $$ \bigcup_{p \in \Spec(\IZ[1/N])} [T_{\IF_p} \times_\shS \shS^{\mathrm{ss}, \sigma_0}_{\IF_p}]^{\mathrm{atyp}, >0} $$
    is not Zariski dense in $T$. Moreover, the $\IC$-fiber of the Zariski closure of this union is contained in a finite union of Shimura subvarieties strictly contained in $\shS_{\IC}$. 
\end{theorem}

\begin{proof}
In our setting,
$$\che{X}=\{\w\in\IP(L):\<\w,\w\>=0\}$$
is just the space of isotropic lines in $L$. After replacing the base by $\sO_F[1/N]$ for a suitable number field $F$ and possibly enlarging $N$, for every $2\leq s\leq\sigma_0$ we may choose a totally isotropic direct summand $M_s\subset L\tensor\sO_F[1/N]$ of rank $s$. Set $Z_s:=\IP(M_s)\subseteq\che{X}$. We note that $\dim Z_s=s-1=\dim\shS^{\mathrm{ss},s}_{\IF_p}$ for every $p$ (cf. \cite[Thm.~4.6]{Ogu79}).

Increase $N$ whenever necessary. For each $2\leq s\leq\sigma_0$, let $d_s= \max\{ 1, \dim_\IZ T+s-\dim_\IZ\shS \}$, and let $E_s$ be the scheme obtained by applying \cref{atyplocusconstrprop} to $Z_s$ with $d=d_s$. Choose $r_0$ as in \cref{atyplocusconstrprop} to be big enough for every $s$.\footnote{The reader should not confuse the subscript $s$ where with the subscript $r_0$ in \cref{atyplocusconstrprop}. They have different meanings. We apologize for this unfortunate overlap of notation.} Then set $E:=\bigcup_{s=2}^{\sigma_0}E_s$. It suffices to show that, for every $p > r_0$, $[T_{\IF_p} \times_\shS \shS^{\mathrm{ss}, \sigma_0}_{\IF_p}]^{\mathrm{atyp}, >0}$ is contained in $E$. 

Let $C$ be a component of $[T_{\IF_p}\times_\shS\shS^{\mathrm{ss},\sigma_0}_{\IF_p}]^{\mathrm{atyp},>0}$. Set $k=\bar{\IF}_p$, let $x\in C$ be a general point, and write $s=\sigma_0(x)$. Since $C$ is positive dimensional, $s\geq2$. Moreover, by \eqref{eqn: dim of C} and $\dim\shS^{\mathrm{ss},\sigma_0}_{\IF_p}=\sigma_0-1$, we have $\dim C\geq d_{\sigma_0}\geq d_s$, so we may choose a jet $j\in J^{d_s}_rC(k)$ centered at $x$. Consider the subspace $M^i_x:=\mathrm{im}(T_i\tensor_{\IZ_p}k\to\bL_{\dR,x})$. Using parallel transport, we identify $j^*(\bL_{\dR},\nabla)$ with $(\bL_{\dR,x}\tensor A^{d_s}_{r,k},1\tensor d)$. The entire space of special endomorphisms $\mathrm{LEnd}(\sA_x)$ deforms along $j$. Because $\mathrm{LEnd}(\sA_x)\tensor\IZ_p=\bL_{\crys,x}^{\varphi=p}$, and $j^*\Fil^2$ must remain orthogonal to its de Rham realizations, $j^*\Fil^2$ is forced to remain orthogonal to $(M^0_x\oplus M^1_x)\tensor A^{d_s}_{r,k}$. As $M^1_x$ is self-dual, $j^*\Fil^2$ must lie in the orthogonal complement $(M^1_x\tensor A^{d_s}_{r,k})^\perp$ within $\bL_{\dR,x}\tensor A^{d_s}_{r,k}$. Moreover, because $M^0_x$ is the maximal isotropic subspace of $(M^1_x)^\perp$, $M^0_x\tensor A^{d_s}_{r,k}$ is its own orthogonal complement inside $(M^1_x\tensor A^{d_s}_{r,k})^\perp$. Therefore, $j^*\Fil^2$ must lie entirely within $M^0_x\tensor A^{d_s}_{r,k}$.

Finally, since $\mathrm{O}(L_k)$ acts transitively on the space of totally isotropic subspaces of $L_k$ of dimension $s$, we may find some isometry $b^\bullet:L_k\simeq\bL_{\dR,x}$ which sends $M_s$ to $M^0_x$. Recall that there is a period map $\sI^{d_s}_r\to J^{d_s}_r\che{X}$. We have just argued that the point $(j,b^\bullet)\in\sI^{d_s}_r$ is sent into $J^{d_s}_rZ_s\subseteq J^{d_s}_r\che{X}$. This implies that $x\in E_s \subseteq E$, as desired.
\end{proof}





\subsection{An interlude on Artin invariants}

Let us first generalize the notion of Artin invariants.

\begin{definition}
    Let $(M, \varphi)$ be a finite free F-crystal over an algebraically closed field $k$ of characteristic $p > 0$, i.e., $M$ is a finite free $W(k)$-module with a $\sigma$-linear endomorphism. If $(M, \varphi)$ is isoclinic of some slope $r$, then we define the Artin-invariant $\sigma_0$ of $(M, \varphi)$ to be $\dim_k \ker(T \tensor_{\IZ_p} k \to M_k)$ for the Tate module $T = M^{\varphi = p^r}$.  
\end{definition}

\begin{remark}
\label{rmk: unit root = slope 0}
    Recall that we say that a $(M, \varphi)$ as above is a unit root F-crystal when it satisfies one of the following equivalent conditions: (a) $\varphi : \sigma^* M \to M$ is an isomorphism. (b) $T \tensor_{\IZ_p} W \to M$ is an isomorphism. (c) $\varphi$ is purely of slope $0$. The implication (a) $\Rightarrow$ (b) is a special case of \cref{cor: Katz equiv free} below. That (b) $\Rightarrow$ (c) is clear. To see (c) $\Rightarrow$ (a), choose a basis of $M$, so that $\varphi = B \sigma$ for a square matrix $B$. The slope $0$ condition implies that $\mathrm{val}_p(\det(B)) = 0$, which implies that $\det(B) \in W^\times$, or equivalently that $\varphi$ is an isomorphism.
\end{remark}

By \cite[Thm.~3.5.1]{Ogu79}, when $(M, \varphi)$ is given by a supersingular K3 crystal, our definition of the Artin invariant $\sigma_0$ agrees with Ogus' (or Artin's) definition. 

We recall some basic facts about isoclinic F-crystals, due to Katz.

\begin{theorem}[Katz]
\label{thm: Katz equiv torsion}
    Let $k$ be a perfect field of characteristic $p$. 
    Let $S_n$ be a flat affine $W_n$-scheme whose special fiber is normal, reduced and irreducible, which admits a lift of Frobenius $\varphi_S$. Then there is an equivalence of categories
$$ \beta_n : \mathsf{Loc}_{\IZ / p^n \IZ}(S_n) \stackrel{\sim}{\to} \mathsf{Vect}^{\varphi}(S_n) $$
where $\mathsf{Loc}_{\IZ / p^n \IZ}(S_n)$ is the category of \'etale local systems with $\IZ / p^n \IZ$ coefficients and $\mathsf{Vect}^{\varphi}(S_n)$ is that of vector bundles $M$ over $S_n$ equipped with an isomorphism $\varphi : \varphi_{S_n}^* M \stackrel{\sim}{\to} M$.
\end{theorem}

\begin{proof}
    See \cite[Prop.~4.1.1]{Kat73}, which we apply with $q = p$. We summarize the main point of the proof for later reference. Given an $\IL \in \mathsf{Loc}_{\IZ / p^n \IZ}(S_n)$, choose a finite Galois cover $T_n/S_n$ that trivializes $\IL$. Because $T_n$ is \'etale, $\varphi_{S_n}$ lifts uniquely to an endomorphism $\varphi_{T_n}$. The trivialized object 
    \begin{equation}
    \label{eqn: define beta}
        (\IL(T_n) \tensor_{\IZ/p^n \IZ} \sO_{T_n}, m \tensor f \mapsto m \tensor \varphi_{T_n}(f)) \in \mathsf{Vect}^\varphi(T_n)
    \end{equation}
    is naturally $\Gal(T_n/S_n)$-equivariant and therefore descends to an object in $\mathsf{Vect}^\varphi(S_n)$, which is the desired $\beta(\IL)$. Conversely, given an $M \in \mathsf{Vect}^{\varphi}(S_n)$, Katz showed that there exists a finite Galois cover $T_n/S_n$ over which $M|_{T_n}$ admits a basis of Frobenius invariants. This basis defines an object in $\mathsf{Loc}_{\IZ / p^n \IZ}(S_n)$, yielding the desired quasi-inverse $\beta_n^{-1}(M)$.
\end{proof} 

The above has the following well known corollary (cf. \cite[p.~148 Remarks]{Kat79}). Part (b) is not explicitly stated, but it is clear from the constructions. 
\begin{corollary}[Katz]
\label{cor: Katz equiv free}
    Let $S$ be a smooth affine variety over $k$. 
    \begin{enumerate}[label=\upshape{(\alph*)}]
        \item There is an equivalence $\beta$ from $\mathsf{Loc}_{\IZ_p}(S)$ to the category of unit root F-crystals on $S$.
        \item Suppose that $\IL \in \mathsf{Loc}_{\IZ_p}(S)$ and $\sE = \beta(\IL)$. Given an algebro-geometric point $s \to S$, there is a natural identification $\IL_s = [\sE_s (W(s))]^{\varphi = 1}$. 
    \end{enumerate} 
\end{corollary}
\begin{proof}
    To construct $\beta$, one just put together the equivalences $\beta_n$ for each $S_n := \what{S}/p^n$ and $\IL_n := \IL/p^n$. Note that for each $\beta_n(\IL_n)$ over $S_n$, one can put a natural connection by descending the trivial connection $1 \tensor d$ on (\ref{eqn: define beta}). Part (b) can be easily deduced from the construction (\ref{eqn: define beta}). 
\end{proof}

Let us recall two elementary algebra facts. 
\begin{lemma}
    Let $f:R\to R'$ be a homomorphism of rings. Assume that $R$ is $p$-adically separated, $R'$ is $p$-torsion-free, and the induced map $R/pR\longrightarrow R'/pR'$ is injective. Then $f$ is injective.
\end{lemma}

\begin{proof}
    If $a\in\ker(f)$, then its image in $R'/pR'$ is zero. By injectivity of $R/pR\longrightarrow R'/pR'$, we have $a\in pR$, so $a=pa_1$ for some $a_1\in R$. Since $R'$ is $p$-torsion-free, $f(a_1)=0$. Iterating, we obtain $a\in p^nR$ for every $n$. Since $R$ is $p$-adically separated, $a=0$.
\end{proof}

\begin{lemma}
    Let $R\subseteq R'$ be an inclusion of $p$-torsion-free rings. If $R/pR \to R'/pR'$ is injective, then $R$ is $p$-saturated in $R'$, or equivalently $R=R'\cap R[1/p]$ inside $R'[1/p]$.
\end{lemma}

\begin{proof}
    Suppose $a\in R'$ and $pa\in R$. The image of $pa$ in $R'/pR'$ is zero, so injectivity modulo $p$ gives $pa\in pR$. Write $pa=pb$ with $b\in R$. Since $R'$ is $p$-torsion-free, $a=b\in R$. Thus $R$ is $p$-saturated. 
\end{proof}

\begin{proposition}
\label{prop: semi-continuity}
    Let $S$ be a smooth affine variety over an algebraically closed field $k$ of characteristic $p > 0$ and $\sE$ be an F-crystal over $S$ isoclinic of slope $r \in \IZ$. Then up to replacing $S$ by a nonempty Zariski open subvariety and then by a finite \'etale cover, there exists a sub-F-crystal $\sE' \subseteq \sE$ such that $\sE'(r)$ is unit-root and for a general point $s \in S$, $\sE_s(W)^{\varphi = p^r} = \sE'_s(W)^{\varphi = p^r}$. 
\end{proposition}
\begin{proof}
By \cite[Thm.~2.6.1]{Kat79} and the remark thereafter, up to shrinking $S$ to a Zariski open subvariety, we can always find a sub-F-crystal $\sE'' \subseteq \sE$ which is isogenous to $\sE$ such that the Tate twist $\sE''(r)$ is a unit root F-crystal. Let $\IL'' := \beta^{-1}(\sE''(r)) \in \mathsf{Loc}_{\IZ_p}(S)$. Let $\bar{\eta}$ be an algebro-geometric generic point on $S$ and set $T := \sE_{\bar{\eta}}(W(\bar{\eta}))^{\varphi = p^r}$. Note that the stalk $\IL''_{\bar{\eta}}$ is a $\IZ_p$-sublattice in $T$. Up to replacing $S$ by a finite \'etale cover, the action of $\pi_1^\et(S, \bar{\eta})$ on $T[1/p] = \IL''_{\bar{\eta}}[1/p]$ stabilizes the lattice $T$. Therefore, $T$ spreads out to some $\IL' \in \mathsf{Loc}_{\IZ_p}(S)$ such that $\IL'_{\bar{\eta}} = T$; moreover, up to taking a further cover, we may assume that $\IL' / p \IL'$ is a constant local system.

By replacing $S$ by a Zariski open subvariety again, we may assume that $S$ admits an \'etale morphism to $\IA^n_k = \Spec(k[x_1, \cdots, x_n])$. Set $\sE' := \beta(\IL')(-r)$. Note that by construction we have an identification $\sE'[1/p] = \sE[1/p]$. We argue that under this identification $\sE'$ embeds into $\sE$. 

First, let $R$ be a $p$-complete smooth $W$-algebra lifting $\sO_S$. Then there exists a lift of Frobenius $\varphi_R$ on $R$ and an equivariant embedding $R \into W(\bar{\eta})$. To see this, we may lift the \'etale map $k[x_1, \cdots, x_n] \to \sO_S$ to an \'etale map $W \< x_1, \cdots, x_n \> \to R$. Then we reduce to producing such $\varphi_R$ when $R = W \< x_1, \cdots, x_n \>$. In this case, let $\varphi_R$ just be defined by applying $\sigma$ to $W$ and sending $x_i$ to $x_i^p$. Setting $k(\bar{\eta}) = \overline{k(t_1, \cdots, t_n)}$, the natural map $R \to k(\bar{\eta})$ sending each $x_i$ to $t_i$ easily lifts to $R \to W(\bar{\eta})$ by sending $x_i$ to the Teichm\"uller lift $[t_i]$. Then $\varphi_R$ is compatible with $\sigma_{W(\bar{\eta})}$. 

By the two preceding lemmas, we have $R = R[1/p] \cap W(\bar{\eta})$ inside $W(\bar{\eta})[1/p]$. Set $M' = \sE'(R)$ and $M = \sE(R)$. Our goal is to show that under the identification of $R[1/p]$-modules $M'[1/p] = M[1/p]$, we have $M' \subseteq M$. Base changing to $W(\bar{\eta})$, we clearly have $M' \tensor_R W(\bar{\eta}) \subseteq M \tensor_R W(\bar{\eta})$, because we have $\sE'(W(\bar{\eta})) = T \tensor W(\bar{\eta}) \subseteq \sE(W(\bar{\eta}))$ by construction. On the other hand, since $M$ is a finite projective $R$-module, we have $M = (M \tensor_R W(\bar{\eta})) \cap M[1/p]$ and the same is true for $M'$. Taking intersections yields $M' \subseteq M$, which implies that $\sE' \subseteq \sE$, as desired. Note that we have not yet used the Frobenius equivariance of $R \into W(\bar{\eta})$ for this step. 

It remains to show that up to shrinking $S$ to some Zariski open $U$, $\sE_s(W)^{\varphi = p^r} = \sE_s'(W)^{\varphi = p^r}$ for every $s \in U$. This property is known at the generic point $\bar{\eta}$, so we essentially need to spread it out. Recall that we assumed that $\IL'/p\IL'$ is constant, given by a finite dimensional vector space $T / p T$ over $\IF_p$. Each section of this finite constant local system maps naturally to a global section of the vector bundle $\sE(S)$. At $\bar{\eta}$, the resulting map $T/pT \to \sE_{\bar{\eta}}(k(\bar{\eta}))$ is injective: if $t \in T$ is divisible by $p$ in $\sE_{\bar{\eta}}(W(\bar{\eta}))$, then $t/p$ is again a $\varphi=p^r$ element and hence belongs to $T$. Thus every nonzero section has nonzero image at $\bar{\eta}$. Removing the zero loci of these finitely many sections, we obtain an open dense subset $U \subseteq S$ on which $\IL'_s/p\IL'_s \to \sE_s(k)$ is injective for every closed point $s$.

Fix such an $s$ and an element $x \in \sE_s(W)^{\varphi=p^r}$. Since $\sE'$ is isogenous to $\sE$, there is a minimal integer $a \ge 0$ such that $p^a x \in \sE'_s(W)$. If $a>0$, then by the minimality of $a$, $p^a x \in \sE'_s(W)^{\varphi=p^r}=\IL'_s$ defines a nonzero class in $\IL'_s/p\IL'_s$, but its image in $\sE_s(k)$ is zero. This contradicts the injectivity just established. Hence $a=0$, and we conclude that  $\sE_s(W)^{\varphi=p^r}=\sE'_s(W)^{\varphi=p^r}$.
\end{proof}

The above proposition and its proof has the following two consequences. The first implies that it makes sense to talk about the ``generic Artin invariant'' for an isoclinic F-crystal. 

\begin{corollary}
     In the setting of \cref{prop: semi-continuity}, for a general point $s \in S$ and an algebro-geometric generic point $\bar{\eta}$, we have $\sigma_0(\sE_s) = \sigma_0(\sE_{\bar{\eta}})$. 
\end{corollary}
\begin{proof}
    Up to shrinking $S$ and replacing it by a finite \'etale cover, assume that a $\sE' \subseteq \sE$ as in \cref{prop: semi-continuity} exists. Up to further shrinking $S$, the function $s \mapsto \dim \ker(\sE'_s \tensor_W k \to \sE_s \tensor_W k)$ is constant. On the other hand, for any algebro-geometric point $t \to S$, we have $\sE_t^{\varphi = p^r} = (\sE'_t)^{\varphi = p^r}$ and $(\sE')_t^{\varphi = p^r} \tensor_{\IZ_p} k = \sE'_t \tensor_W k$. The conclusion easily follows. 
\end{proof}

\begin{corollary}
\label{cor: Frob inv horizontal}
    In the setting of \cref{prop: semi-continuity}, for a general point $s \in S$, every element $\alpha \in \mathrm{im}( \sE_s(W)^{\varphi = p^r} \tensor_{\IZ_p} k \to \sE_s(k))$ deforms to a horizontal section $\wt{\alpha}$ over the formal completion $S^\wedge_s$; moreover, if $r \ge 1$, then $\wt{\alpha}$ is killed by $\varphi$. 
\end{corollary}
\begin{proof}
    Since the conclusion concerns the formal completion at a general point, we are allowed to replace $S$ by a Zariski open subvariety and pass to an \'etale cover, so that a $\sE' \subseteq \sE$ as in \cref{prop: semi-continuity} exists. Let $\IL' = \beta_1^{-1}\sE'(r)$. Then the restriction of $\IL'$ over $S^\wedge_s$ is trivial. 
    
    By linearity, it suffices to consider the case when $\alpha = \alpha_0 \mod p$ for some $\alpha_0 \in \sE_s(W)^{\varphi = p^r} = \IL'_s$. Then $\alpha_0$ extends to some $\wt{\alpha}_0 \in \IL'|_{S^\wedge_s}$, and we may take $\wt{\alpha}$ to be the image of $\wt{\alpha}_0$ in $\sE(S^\wedge_s)$. The second statement is clear.  
\end{proof}

\subsection{Supersingular surfaces in $\IP^3$}

\begin{proposition}
    If $(H, \< -, - \>, \varphi)$ is a supersingular orthogonal F-crystal of level $2$ such that the cokernel of $H^{\varphi = p} \tensor_{\IZ_p} W \to H$ is $p$-torsion, then $\sigma_0 \le \mathrm{rank\,} H / 2$. 
\end{proposition}
    \begin{proof} The argument of \cite[Thm.~3.5.1]{Ogu79} for the case of K3 crystals ($h^0=1$) applies verbatim under the $p$-torsion assumption. We reproduce the argument for readers' convenience. 
    Following Ogus \cite[Def.~3.7]{Ogu79}, we define $E_H := \{ x \in H : \forall n \in \IN, \varphi^n(x) \in p^n H \}$ and $L_H := \sum_n p^{-n} \varphi^n (H)$. Define $\varphi_0 = p^{-1} \varphi$. Clearly, $E_H$ is the largest $\varphi_0$-stable submodule of $H$, and $L_H$ is the smallest $\varphi_0$-stable submodule of $H[1/p]$ which contains $H$. By \cref{rmk: unit root = slope 0}, $(E_H, \varphi_0)$ and $(L_H, \varphi_0)$ are both unit root F-crystals. 

    First, we argue that $E_H = T \tensor_{\IZ_p} W$ for $T := H^{\varphi = p}$. Indeed, we have $E_H = (E_H)^{\varphi_0 = 1} \tensor_{\IZ_p} W$ (by \cite[Rem. 5.5]{zbMATH03409516}, or \autoref{cor: Katz equiv free}) and $(E_H)^{\varphi_0 = 1} \subseteq T$. Therefore, $E_H \subseteq T \tensor_{\IZ_p} W$. Conversely, $T \tensor_{\IZ_p} W$ is clearly $\varphi_0$-stable sublattice of $H$, so we obtain the reverse inclusion. Next, we argue that $L_H = E_H^\vee$ (cf. \cite[Lem.~3.10]{Ogu79}). We have $\varphi_0 (E_H) = E_H$, so that $\varphi_0(E_H^\vee) = E_H^\vee$. By the minimality of $L_H$, we have $L_H \subseteq E_H^\vee$. For the reverse inclusion, since $H \subseteq L_H$ and $H$ is self-dual, we have $L_H^\vee \subseteq H$. Moreover, we also have $\varphi_0(L_H^\vee)=L_H^\vee$. By the maximality of $E_H$, it follows that $L_H^\vee \subseteq E_H$. Dualizing again gives $E_H^\vee \subseteq L_H$. 

    We now use the $p$-torsion assumption. Since $E_H=T \tensor_{\IZ_p} W$, the $p$-torsion assumption implies $p H \subseteq E_H$, and dually $p L_H \subseteq H$. Since $p L_H$ is $\varphi_0$-stable, the maximality of $E_H$ gives $p L_H \subseteq E_H$. To sum up, we have now $$ pH\subseteq pL_H\subseteq E_H\subseteq H\subseteq L_H
    \subseteq p^{-1}H. $$ Recall that by definition, $$\sigma_0 = \dim_k\ker(E_H/pE_H\to H/pH) =\dim_k(pH/pE_H) = \dim_k(H/E_H), $$
    where the second equality uses $pH\subseteq E_H$. 
    
    As in \cite[Lem.~3.10]{Ogu79}, dualizing $0\to E_H\to H\to H/E_H\to0$ gives a natural duality between
$H/E_H$ and $L_H/H$. Hence $\dim_k(L_H/H)=\dim_k(H/E_H)=\sigma_0$. 
It follows that $$\mathrm{length}_W(L_H/E_H) = \mathrm{length}_W(H/E_H)+ \mathrm{length}_W(L_H/H) =2\sigma_0.$$ On the other hand, $pL_H\subseteq E_H$, so $L_H/E_H$ is a quotient of $L_H/pL_H$. Consequently
$$ 2\sigma_0
    =\dim_k(L_H/E_H)
    \leq \dim_k(L_H/pL_H)
    =\mathrm{rank\,} H.$$
This proves the proposition.
\end{proof}

Let us define some polynomials in $d$ as follows. 

$$ h = { d - 1 \choose 3 } \quad n = d^3 - 4d^2 + 6d - 3 \quad n' = {d + 3 \choose 3} - 16. $$

\noindent The point is that for any smooth $X \subseteq \IP^3_\IC$ of degree $d$, $h = h^{2, 0}(X)$, and $n = b_2(X) - 1$; $n' = \dim_\IZ |\sO_{\IP^3}(d)|/ \PGL_4$. We additionally set $$ \sigma_0(d) = n - \lceil \frac{n'}{h} \rceil. $$



In what follows we fix a quadratic lattice $V$ which is isomorphic to the primitive part of $H^2(X, \mathbb{Z})$ for any smooth degree $d$ surface $X \subset \mathbb{P}^3_{\IC}$. We write $\che{D} = \mathrm{IGr}(h, V)$ for the flag variety parameterizing $h$ dimensional isotropic subspaces in $V$. Define $m = \lceil \dim \sfM_\IC / h \rceil$. We can find a number field $L$ such that for each $0 \le j \le m$, we can find a direct summand $M_j \subseteq V_L$ of dimension $m$ such that $\dim M_j \cap M_j^\perp = j$. For some large enough $N$, we can spread out $M_j$ and $M_j^\perp$ to direct summands of $V \tensor \sO_L[1/N]$. We then set $Z_{j}$ to be the scheme $\mathrm{IGr}(M_j^\perp, h)$, which is naturally contained in $\che{D}_{\sO_L[1/N]}$. \cref{supsingquintic} follows directly from the preceding proposition and the theorem below. 


\begin{theorem}
\label{thm: supersingular surface general}
     Let $\sfM$ be the open substack of $|\sO_{\IP^3}(d)| / \mathrm{PGL}_4$ parametrizing those smooth degree $ d \ge 5$ surfaces, and let $E = \bigcup_{i = 0}^m E_m \subset \sfM$ be the union of the atypical loci associated to the $Z_{i}$, defined as in \autoref{Zatyplocusdef}. Then on $\sfM \smallsetminus E$ there are only finitely many primes $p$ for which there exists one-parameter families of supersingular surfaces with generic Artin invariant $\le \sigma_0(d)$. 
\end{theorem}

\begin{proof}
    Let $\sV$ be the primitive part of the second de Rham cohomology of the universal family $f : \sX \to \sfM$.  Apply the construction in \cref{constr: atypical E} to $\sfM$ over $\sO_L[1/N]$ for $N$ suitably large, with $d = 1$. In particular, we obtain an orthogonal frame bundle $\sM$ such that for every algebraically closed field $k$, $\sM(k)$ is the set of pairs $\{ (s, b^\bullet)\}$ where $s \in \sfM(k)$ and $b^\bullet$ is an isometry $V_k \simeq \sV_s$. We have a period map $\sM^1_r \to J^1_r \che{D}$ over $\sO_L[1/(Nr!)]$. By construction, every point $(s, b^\bullet, j) \in \sM^1_r(k)$ such that the composition $\ID^1_{r, k} \to \che{D}$ factors through $\union_{j = 0}^m Z_j$ projects into $E$ (i.e., $s \in E(k)$). 

    Note that $2h \le n - 2m \le n - m - j$, so by \cref{lem: codim IGr orthogonal} below, the codimension of $\IGr(M_{j, k}^\perp, h)$ in $\IGr(V_k, h)$ is $hm \ge \dim \sfM_\IC$. In particular $d = 1$ satisfies the inequality of \cref{atyplocusconstrprop}. The universal family over $\sfM_{\IC}$ has $\mathbb{Q}$-simple monodromy by a result of Beauville \cite{zbMATH03973031}, and the period map is quasi-finite as a consequence of the local Torelli theorem of \cite[Thm. 5.4]{zbMATH03458789}. It then follows from \cref{atyplocusconstrprop} that $E$ is not Zariski dense, and $E_{\IC}$ is a finite union of special subvarieties of $\sfM_{\IC}$.

    Suppose now that for some $p$, there is a curve $C$ in the supersingular locus of $\sfM_{k}$. Let $s$ be a general point on $C$. Set  
    $$ I_s := \mathrm{im}(\mathrm{PH}^2_\crys(\sX_s / W)^{\varphi = p} \tensor_{\IZ_p} k \to \mathrm{PH}^2_\dR(\sX_s/k) = \sV_s). $$
    Since $\dim I_s \ge m$, choose an $m$-dimensional subspace $I'_s \subseteq I_s$, set $j = \dim(I'_s \cap (I'_s)^\perp)$, and choose an isometry $b^\bullet : V_k \stackrel{\sim}{\to} \sV_s$ such that $M_j$ maps isomorphically onto $I'_s$.

    By \cref{cor: Frob inv horizontal}, every element $\alpha \in I_s$ deforms to a horizontal section $\wt{\alpha}$ over the formal completion $C_s^\wedge$. Note that by Mazur-Ogus inequality \cite[Thm.~8.26]{BOBook}, as $\wt{\alpha}$ is killed by $\varphi$, it must lie in the $\Fil^1$ part of the Hodge filtration. This implies that $s \in E(k)$, as desired. 
\end{proof}

\begin{lemma}
\label{lem: codim IGr orthogonal}
    Let $V$ be a nondegenerate quadratic space of dimension $n$ over an algebraically closed field $k$. Let $M \subseteq V$ be a subspace of dimension $m$ and let $j$ be the dimension of $M \cap M^\perp$. If $2h \le n - m - j$, then the codimension of $\IGr(M^\perp, h)$ in $\IGr(V, h)$ is $hm$. 
\end{lemma}

\begin{proof}
    The dimension of the ambient isotropic Grassmannian is given by the standard formula:
    \begin{equation}
    \label{eqn: dim IGr orthogonal}
        \dim \IGr(V, h) = hn - \frac{h(3h+1)}{2}.
    \end{equation}
    The orthogonal complement $M^\perp$ has dimension $n - m$. The radical of the restriction of the quadratic form to $M^\perp$ is exactly $R = M \cap M^\perp$, which by assumption has dimension $j$. We can project $M^\perp$ onto its non-degenerate quotient space $\overline{M^\perp} = M^\perp / R$, which has dimension $n - m - j$. 
    
    The hypothesis $2h \le n - m - j$ ensures that a generic totally isotropic $h$-dimensional subspace $W \subseteq M^\perp$ intersects the radical $R$ trivially, i.e., $W \cap R = \{0\}$. Such $W$'s form a dense open subvariety $U \subseteq \IGr(M^\perp, h)$. To see the density, consider the stratification of the complement $\IGr(M^\perp, h) \smallsetminus U$ given by $r=\dim(W\cap R)\geq 1$. The $r$-th stratum is parametrized by the choice of $A\in \mathrm{Gr}(r,R)$, an isotropic $(h-r)$-plane in $\overline{M^\perp}$, and a lift in $\Hom(\overline W,R/A)$. The standard dimension formula shows that each such stratum has strictly smaller dimension than $U$ under the assumption $2h\leq n-m-j$. We compute the dimension of $U$ (and thus of $\IGr(M^\perp, h)$) by considering the projection $\pi: M^\perp \to \overline{M^\perp}$. The restriction of $\pi$ to any $W \in U$ is injective, and sends $W$ to a totally isotropic $h$-plane $\overline{W} \subseteq \overline{M^\perp}$. 
    
    This gives a projection from $U$ to an open dense subset of $\IGr(\overline{M^\perp}, h)$. For a fixed $\overline{W} \in \IGr(\overline{M^\perp}, h)$, the fiber consists of all totally isotropic lifts $W \subseteq M^\perp$ such that $\pi(W) = \overline{W}$. Since $R$ is the radical of the quadratic form on $M^\perp$, the preimage $\pi^{-1}(\overline{W})$ is a totally isotropic subspace of dimension $h + j$. A choice of a lift $W$ of $\overline{M}$ corresponds to a splitting of the exact sequence
    $ 0 \to R \to \pi^{-1}(\overline{W}) \to \overline{W} \to 0$, which is parametrized by the affine space $\mathrm{Hom}(\overline{W}, R)$. 
    Finally, the codimension we are seeking is computed by 
    $\dim \IGr(V,h) - \dim \IGr(\overline{M^\perp}, h) - \dim \mathrm{Hom}(\overline{W}, R)$, which one easily finds to be $hm$ as (\ref{eqn: dim IGr orthogonal}) also applies to $\overline{M^\perp}$. 
\end{proof}

\begin{remark}
\label{rmk: singular quintic}
    In \cref{thm: supersingular surface general}, we can enlarge $\sfM$ to the open substack of $|\sO_{\IP^3}(d)| / \PGL_4$ parametrizing those surfaces with at worst du Val (RDP) singularities. For why the substack is open, the reader may look at \cite[Prop.~2.7]{GY25}. We give a sketch on how it works: By Artin-Brieskorn resolution \cite{Art74}, we can find a cover $\sfM'$ of $\sfM$ over which the universal hypersurface $\sX$ admits a simultaneous resolution. Then by resolution of singularities and spreading out, one replaces $\sfM'$ by a smooth scheme, after inverting some $N$. We apply the same analysis to $\sfM'$, and obtain an $E$ on $\sfM'$ so that the conclusion of \cref{thm: supersingular surface general} holds for $\sfM' \smallsetminus E$. 

    The caveat is that the discriminant locus of $\sfM$ parametrizing the singular fibers is a divisor, and remains a divisor in $\sfM'$. Let us call this divisor $\sfD$. As resolving RDP singularities incurs exceptional divisors on the resolved surfaces, $\sfD$ is an atypical component of the Noether-Lefschetz locus. There is a possibility that $\sfD$ is contained in $E$, so that there is no added generality. However, whether $\sfD$ is contained in $E$ can in principle be decided effectively (cf. \cref{rmk: effectivity}). 
\end{remark}

\subsection{Supersingular Curves}

In previous sections, we studied the intersection of a closed subscheme of $\sfA_g$ with the locus of abelian varieties which admit a prime-to-$p$ splitting. Here we prove some results when the ``prime-to-$p$'' requirement is dropped. Then we apply these results to the study of supersingular curves.

\begin{proposition}
\label{prop: alternating}
    Let $k$ be a field. 
    \begin{enumerate}[label=\upshape{(\alph*)}]
        \item An alternating form $(V, \w)$ over $k$ is classified up to isometry by $\dim V$ and $\dim \mathrm{Rad}(\w)$. Furthermore, the induced alternating form on the quotient $V / \mathrm{Rad}(\w)$ is always symplectic (i.e., non-degenerate).  
        \item Let $(V, \w)$ be a symplectic space and let $M_1, M_2 \subseteq V$ be subspaces. There exists $\gamma \in \mathrm{Sp}(V, \w)$ such that $\gamma(M_1) = M_2$ if and only if $\dim M_1 = \dim M_2$ and $\dim \mathrm{Rad}(\w|_{M_1}) = \dim \mathrm{Rad}(\w|_{M_2})$. 
    \end{enumerate}
\end{proposition}
In the above, $\mathrm{Rad}(\w)$ denotes the radical, i.e., the kernel of $V \to V^\vee$ induced by $\w$. We will keep using the notation below. 
\begin{proof}
    Part (a) is standard, and (b) follows from (a) by Witt's cancellation theorem for symplectic forms (see e.g., \cite{James}, which holds over any PID). 
\end{proof}

\begin{proposition}
\label{prop: isotropic in M}
    Let $(M, \w)$ be an alternating space of dimension $2m$ over a field $k$, and let $R = \mathrm{Rad}(\w)$ have dimension $2r$. The variety $X = \mathrm{IGr}(m, M)$ of $m$-dimensional isotropic subspaces admits a locally closed stratification
    $$ X = \bigsqcup_{j = r}^{\min(m, 2r)} X_j, $$
    where $X_j = \{ W \in X \mid \dim(W \cap R) = j \}$. Each stratum $X_j$ is an affine bundle of rank $(m-j)(2r-j)$ over $\mathrm{Gr}(j, R) \times \mathrm{IGr}(m-j, M/R)$, and has dimension
    $$ \dim X_j = \frac{1}{2}m(m+1) + 2jr - \frac{1}{2}j(3j+1). $$
    Moreover, $X$ is irreducible with the unique open dense stratum being $X_r$. In particular, 
    $$ \dim X = \dim X_r = \frac{1}{2}m(m+1) + \frac{1}{2}r(r-1). $$
\end{proposition} 

\begin{proof}
    Let $\bar{M} = M/R$ be the symplectic quotient, which has dimension $2m-2r$. The natural projection $\pi \colon M \to \bar{M}$ maps $W \in X_j$ to an isotropic subspace $\bar{W} \subset \bar{M}$ of dimension $m-j$. The isotropy of $\bar{W}$ forces $m-j \le (\dim \bar{M})/2 = m-r$, giving the lower bound $j \ge r$.
    
    A subspace $W \in X_j$ is uniquely determined by the choices of the intersection $K = W \cap R \in \mathrm{Gr}(j, R)$, and a linear complement to $R \cap W$ in $W$. It is easy to check that the natural map $X_j \to \mathrm{Gr}(j, R) \times \mathrm{IGr}(m-j, \bar{M})$ given by sending $W$ to $(K, \bar{W})$ is surjective. Fixing $K$ and $\bar{W}$, the space of subspaces $W \subset \pi^{-1}(\bar{W})$ satisfying $W \cap R = K$ and mapping surjectively onto $\bar{W}$ is a torsor under $\mathrm{Hom}(\bar{W}, R/K)$. This affine space has dimension $\dim(\bar{W}) \dim(R/K) = (m-j)(2r-j)$, yielding the desired fibration. 
\end{proof}

\begin{lemma}
\label{lem: Z_M}
    Let $(V, \w)$ be a symplectic space of dimension $2g$, and let $M \subseteq V$ be a subspace of dimension $2m \le g$. Set $2r = \dim \Rad(\w|_M)$. Let $Z_M \subseteq \IGr(g, V)$ be the subvariety defined by 
    $$ Z_M = \{ F \in \IGr(g, V) \mid \dim (F \cap M) = m \}. $$
    Then $Z_M$ is an irreducible variety in $\IGr(g, V)$ of codimension
    \begin{align*}
        s(g, m, r) &= \frac{1}{2}g(g + 1) - \big[ \frac{1}{2}m(m+1) + \frac{1}{2}r(r-1) + \frac{1}{2}(g-m)(g-m+1) \big] \\
        &= m(g-m)-\frac{r(r-1)}2. 
    \end{align*}
\end{lemma}

\begin{proof}
    Consider the projection $\phi \colon Z_M \to \IGr(m, M)$ given by $F \mapsto F \cap M$. For any $E \in \IGr(m, M)$, the fiber $\phi^{-1}(E)$ consists of Lagrangians $F \subset V$ containing $E$ such that $F \cap M = E$. Passing to the symplectic quotient $E^\perp/E$, these correspond exactly to Lagrangians in $E^\perp/E$ that trivially intersect the $m$-dimensional subspace $M/E$. Since the symplectic quotient has dimension $2(g-m)$ and $2m \le g$, a generic Lagrangian in $E^\perp/E$ intersects $M/E$ trivially. Thus, the fibers of $\phi$ are non-empty, Zariski open dense subvarieties of the Lagrangian Grassmannian $\IGr(g-m, E^\perp/E)$. The dimension formula now follows from \cref{prop: isotropic in M}. 
\end{proof}

\begin{proposition}
\label{prop: strict filtered}
    Let $S$ be any base scheme. If $0 \to B \to A \to Q \to 0$ is an exact sequence of abelian schemes over $S$, then the induced sequence $0 \to \H^1_\dR(Q/S) \to \H^1_\dR(A/S) \to \H^1_\dR(B/S) \to 0$ is also exact and strict (\cite[\href{https://stacks.math.columbia.edu/tag/0123}{Tag 0123}]{stacks-project}) as a sequence of filtered vector bundles. 
\end{proposition}
Note that this remains true even in the absence of a separable isogeny between $A$ and $B \times_S Q$.
\begin{proof}
    For any abelian scheme $\pi_X \colon X \to S$ with zero section $e_S$, the Hodge-de Rham spectral sequence degenerates at $E_1$. This induces a canonical short exact sequence of vector bundles
    $$ 0 \to \omega_{X/S} \to \H^1_\dR(X/S) \to R^1\pi_{X,*}\sO_X \to 0, $$
    which defines the Hodge filtration. Moreover, there are canonical identifications $\Lie(X/S) \simeq e_S^* \Omega_{X/S}^\vee$ and $\Lie(X^\vee /S) \simeq R^1\pi_{X,*}\sO_X$ (\cite[\S8.4~Thm.~1]{BLR}). The statement then follows from the fact that $X \mapsto \Lie(X/S)$ and $X \mapsto X^\vee$ are both exact functors on abelian schemes. 
\end{proof}

Let $A$ be an abelian scheme over $S$ and $\lambda$ be a polarization. We let $\w_\lambda$ denote the alternating pairing on $\H^1_\dR(A/S)$ induced by $\lambda$. 

\begin{definition}
    Let $(A, \lambda)$ be a principally polarized abelian variety over a field $k$ and $B \subseteq A$ be an abelian subvariety. We define the Artin invariant of the pair $(B, A, \lambda)$ to be 
    \begin{equation}
        \sigma_0(B, A, \lambda) = \frac{1}{2} \dim_k \Rad(\w_{\lambda}|_{\H^1_\dR(Q/k)}), \textit{ where } Q = A/B. 
    \end{equation}
\end{definition}

Note that the restriction of $\w_{\lambda}$ to the subspace $\H^1_\dR(Q/k) \subseteq \H^1_\dR(A/k)$ remains alternating but fails to be symplectic in general. Nonetheless, its radical is always even dimensional. 

\begin{notation}
    Given a scheme $S$, we write $S_{\pos}$ for the set of those points which are of positive transcendence degree over their prime field. 
\end{notation}

\begin{lemma}
\label{lem: separable}
    Let $A$ be an abelian variety over any field $k$. Every abelian subvariety $B$ of $A_{\bar{k}}$ descends to $A_{k^{\mathrm{sep}}}$. 
\end{lemma}
\begin{proof}
    Whether or not the inclusion $B \into A_{\bar{k}}$ admits a splitting, there is always some $\alpha \in \End(A_{\bar{k}})$ such that $B = \alpha(A_{\bar{k}})$. Then apply the well known fact that $\End(A_{\bar{k}}) = \End(A_{k^{\mathrm{sep}}})$. 
\end{proof}

For the rest of this subsection, let $(V,\w)$ be the standard symplectic space of dimension $2g$ over $\IZ$ and set $\che{D}=\IGr(g,V)$. For integers $m,\sigma_0$ with $2m \leq g$ and $0\leq \sigma_0\leq m$, fix a subspace $M_{m,\sigma_0}\subseteq V$ of dimension $2m$ such that $\dim\Rad(\w|_{M_{m,\sigma_0}})=2\sigma_0$, and set
$$Z_{m,\sigma_0}:=\{F\in\IGr(g,V)\mid \dim(F\cap M_{m,\sigma_0})=m\}.$$
By \cref{lem: Z_M}, $Z_{m,\sigma_0}$ is irreducible of codimension $s(g,m,\sigma_0)=m(g-m)-\sigma_0(\sigma_0-1)/2$.

\begin{theorem}
\label{thm: contains sub}
In $\sfA_g$ consider the locus
$$\sS(m,\sigma_0)=\{x\in\sfA_{g,\pos}:\sA_{\bar{x}}\text{ admits an abelian subvariety $B$ with }\dim B=g-m\text{ and }\sigma_0(B,\sA_{\bar{x}},\lambda_{\bar{x}})=\sigma_0\}.$$
Let $T\subseteq\sfA_g$ be a Hodge-generic locally closed subscheme whose generic fibre $T_{\IC}$ has dimension $\leq s(g,m,\sigma_0)$, and let $E_{m,\sigma_0}\subseteq T$ be the atypical locus of type with respect to $Z_{m, \sigma_0}$, defined as in \autoref{Zatyplocusdef}. Then
$(T\cap\sS(m,\sigma_0))_{\pos}\subseteq E_{m,\sigma_0}.$ If moreover the monodromy group of $T$ is $\IQ$-simple, then $(T\cap\sS(m,\sigma_0))_{\pos}$ is contained in finitely many strict Shimura subvarieties. 
\end{theorem}


Here $\sS(m,\sigma_0)$ is just understood as a set of points on the underlying topological space of $\sfA_g$, and the subscript $(-)_{\operatorname{pos}}$ is used to refer to those points of positive dimension. The non-density of $(T\cap\sS(m,\sigma_0))_{\pos}$ implies the non-density of positive-dimensional families of supersingular curves with the specified properties. 

\begin{proof}
    Each component of $(T \cap \sS(m, \sigma_0))_{\mathrm{pos}}$ either intersects the non-empty open locus where $T$ is smooth (over $\Spec\, \mathbb{Z}$), or lies in a strict closed locus. By replacing $T$ with an open subscheme, we may reduce to the case where $T$ is smooth.
    
    Let $\sP$ be the symplectic frame bundle over $\sfA_g$ whose points are given by $(x, b)$ where $x \in \sfA_g$ and $b$ is a symplectic identification $(V, \w) \tensor k(x) \simeq (\H^1_\dR(\sA_{x}), \w_{\lambda_x})$. Then for all $d$ and $r$, we have a morphism $\alpha^d_r : \sP^d_r \to J^d_r \che{D}$ over $\IZ[1/r!]$, constructed as in \autoref{refinedjetthm}.

    Suppose now that $x \in \sS(m, \sigma_0)$. Let $S$ be the Zariski closure of $x$. Set $k = \kappa(x)$. By \cref{lem: separable}, the embedding $B \subseteq \sA_{\bar{x}}$ is defined over $k^{\mathrm{sep}}$, so that we can find an \'etale morphism $S' \to S$ such that $B$ descends and spreads out to an abelian subscheme $\sB$ of $\sA_{S'}$. 


    
    Let $x_0$ be a general point on $S$, and $\wt{x}_0 \in S'$ be a lift. We now claim the existence of a non-degenerate jet $\ID^1_{r, k(x_0)} \to \sfA_g$ that lifts to $\sM^1_r \times_{J^1_r \che{D}} J^1_r Z_{m, \sigma_0}$. By \cref{prop: alternating}(b), for some choice of $b : (V, \w) \tensor k(\wt{x}_0) \simeq (\H^1_\dR(\sA_{\wt{x}_0}), \w_{\lambda_{\wt{x}_0}})$, $M$ is sent isometrically to $\H^1_\dR(\sQ_{\wt{x}_0}) \subseteq \H^1_\dR(\sA_{\wt{x}_0})$ for $\sQ := \sA_{S'} / \sB$. Then it follows from \cref{prop: strict filtered} that the induced local period morphism at $\wt{x}_0$ takes values in $Z_M \subseteq \che{D}$. Since the map $S' \to S$ is \'etale, this jet remains non-degenerate after projection to $S$. 

    Now apply \autoref{atyplocusconstrprop} as follows. The embedding $\iota : L \hookrightarrow \mathbb{C}$ is the obvious embedding $\mathbb{Q} \hookrightarrow \mathbb{C}$. The triple $(\sH, \Fil^{\bullet}, \nabla)$ is that associated to the de Rham cohomology of the universal abelian scheme over $\sfA_{g}$ in degree one. The Hodge datum $(G, \Omega) = (\GSp_{2g}, \mathbb{H}_{g})$. We take $Z = Z_{m, \sigma_0}$. Then $j = 1$ and $I_Z^c = \varnothing$, hence $D_{Z}$ is a point. In particular $n(I) = \dim T_{\IC}$, $\Omega' = \Omega$ and $\che{D}^{\mathrm{nc}}_Z = \che{D}$. The right-hand side of the inequality of \autoref{atyplocusconstrprop} then reads $\dim T_{\IC} + \dim Z_{m, \sigma_0,\IC} - \dim \sfA_{g,\IC}$; by assumption this is less than $d := 1$. Then setting $E = E^{d}$, we obtain a strict Zariski closed locus $E \subset T$. By construction the jets from the previous paragraph lie above $E$, so the result follows.
\end{proof}


\begin{corollary}
    \label{supnotdensecor}
    Let $\sS_g \subseteq \sfA_g$ be the locus of all supersingular abelian varieties (in all characteristics). Then for each $g \ge 16$, we have that away from finitely many primes, $\sS_{g, \pos} \cap \tau(\sM_g)$ is contained in finitely many strict Shimura subvarieties of $\sfA_{g}$. 
\end{corollary}

As before, $\sS_g$ and $\sS_{g, \pos}$ are regarded as sets.

\begin{proof}
Note that $s(g, m, r)$ achieves its minimum when $r = m$. For fixed $g$, $s(g, m, m)$ is a quadratic polynomial in $m$ and achieves its maximum at $m = \lfloor (g + 1)/3 \rfloor$. Taking this $m$, whenever $g \ge 16$ we have $3g - 3 \le s(g,m, m)$. 
\end{proof}


\begin{remark}
\label{rmk: density of ss in A_g}
    This fact is well known, but we give a quick argument, which we learned from Ananth Shankar a few years ago, to show that 
    \begin{equation}
    \label{eqn: density of supersingular}
        \bigcup_{p \in \Spec(\IZ)} \sfA_{g, \IF_p}^{\mathrm{ss}} \text{ is Zariski dense in $\sfA_g/\IZ$.}
    \end{equation}
    Let $E$ be a CM elliptic curve with CM action by $F$, defined over a number field $L$. Spread out $E$ to some $\sO_L[1/N]$. For every $p$ inert in $F$ and $v \in \Spec(\sO_L[1/N])$, the reduction $E_{k(v)}$ is supersingular. This of course also holds for all abelian varieties isogenous to $E^g$. By Chabotarev density, half of the rational primes are inert in $F$. So one easily obtains (\ref{eqn: density of supersingular}). In fact, a refinement of this argument shows that the union of superspecial abelian varieties, over all primes, is already dense in $\sfA_g / \IZ$. 
\end{remark}

\begin{example}
\label{ex: Artin-Schreier}
    If $\IF_q$ is a finite field of characteristic $p$ and $R(x) \in \IF_q[x]$ is an additive polynomial of degree $p^h$, then the curve 
    \begin{equation}
        \label{eqn: Artin-Scheier}
        Y: y^p - y = xR(x)
    \end{equation}
    is supersingular and has a genus $g = p^h(p-1)/2$ (\cite{vdGvdV92, Bla12, BHM+16}). Here $R[x]$ is said to be additive if $R(x_1 + x_2) = R(x_1) + R(x_2)$, or equivalently, $R(x)$ is a linear combination of monomials $x^d$, whether $d$ is a power of $p$. Clearly, such families are positive-dimensional. Our \autoref{supnotdensecor} says that, when $g \geq 16$, one cannot expect to find a Zariski dense collection of positive-dimensional families of supersingular curves in $\sfM_{g}$ if one considers all primes $p$ at once. This in some sense explains why $p$ must be a function of $g$ in this example. 
\end{example}

\begin{example}
\label{ex: Karemaker-Pries}
    Karemaker and Pries have a very interesting construction in \cite[Cor.~4.22]{Pri25}, which is closely related to the example above. They showed that for each fixed $p$, there are infinitely many $g$'s such that there exists a smooth supersingular curve of genus $g$ in characteristic $p$. 

    
\end{example}

\paragraph{Acknowledgments} The project began back in July 2024, when the second author visited the first author at IHES. We thank IHES for its hospitality. 

Special thanks go to Peihang Wu for helping us understand compactifications of Shimura varieties, and to Qiao He and Xu Shen for conversations on the basic locus. We also thank Ziyang Gao, Ruofan Jiang, Ananth Shankar, and Yunqing Tang for their interest in our work. The first author thanks Gregorio Baldi for helpful conversations. The second author would also like to thank his students Xiaoyuan Tan and Zuozheng Lyu for discussions on supersingular surfaces. We thank Rachel Pries for comments on her work. 

The second author is supported by a CUHK start up grant, and an ECS grant from the Hong Kong RGC (Project No. 24308225).

\paragraph{AI Usage} The authors used AI tools (ChatGPT 5.6 and Gemini Pro) for literature searches, proofreading, polishing, and some routine algebra computations. The only proof obtained with AI assistance was \cref{prop: semi-continuity}. The manuscript was written entirely by the authors, and they take full responsibility for its content.

\printbibliography

{\footnotesize 

\paragraph{David Urbanik} Department of Mathematics, Rice University, 6100 Main St MS 136, Houston, TX 77005  \, \, \, Email: \url{du8@rice.edu}

\paragraph{Ziquan Yang} The Institute of Mathematical Sciences and Department of Mathematics, The Chinese University of Hong Kong, Shatin, N.T., Hong Kong.  \, \, \,  Email: \url{zqyang@cuhk.edu.hk}

}

\end{document}